\documentclass[11pt,reqno,twoside]{amsart}
\usepackage{amscd}
\usepackage{amsfonts}
\usepackage{amsmath}
\usepackage{amssymb}
\usepackage{amsthm}
\usepackage{amsaddr}
\usepackage{fancyhdr}
\usepackage{latexsym}
\usepackage[colorlinks=true, pdfstartview=FitV, linkcolor=blue, citecolor=blue, urlcolor=blue]{hyperref}
\usepackage{enumitem}      
\usepackage{mathtools}            
\usepackage{indentfirst} 
\usepackage{color}
\usepackage{caption}          
\usepackage[normalem]{ulem}
\newcommand{\nn}{\nonumber}
\newcommand{\p}{\partial}
\newcommand{\ve}{\varepsilon}

\newcommand{\no}[1]{\left\| #1 \right\|}
\newcommand{\what}{\widehat}

\DeclareMathOperator{\sech}{sech}
\numberwithin{equation}{section}

\renewcommand{\Im}{\mathrm{Im}}
\usepackage{thmtools}
\declaretheoremstyle[headfont=\normalfont\bfseries, bodyfont=\itshape, spaceabove=7pt, spacebelow=7pt]{theorem} 
\theoremstyle{theorem} 
\newtheorem{theorem}{Theorem}[section] 
\newtheorem{remark}{Remark}[section]
\newtheorem{proposition}{Proposition}[section]
\newtheorem{lemma}{Lemma}[section]

\theoremstyle{definition}

\allowdisplaybreaks
\numberwithin{equation}{section}
\numberwithin{figure}{section}
\usepackage{geometry}
\let\OLDthebibliography\thebibliography
\renewcommand\thebibliography[1]{
  \OLDthebibliography{#1}
  \setlength{\parskip}{0pt}
  \setlength{\itemsep}{0pt plus 0.3ex}
}
\AtBeginDocument{
   \def\MR#1{}
}

\makeatletter
\@namedef{subjclassname@2020}{%
  \textup{2020} Mathematics Subject Classification}
\makeatother

\begin{document}

\title{
Discrete Nonlinear Schr\"odinger versus Ablowitz-Ladik:
Existence and dynamics of generalized focusing NLS-type lattices over a nonzero background
}

\author{Dirk Hennig$^{\MakeLowercase{a}}$, Nikos I. Karachalios$^{\MakeLowercase{a}}$, Dionyssios Mantzavinos$^{\MakeLowercase{b}}$,
\\
Dimitrios Mitsotakis$^{\MakeLowercase{c}}$ 
}
	\address{\normalfont $^a$Department of Mathematics,  University of Thessaly, 35100, Lamia, Greece\vspace*{-3.5mm}}
    \email{dirkhennig@uth.gr,  karan@uth.gr}
	\address{\normalfont $^b$Department of Mathematics, University of Kansas, Lawrence, KS 66045, USA\vspace*{-3.5mm}} \email{mantzavinos@ku.edu (corresponding author)}
	\address{\normalfont $^c$Victoria University of Wellington, School of Mathematics and Statistics, PO Box 600, \\ Wellington 6140, New Zealand
    \\[5mm]
    \begin{center}
\textit{Dedicated to the memory of Yannis Karachalios}
\end{center}} \email{dimitrios.mitsotakis@vuw.ac.nz}

\thanks{\textit{Acknowledgements.} The third author gratefully acknowledges support from the U.S. National Science Foundation (NSF-DMS 2206270 and NSF-DMS 2509146) and the Simons Foundation (SFI-MPS-TSM-00013970). Furthermore, the third author is thankful to the Department of Mathematics of the University of Thessaly, Lamia, Greece for their warm hospitality during February of 2025, when a significant portion of the present research was performed.
The authors are grateful to the reviewers for their careful assessment and constructive feedback that led to the improvement of the paper.
}
\date{October 31, 2025. \textit{Revised}: January 28, 2026}

\begin{abstract}
The question of well-posedness of the generalized focusing Ablowitz-Ladik and Discrete Nonlinear Schr\"{o}dinger equations with \textit{nonzero} boundary conditions on the infinite lattice is far less understood than in the case where the models are supplemented with vanishing boundary conditions. This question remains largely unexplored even in the standard case of cubic nonlinearities in which, in particular, the Ablowitz-Ladik equation is completely integrable while the Discrete Nonlinear Schr\"{o}dinger equation is not (in contrast with its continuous counterpart). We establish local well-posedness for both of these generalized nonlinear systems supplemented with a broad class of nonzero boundary conditions and, in addition, derive analytical upper bounds for the minimal guaranteed lifespan of their solutions. These bounds depend explicitly on the norm of the initial data, the background, and the nonlinearity exponents. In particular, they suggest the possibility of finite-time collapse (blow-up) of solutions. Furthermore, by comparing models with different nonlinearity exponents, we prove estimates for the distance between their respective solutions (measured in suitable metrics), valid up to their common minimal guaranteed lifespan. Highly accurate numerical studies illustrate that solutions of the generalized Ablowitz-Ladik equation may collapse in finite time. Importantly, the numerically observed blow-up time is in excellent agreement with the theoretically predicted order of the minimal guaranteed lifespan. Furthermore, in the case of the Discrete Nonlinear Schr\"{o}dinger equation on a finite lattice we prove global existence of solutions; this is consistent with our numerical observations of the phenomenon of \textit{quasi-collapse}, manifested by narrow oscillatory spikes that nevertheless persist throughout time. Once again, the time of the emergence of this phenomenon is in excellent agreement with the theoretically established minimal guaranteed lifespan. Notably, our numerical simulations confirm our theoretical result on the proximity of the dynamics between the two models over time scales up to the common solution lifespan. Finally, for power nonlinearities, we prove the asymptotic equivalence between the two discrete models in the continuous limit. 
\end{abstract}

\maketitle
\markboth
{D. Hennig, N.I. Karachalios, D. Mantzavinos, D. Mitsotakis}
{Discrete NLS versus AL:
Existence and dynamics of generalized focusing NLS-type lattices over a nonzero background}

{\hypersetup{linkcolor=black}
\tableofcontents}

\section{Motivation and main results}
\label{nzbc-s}

The focusing Ablowitz-Ladik (AL)  equation \cite{al1976b,al1976a,ha1989}
\begin{equation}\label{al}
i\frac{d u_n}{dt} + \kappa \left(\Delta u\right)_n +\frac{1}{2}\mu\,|u_{n}|^2\left(u_{n+1}+u_{n-1}\right)=0, \quad n\in {\mathbb{Z}}, \ t>0, \ \kappa, \mu >0,
\end{equation}
is an integrable discretization on lattices of the focusing cubic nonlinear Schr\"{o}dinger partial differential equation
\begin{equation}\label{nls}
i u_t + u_{xx} +\mu |u|^{2}u=0,\quad x \in \mathbb{R}, \ t>0, \ \mu>0.
\end{equation} 
The lattice model \eqref{al} involves the discrete Laplacian 
$
\kappa\left(\Delta u\right)_n = \kappa \left(u_{n+1}+u_{n-1}-2u_n\right)
$ 
where, typically, $\kappa=h^{-2}$ with $h$ denoting the distance between two subsequent lattice points. 
We remark that the AL and NLS models \eqref{al} and \eqref{nls}, as well as all other models considered in this work, are time-reversible and so are also valid for $t<0$. However, for simplicity of presentation, we take $t>0$.

The complete integrability of the AL equation \eqref{al} was established in  \cite{al1976a} where, in particular, it was shown that the AL equation admits an infinite number of conserved quantities. 
Furthermore, on the infinite lattice with vanishing boundary conditions, the AL equation was solved by means of the discrete version of the inverse scattering transform \cite{al1976b}. 
A particular solution of great interest is the one-soliton   
\begin{equation}\label{eq:one-soliton}
u^{\text S}_n(t) = \frac{\sqrt{2}\sinh(\beta h)}{h\sqrt{\mu}}\, \mathrm{sech}\left[\beta( hn-c t)\right] e^{i\left(\alpha hn-\omega t\right)}
\end{equation}
with 
$\omega= -2 \left[\cos(\alpha h)\cosh(\beta h)-1\right]/h^2$ and $c = 2\sin(\alpha h)\sinh(\beta h)/(\beta h^2)$ where $\alpha \in [-\pi,\pi]$, $\beta \in [0,\infty)$.
The inverse scattering transform method for the AL equation in the case of nonzero boundary conditions at infinity was developed in \cite{v1998,abp2007,vdm2015,b2016,pv2016}, and its dark $N$-soliton solutions were expressed in terms of the Casorati determinant in \cite{mo2006}. Another important class of solutions of the AL equation is that of rational solutions \cite{aas2010,aa2011}, which provide the discrete analogues to the corresponding solutions of the integrable NLS equation \eqref{nls}. One of the most famous representatives of this class is the discrete Peregrine soliton
\begin{equation}
 \label{DPeregrine}  
 u^{\text{PS}}_n(t)=\frac{q}{h\sqrt{\mu}}\left(1-\frac{4(1+q^2)(1+4iq^2t/h^2)}{1+4n^2q^2+16q^4(1+q^2)t^2/h^4}\right) e^{2iq^2 t/h^2}.
\end{equation}

Beyond the AL equation \eqref{al}, the focusing discrete nonlinear Schrödinger (DNLS) equation
\begin{equation}\label{dnls}
i \frac{d U_n}{dt} + \kappa \left(\Delta U\right)_n + \gamma |U_n|^2 U_n = 0,
\quad n \in \mathbb{Z}, \ t>0, \ \kappa, \gamma > 0,
\end{equation}
is one of the most important nonlinear lattice models. For more than forty years, DNLS has played a central role in modeling a wide range of physical phenomena, from nonlinear optics to localization effects in biological systems \cite{HenTsi99, PGKDNLS2, EilJo2003, PGKDNLS1}.

An intriguing feature of the AL and the DNLS lattices is that, although in the continuous limit $\kappa \rightarrow \infty$ they both approximate the completely integrable NLS equation \eqref{nls}, DNLS itself is not integrable. This lack of integrability also extends to generalizations of AL and DNLS and, in particular, to the following generalized AL (gAL) and generalized DNLS (gDNLS) equations: 
\begin{align}\label{gal}
&i\frac{d u_n}{dt} + \kappa \left(\Delta u\right)_n +\frac{1}{2}\mu\left|u_{n}\right|^{2p}\left(u_{n+1}+u_{n-1}\right)=0, \quad n\in {\mathbb{Z}}, \ t>0, \ \kappa, \mu >0, \ p\geq 1,
\\
\label{gdnls}
&i \frac{d U_n}{dt} + \kappa \left(\Delta U\right)_n+\gamma F( |U_n|^2) U_n = 0,
  \quad n \in \mathbb Z, \ t>0, \  \kappa, \gamma>0,
 \end{align}
where the nonlinearity function $F: \mathbb R \to \mathbb R$ in \eqref{gdnls} is assumed to satisfy the conditions
\begin{equation}\label{F-prop}
|F(x)-F(y)| \leq K(|x|^{p-1}+|y|^{p-1})|x-y|, 
\quad F(0)=0, \quad |F'(x)| \leq K |x|^{p-1},
\end{equation} 
for some $p\geq 1$, a constant $K>0$, and any $x,y \geq 0$. Note that the integrable AL equation \eqref{al} corresponds to \eqref{gal} with $p=1$. 

A key work studying the non-integrable gAL equation \eqref{gal} for $p>1$
with zero boundary conditions at infinity is \cite{PGKgAL}, which establishes the existence of discrete solitons and examines their bistability properties. One of the most exciting aspects of the gAL dynamics identified in that study is the potential for collapse, identified by the authors of \cite{PGKgAL} as \textit{“an intriguing question that merits further study both from a theoretical and from a numerical perspective”} and strongly motivating the present investigation of the gAL model, as it will be explained below. 

It is worth noting that the physical relevance of nonlocal nonlinearities like the one of the gAL equation \eqref{gal} in NLS and Ginzburg-Landau lattices is supported by numerous physical considerations. These include models of coupled waveguide arrays, where nonlocal terms become significant in higher-order approximations, especially when the penetration length is large or the waveguides are closely spaced \cite{Joh06}. Similar nonlocal terms also arise in nonlinear lattice models describing the evolution of amplitude in separatrices between vortex arrays or in coupled waves observed in low-dimensional hydrodynamic systems \cite{Wi91}.

On the other hand, the (cubic but not integrable) DNLS equation \eqref{dnls} is obtained from \eqref{gdnls} in the special case of $F(x)=x$. The condition \eqref{F-prop} also includes the physically significant rational nonlinearity $F(x) = \dfrac{x}{\Lambda(1+x)}$ with $\Lambda>0$, which gives rise to the \textit{saturable DNLS} equation~\cite{Satn2,Satn3,Satn1,Satn4}.

In the continuous limit $\kappa\rightarrow\infty$,   gAL and   gDNLS respectively approximate the generalized NLS (gNLS) partial differential equations
\begin{align}\label{gnlsp}
&i u_t + u_{xx} +\mu |u|^{2p}u=0,
\\
\label{gnlsF}
&i U_t + U_{xx} +\gamma F(|U|^2) U=0.
\end{align} 
However, there are crucial features that distinguish the dynamics of the continuous limits and their discrete counterparts. In particular, in the case of \textit{zero boundary conditions at infinity}:
\begin{enumerate}[label=$\bullet$, leftmargin=4mm, itemsep=1mm, topsep=2mm]
    \item \textit{gDNLS  with $p\geq 1$ versus NLS with $p\geq 1$:} The solutions of  gDNLS  \eqref{gdnls} exist unconditionally for any $p\geq 1$ and for all initial data \cite{ky2005,brc1994}. This is a vital difference from its continuous counterpart of NLS \eqref{gnlsF}, for which solutions may blow up in finite time when $p\geq 2$ and  global existence holds only for sufficiently small initial data \cite{BourgainBook,CazenaveBook}. 
    \item \textit{gAL with $p\geq 1$ versus gNLS with $p\geq 1$}: In the case $p=1$, gAL \eqref{gal} becomes the integrable AL lattice \eqref{al} for which solutions exist globally in time for all initial data. When $p>1$, \textit{it is unknown} if the solutions of gAL exist globally, unconditionally for all initial data as in the case of gDNLS~\eqref{gdnls}.  Numerical indication for potential collapse in finite time is given in \cite{PGKgAL}, as mentioned earlier.
\end{enumerate}

\textit{The questions of global existence and blow-up for gAL and gDNLS become  even more challenging in the case of \textit{nonzero boundary conditions at infinity}.} Motivated by the significance of such boundary conditions in the study of an abundance of physical phenomena associated with the presence of modulational instability in discrete settings \cite{onorato2013,maluckov2009,OhtaYang14,blmt2018, cs2024, cp2024, lcmmkk2025, bckop2025}, in the present work we consider the gAL equation \eqref{gal} and the gDNLS equation~\eqref{gdnls}  supplemented with the following  \textit{broad class} of nonzero boundary conditions at infinity:
\begin{equation}\label{uU-bc}
\lim_{|n|\rightarrow\infty} u_n(t) = \lim_{|n|\rightarrow\infty}e^{i\mu q_0^{2p} t}\zeta_n, 
\quad 
\lim_{|n|\rightarrow\infty} U_n(t) = \lim_{|n|\rightarrow\infty} e^{i\gamma F(q_0^2)t} \zeta_n, \quad t \geq 0,
\end{equation}
where, unless otherwise stated, the complex-valued function 
$\zeta=(\zeta_n)_{n\in \mathbb{Z}}$ belongs to the discrete Zhidkov space 
\begin{equation}\label{zhi1}
X^1(\mathbb Z)
:=
\left\{ \zeta\in \ell^\infty: \zeta'=(\zeta_n')_{n\in \mathbb Z}:=(\zeta_n-\zeta_{n-1})_{n \in \mathbb Z}\in \ell^2\right\},
\end{equation}
and satisfies the boundary conditions
\begin{equation}\label{zhi2}
\lim_{n \to \pm \infty} \zeta_n = \zeta_\pm \in \mathbb C, \quad |\zeta_\pm| = q_0>0.	
\end{equation}
We can think of $\zeta_n$ as coming from an element $\zeta$ of the Zhidkov space on the real line defined by $X^1(\mathbb R):=
\left\{ \zeta\in L^\infty(\mathbb R):  \zeta' \in L^2(\mathbb R)\right\}$ with $\zeta_n = \zeta(x_n)$ and $x_n = n h$ for some $h>0$ (we take $h=1$ unless otherwise stated). 
Note that, in view of \eqref{zhi2}, the boundary conditions \eqref{uU-bc} simplify to
\begin{equation*}
\lim_{n\rightarrow\pm\infty} u_n(t) = e^{i\mu q_0^{2p} t}\zeta_\pm, 
\quad 
\lim_{n\rightarrow\pm\infty} U_n(t) =  e^{i\gamma F(q_0^2)t} \zeta_\pm, \quad t \geq 0.
\end{equation*}

\begin{remark}
The assumption for $\zeta \in X^1(\mathbb Z)$ includes the standard case 
$\zeta_n=\zeta_0\in\mathbb{C}$ with $\zeta_0$ a constant such that $|\zeta_0|=q_0>0$, for which soliton solutions of the AL equation 
were treated within the framework of the inverse scattering transform  in \cite{b2016,abp2007,vdm2015,pv2016,op2019}.
\end{remark} 

With  the boundary conditions \eqref{uU-bc} in place, the asymptotics of the solutions $u_n(t)$ and $U_n(t)$ as $|n| \to \infty$ is governed by the function $\zeta_n$, which approaches a constant background of size $q_0>0$ according to the limit conditions~\eqref{zhi2}. 
We note that the boundary conditions \eqref{uU-bc} are a discretized version of the nonzero boundary conditions considered in the recent work \cite{hkmms2024} on the NLS equation~\eqref{gnlsF}. Not only are such conditions interesting from a mathematical point of view, but also they have been shown to be of great significance in applications. Indeed, even in experimental settings, the local emergence of localized structures on top of a non-vanishing background (reminiscent of rogue waves) has been observed. For such experimental observations, we refer the reader to \cite{tikan2017universality,copie2020physics}.

We now outline  the motivation for the questions to be addressed in the present work for the gDNLS and gAL lattices with the nonzero boundary conditions \eqref{uU-bc}, treating each case separately:
\begin{enumerate}[label=$\bullet$, leftmargin=4mm, itemsep=1mm, topsep=2mm]
\item \textit{gDNLS with $p \geq 1$ and gNLS  with $p \geq 1$}. 
While the cubic NLS (corresponding to \eqref{gnlsF} with the nonlinearity $F(x) = x$) is completely integrable when supplemented with the continuous analogue of the nonzero boundary conditions \eqref{uU-bc} in the case of a constant background, the cubic DNLS~\eqref{dnls} is not. In fact, the long-time dynamics of the cubic NLS over a constant background was rigorously derived to leading order in \cite{bm2016,bm2017,blm2021}, resolving what is sometimes referred to as the nonlinear stage of modulational instability. The same type of asymptotic behavior,  arising from localized initial conditions imposed on a constant background, was also detected numerically in a variety of continuous and discrete non-integrable systems, including the integrable AL and non-integrable DNLS models \eqref{al} and~\eqref{dnls}, as demonstrated in \cite{blmt2018}.

On the other hand, for the non-integrable gNLS with a general nonlinearity $F(x)$ described by \eqref{F-prop}, the question of global well-posedness under nonzero boundary conditions is far less understood than in the case of zero boundary conditions. Local existence of solutions consisting of $H^1$-perturbations over a nonzero background was recently established in \cite{hkmms2024}, while in \cite{hkmm2025} it was shown that blow-up can occur instantaneously for large initial data in the case of $p > 2$.
\item \textit{gAL with $p \geq 1$ and gNLS with $p \geq 1$}. 
The question of a possible dichotomy between global existence and finite-time blow-up for solutions of the gAL equation with 
$p>1$ remains \textit{entirely unexplored} in the case of nonzero boundary conditions at infinity.  Furthermore, in the scenario of global existence of solutions or of their potential large lifespan, gAL is a natural candidate to explore  possible persistence of the universal features of modulational instability that were described in~\cite{blmt2018}.
\end{enumerate}
Hence, for both the gDNLS and the gAL lattices, the following important questions arise:
\begin{enumerate}[label=(\roman*), leftmargin=8mm, itemsep=1mm, topsep=1mm]
    \item Do solutions of gDNLS with nonzero boundary conditions at infinity exist globally in time, unconditionally with respect to the size of the initial data and the exponent $p$, as in the case of zero boundary conditions? Or do blow-up phenomena emerge, depending on the size of the initial data and the value of $p$, as in the continuous case?
    \item Do solutions of gAL on the infinite lattice, under either nonzero or zero boundary conditions at infinity, exist globally in time? Or do blow-up phenomena occur, depending on the size of the initial data and the value of the exponent $p$, as in the case of the NLS equation \eqref{gnlsp}?
    \item When solutions of both systems exist globally or they have a large lifespan, is it possible to rigorously establish (or justify, using arguments as rigorous as possible)  information about their potential long-time behavior? For instance, does the dynamics exhibit features characteristic of the universal behavior associated with modulational instability highlighted in \cite{blmt2018}?
\end{enumerate}

The aim of the present work is to take a step forward in investigating the above fundamental questions for both the gAL system \eqref{gal} and the gDNLS system  \eqref{gdnls} in the presence of the nonzero boundary conditions \eqref{uU-bc}. Our main results can be classified into four main categories as outlined below.
\\[2mm]
\textbf{1. Local well-posedness over a nonzero background.} 
In the spirit of the analysis carried out for the continuous case in \cite{hkmms2024}, through suitable changes of variables we transform the original gAL and gDNLS problems with nonzero boundary conditions at infinity into ones for modified nonlinear lattices satisfying zero boundary conditions. This transition allows us to combine the discrete Fourier transform with  a fixed point argument in order to establish our first group of results, namely local Hadamard well-posedness (for more precise statements, see Theorems~\ref{lwp-dnls-t}-\ref{lwp-gal-t-2}):

\begin{theorem}[Local well-posedness]\label{lwp-t-intro}
For any $p\geq 1$, initial data in the $\ell^2$ space and nonzero boundary conditions at infinity of the form \eqref{uU-bc}, the Cauchy problems for the gAL and gDNLS equations~\eqref{gal} and~\eqref{gdnls} on the infinite lattice  are locally well-posed in the sense of Hadamard, namely they each admit a unique solution which depends continuously on the data. 
\end{theorem}

\noindent
\textbf{2. Lifespan, global existence and blow-up.}
As a byproduct of the contraction mapping approach used in the proof of Theorem \ref{lwp-t-intro}, we obtain the \textit{minimum guaranteed lifespan} of solutions for each of the gAL and gDNLS models. The fact that our results provide a minimum lifespan is illustrated by the integrable AL equation, which falls within our framework (as it corresponds to the gAL equation~\eqref{gal} with $p=1$) but has solutions with lifespan $T_{\max} = \infty$, thus exceeding the theoretically guaranteed lifespan.
To the best of our knowledge, this is the first rigorous demonstration of the potential dichotomy between global existence and finite-time blow-up for gAL and gDNLS with nonzero boundary conditions. 

Numerical simulations based on the fourth-order accurate and conservative numerical scheme of~\cite{m2025} show that for gAL with $p \geq 2$ blow-up in finite time may occur for sufficiently large initial data and background. Furthermore, it is remarkable that \textit{the numerical blow-up times quantitatively match the theoretically obtained minimum guaranteed lifespan of solutions}. In the case of gDNLS, while theoretical results for the infinite lattice still allow for the possibility of either global existence or blow-up, \textit{on the finite lattice} with either periodic or Dirichlet boundary conditions we rigorously establish the existence global-in-time solutions, 
independently of the size of the initial data and of the background. The proof relies on a conserved quantity involving the $\ell^2$ norm of the solution, as well as the background $\zeta_n$, over the relevant finite lattice. We emphasize that the validity of this argument independently of the size of the data relies on the finite-lattice setting. Our result becomes especially relevant  in view of numerical experiments with periodic boundary conditions that do not detect blow-up. In the case of the infinite lattice, however, the conserved quantity does not imply global existence, and the question of potential blow-up for gDNLS remains open. 

Finally, the case of zero boundary conditions can be treated as a corollary of our theoretical results for the nonzero background (see Remark \ref{case-iii}), in agreement with the scenario of potential blow-up for gAL  reported in \cite{PGKgAL}. 
\\[2mm]
\noindent
\textbf{3. Proximity of gAL and gDNLS dynamics.}
The third class of results concerns the comparison between the dynamics of the gAL and gDNLS systems. Extending the approach of \cite{hkmms2024} to the discrete setting and accounting for the potential global existence/blow-up dichotomy, we prove proximity estimates for the distance between the solutions of the two systems within their common minimum guaranteed lifespan. 
A non-technical description of our result can be given as follows (a more precise statement is provided in Theorem \ref{al-dnls-t}):
\begin{theorem}[Proximity]\label{prox-t-intro}
Given $\ve>0$, for nonlinearity parameters $p_1, p_2 \geq 1$ and initial data whose distance (in an appropriate norm) is of $\mathcal O(\max\{\ve^{2p_1+1}, \ve^{2p_2+1}\})$, there exists a time interval of $\mathcal O(\min\{\ve^{-2p_1}, \ve^{-2p_2}\})$ over which the distance of solutions (in an appropriate norm) between the gAL and gDNLS equations \eqref{gal} and \eqref{gdnls} formulated on the infinite lattice with the nonzero boundary conditions \eqref{uU-bc} remains of $\mathcal O(\max\{\ve^{2p_1+1}, \ve^{2p_2+1}\})$. 
\end{theorem}

A new and interesting numerical finding is that, in the case of small but nonzero background, the dynamics of the gAL and gDNLS equations with the \textit{same} nonlinearity parameter $p \geq 1$ remains proximal for significantly longer times than in the case of the continuous NLS family~\eqref{gnlsF} studied in \cite{hkmms2024}.
Indeed, in \cite{hkmms2024} it was observed that, while the main features of the universal modulational instability pattern persisted in the non-integrable NLS setting, most of the deviation originated from the \textit{finer oscillatory structures} far from the core of the pattern. As a result, \textit{the norms of the distance grew more rapidly in regions far from the core}. In the present work, we observe that this is \textit{not} the case for the lattice systems \eqref{gal} and \eqref{gdnls}. This important difference is highlighted by the proximity study for $p=1$, namely by comparing the integrable AL and non-integrable DNLS equations. In this case, the oscillations in the DNLS dynamics, even far from the core,  closely resemble those of AL, leading to a significantly smaller growth of the distance norm in excellent agreement with the theoretical proximity estimates, although the oscillatory pulses of DNLS have slightly higher propagation speed than those of  AL. The same effect is observed for $p > 1$. 

In fact, our numerical findings for $p>1$ underscore the value of the minimum guaranteed lifespan when investigating proximity. When $p>1$ and blow-up occurs for gAL, the dynamics of gDNLS with the same nonlinearity parameter remains significantly close up to the blow-up time of the gAL solution, which agrees remarkably well with the order of the theoretically predicted minimum guaranteed lifespan. For the gDNLS solution, which exists globally, the maximal amplitude predicted by our theory is attained precisely at times of the order of the minimum guaranteed lifespan (see Theorem \ref{dnls-life-t}) and is maintained thereafter. The evolution of the DNLS solution post its minimum guaranteed lifespan exhibits an interesting phenomenon reminiscent of \textit{quasi-collapse}. This quasi-collapse provides the mechanism for the formation of very narrow, self-trapped states in the lattice, without the occurrence of a finite-time singularity \cite{christiansen1996discrete}.
Thus, on one hand --- and in agreement with our theoretical results --- finite in time blow-up cannot be observed for solutions of gDNLS on the finite lattice, even for initial data resembling perturbations of a nonzero background, while on the other hand the theory of the minimum guaranteed lifespan appears to accurately predict the quasi-collapse phenomenon for such gDNLS solutions on a finite lattice.
\\[2mm]
\textbf{4. Asymptotic equivalence between gAL and gDNLS.}
The present work concludes with the study of another important topic, namely the \textit{asymptotic equivalence between the gAL and the gDNLS equations with power nonlinearity and $p\geq 1$ in terms of the discretization parameter $h$}. 
Specifically, we prove that there is a time $T>0$ such that the distance between the solutions of the gAL and gDNLS equations \eqref{al} and \eqref{dnls} measured in the $\ell^{\infty}$-metric satisfies the estimate
\begin{equation}
\label{AES}
\no{(u_n-U_n)(t)}_{\ell^\infty} \leq C h^2t,
\end{equation}
for all $0<h\leq 1$ and $0<t\leq T$, where  $C>0$ is a constant \textit{independent of $h$}.
The significance of  estimate~\eqref{AES} is twofold. In the discrete regime, it provides, via a different route, a justification that the dynamics of gAL and gDNLS deviates at most at a linear rate, consistent with the proximity estimates of Theorem~\ref{prox-t-intro}. In the continuous regime $h \ll 1$, it implies that the solutions remain proximal for times of $\mathcal{O}(h^{-2})$. The strength of estimate~\eqref{AES} in this latter regime lies in showing that, as $h \to 0$, \textit{the dynamics of the integrable AL and the non-integrable DNLS equations becomes equivalent}, and is formally known to be governed by the integrable NLS. Since $\ell^{\infty} \to L^{\infty}$ as $h \to 0$, the asymptotic equivalence holds in the topology of strong convergence. To the best of our knowledge, the result proved here is the first to rigorously establish the equivalence of the dynamics in the continuous limit.

\begin{remark}[Defocusing case]
Although the objective of the present work is the analysis of the focusing gAL and gDNLS systems, most of our theoretical results are also valid in the defocusing case of $\mu, \gamma < 0$. With this in mind, the proofs of the various theorems (including the precise versions of Theorems \ref{lwp-t-intro} and \ref{prox-t-intro}) are presented for general $\mu, \gamma \in \mathbb R$ rather than being restricted to $\mu, \gamma > 0$. However, it is important to highlight a crucial difference between the focusing and defocusing gAL lattices. While the local existence of Theorem \ref{lwp-t-intro} is valid for both systems for any $p \geq 1$, in the defocusing case the potential dichotomy between local existence and global-in-time existence of solutions may depend on the amplitude of the nonzero background, as shown for the defocusing AL lattice ($p=1$) in \cite{cs2024,OhtaYang14}. A detailed theoretical and numerical study of the defocusing case investigating this dichotomy is the subject of ongoing work.
\end{remark}

\noindent
\textit{Structure.} In Section \ref{lwp-s}, we establish precise versions of Theorem \ref{lwp-t-intro} on the local well-posedness of the gDNLS and gAL equations on the infinite lattice with the nonzero boundary conditions \eqref{uU-bc}.
In Section \ref{lifespan-s}, we take advantage of our local well-posedness results in order to precisely relate the size of the gAL and gDNLS solutions on the infinite lattice with their associated lifespans. Furthermore, based on a suitable conservation law, we prove global existence for gDNLS on a finite lattice with either homogeneous Dirichlet or periodic boundary conditions. 
In Section \ref{prox-s}, we establish estimates for the distance between the solutions of the gAL and gDNLS systems on the infinite lattice with the nonzero boundary conditions \eqref{uU-bc}, leading to the precise version of Theorem \ref{prox-t-intro}. \textit{All of these theoretical results are valid in both the focusing and the defocusing case.}
In Section \ref{as-s}, we establish estimate \eqref{AES} on the asymptotic equivalence between the AL and DNLS systems.
Finally, in Section \ref{sim-s}, we present a variety of numerical simulations in the focusing case, which are in excellent agreement with the theoretical results of the precedings sections. 

\section{Local well-posedness}
\label{lwp-s}
In this section, we establish precise versions of Theorem \ref{lwp-t-intro} for the local Hadamard well-posedness of the Cauchy problems  for the gDNLS and gAL equations on the infinite lattice and supplemented with the nonzero boundary conditions \eqref{uU-bc}.  
The section is organized into three subsections. In the first, we perform a  transformation that leads to modified gDNLS and gAL equations but with \textit{zero} boundary conditions at infinity, thus allowing us to formulate a contraction mapping argument by employing the discrete Fourier transform. 
In the second subsection, we prove local well-posedness for gDNLS. In addition, we derive a conservation law that will be useful later in Section \ref{lifespan-s}, when discussing the potential scenarios of global existence and blow-up for gDNLS.  
Finally, in the third subsection we establish local well-posedness for gAL. This task turns out to be more challenging than in the case of gDNLS due to the nonlocal nonlinearity present in gAL.

\subsection{Modified equations with zero boundary conditions at infinity}

We first note that the precise form of the nonzero boundary conditions \eqref{uU-bc} can be motivated as follows.
Let $u_n(t)$ and $U_n(t)$ be solutions of  equations \eqref{gal} and \eqref{gdnls}, respectively, such that 
\begin{equation}\label{upm-mot-0}
\lim_{n\to\pm\infty}u_n(t) = u_\pm(t),
\quad
\lim_{n\to\pm\infty}U_n(t) = U_\pm(t),
\end{equation}
where $u_\pm(t)$, $U_\pm(t)$ are functions of time but with constant modulus equal to $q_0>0$. Taking the limit of~\eqref{gal} and \eqref{gdnls} as $n\to\pm\infty$ while assuming that the nonlinearity function $F$ has sufficient smoothness so that $\lim_{|n| \to \infty} F(|U_n|^2) = F(\lim_{|n|\to\infty} |U_n|^2) = F(q_0^2)$, we have
\begin{equation}\label{upm-mot}
\begin{aligned}
&i \frac{d u_\pm}{dt} + \mu q_0^{2p} u_\pm = 0 \ \Rightarrow \ u_\pm(t) = e^{i\mu q_0^{2p} t}u_\pm(0),
\\
&i \frac{dU_\pm}{dt} + \gamma F(q_0^2) U_\pm = 0 \ \Rightarrow \ U_\pm(t) = e^{i\gamma F(q_0^2)t} U_\pm(0).
\end{aligned}
\end{equation}
Hence, if the initial data 
\begin{equation}\label{ic}
u_n(0) = u_{n, 0}, 
\quad 
U_n(0) = U_{n, 0},
\end{equation}
have the same nonzero limit $\zeta_\pm$ as $n\to\pm\infty$, 
\begin{equation}
\lim_{n\to\pm\infty} u_{n, 0} = \lim_{n\to\pm\infty} U_{n, 0} = \zeta_\pm,
\end{equation}
then (assuming sufficient smoothness that allows us to interchange the limits with respect to $n$ and $t$)   $u_\pm(0) = U_\pm(0) = \zeta_\pm$ so that, for $\zeta_n$ satisfying \eqref{zhi2}, the expressions \eqref{upm-mot-0} and \eqref{upm-mot} give rise to~\eqref{uU-bc}.

The boundary conditions \eqref{uU-bc} can be converted to vanishing ones via the change of variables
\begin{equation}\label{zhi4}
u_n(t) = e^{i\mu q_0^{2p}t} \left(\phi_n(t) + \zeta_n\right), 
\quad 
U_n(t) = e^{i\gamma F(q_0^2)t} \left(\Phi_n(t) + \zeta_n\right).
\end{equation}
Then, equations \eqref{gal} and \eqref{gdnls} take the form
\begin{align}
\label{mgal}
&i \frac{d\phi_n}{dt} + \kappa \left(\Delta\left(\phi+\zeta\right)\right)_n 
-\mu q_0^{2p} \left(\phi_n + \zeta_n\right)
+ \frac 12 \mu \left|\phi_n + \zeta_n\right|^{2p} \left(\phi_{n+1} + \zeta_{n+1} + \phi_{n-1} + \zeta_{n-1}\right) = 0,
\\
\label{mgdnls}
&i \frac{d \Phi_n}{dt} + \kappa \left(\Delta (\Phi+\zeta)\right)_n +\gamma \left[F(|\Phi_n+\zeta_n|^2) - F(q_0^2)\right] \left(\Phi_n+\zeta_n\right)=0,
\end{align}
with associated initial conditions 
\begin{equation}\label{mod-ic}
\phi_n(0) = u_{n,0} - \zeta_n,
\quad
\Phi_n(0) = U_{n,0} - \zeta_n,
\end{equation}
and zero boundary conditions at infinity
\begin{equation}\label{vbcn}
\lim_{|n|\rightarrow\infty}\phi_n(t)
=
\lim_{|n|\rightarrow\infty}\Phi_n(t) = 0,\quad t \geq 0.
\end{equation}

\begin{remark}
It is interesting to observe that the zero solution is not always admissible by the modified equations \eqref{mgal} or \eqref{mgdnls}. Its admissibility requires that $\zeta$ satisfy, respectively, the stationary gAL and gDNLS equations
\begin{align}
&\kappa \left(\Delta \zeta\right)_n -\mu q_0^{2p} \zeta_n + \frac{1}{2}\mu |\zeta_n|^{2p} \left(\zeta_{n+1}+\zeta_{n-1}\right)=0,
\label{mal-stat}
\\
&\kappa \left(\Delta \zeta\right)_n +\gamma \left[F(|\zeta_n|^2) - F(q_0^2)\right] \zeta_n=0.
\label{pNLS-stat}
\end{align}
\end{remark}

In view of the literature on the existence of stationary states or standing waves for DNLS systems, the solvability of the equations \eqref{mal-stat} or \eqref{pNLS-stat} with the requirement that $\zeta \in X^1(\mathbb{Z})$ may deserve independent interest. Continuous counterparts (nonlinear elliptic equations) of the form of the discrete equations \eqref{mal-stat} or \eqref{pNLS-stat} have been studied in homogeneous Sobolev spaces $D^{1,2}(\mathbb{R^N})$, $N \geq 1$ (also denoted by $\dot{H}^1(\mathbb{R^N})$), which exhibit certain analogies with Zhidkov spaces $X^1(\mathbb{R}^N)$, particularly in the case of dimension $N=1$, where $D^{1,2}(\mathbb{R}) \subset L^{\infty}(\mathbb{R})$. The space $D^{1,2}(\mathbb{R})$ is the completion of $C_0^{\infty}(\mathbb{R})$ with respect to the norm 
$
\no{u}_{D^{1,2}(\mathbb{R})}^2=\int_{\mathbb{R}} |u'|^2 \, dx.
$
However, it is known that a function $u \in D^{1,2}(\mathbb{R})$ may not belong to $L^q(\mathbb{R})$ for $2 < q < \infty$. For further details, we refer to pages 8-9 in \cite{brown1996global} and \cite{chabrowski2008class}.

It is also useful to recall the continuous embedding
$\ell^r \subseteq \ell^q \subseteq \ell^\infty$ whenever $r \leq q \leq \infty$, which will be used often in our analysis. Note that this embedding is in contrast with the case of the spaces $L^p(\Omega)$ when $\Omega$ has finite measure, where the ordering of the exponents is reversed \cite{adams2003}.

For the sake of completeness we recall the proof. For the first embedding, it suffices to prove that
\begin{equation}\label{emb-ineq}
\Big(\sum_{n\in\mathbb Z} |\zeta_n|^q\Big)^{\frac 1q} =: \no{\zeta}_{\ell^q} \leq \no{\zeta}_{\ell^r} := \Big(\sum_{n\in\mathbb Z} |\zeta_n|^r\Big)^{\frac 1r}, \quad r \leq q.
\end{equation}
Suppose that $\no{\zeta}_{\ell^r} = \lambda$. Then, $\sum_{n\in\mathbb Z} |\zeta_n|^r = \lambda^r$ i.e. $\sum_{n\in\mathbb Z} |\xi_n|^r = 1$ where $\xi_n := \frac{\zeta_n}{\lambda}$, 
thus \eqref{emb-ineq} follows from showing that $\sum_{n\in\mathbb Z} |\xi_n|^q \leq 1$. 
But $\sum_{n\in\mathbb Z} |\xi_n|^r = 1$, so $|\xi_n| \leq 1$ for all $n\in\mathbb Z$, which yields the desired inequality. Moreover, the second embedding follows from \eqref{emb-ineq} and the simple observation that 
$
\sup_{n\in\mathbb Z} |\zeta_n| =: \no{\zeta}_{\ell^\infty} \leq  \no{\zeta}_{\ell^1} := \sum_{n\in\mathbb Z} |\zeta_n|.
$

\subsection{Local well-posedness of gDNLS}
We begin with the following auxiliary result concerning the nonlinear operators appearing in the gDNLS lattice \eqref{gdnls}.
\begin{lemma}\label{lip-l}
Let $\zeta \in X^1(\mathbb Z)$. The operator 
\begin{equation}\label{G-def}
\Phi \mapsto G(\Phi):=\left[F(|\Phi+\zeta|^2) - F(q_0^2)\right] (\Phi+\zeta)
\end{equation}
satisfies the inequalities
\begin{align}
&\no{G(\Phi)}_{\ell^2}
\leq
2\sqrt 2 \, K  \left( \left\| \Phi \right\|_{\ell^\infty} + \left\| \zeta \right\|_{\ell^\infty} + q_0\right)^{2p}
\left(\left\| \Phi \right\|_{\ell^2} + \big\| |\zeta| - q_0 \big\|_{\ell^2}\right),
\label{into-ineq}
\\
&\no{G(\Phi)- G(\Psi)}_{\ell^2}
\leq
K \left(\no{\Phi}_{\ell^\infty}+\no{\Psi}_{\ell^\infty}+2\no{\zeta}_{\ell^\infty}+q_0\right)^{2p} \big\| \Phi-\Psi\big\|_{\ell^2}.
\label{contr-ineq}
\end{align}
\end{lemma}

\begin{proof}
In view of \eqref{F-prop}, and similarly to the proof of inequality (3.21) in \cite{hkmms2024}, we have
\begin{equation} 
\left|F(|\Phi+\zeta|^2) - F(q_0^2)\right|
\leq
2K \left(|\Phi| + |\zeta| + q_0 \right)^{2p-1}
\left(|\Phi|+\big| |\zeta| - q_0 \big|\right).
\nn
\end{equation}
Hence, 
\begin{align}\label{into-ineq}
\no{G(\Phi)}_{\ell^2}
&\leq
2K
\bigg(
\sum_{n\in\mathbb Z}
\left(|\Phi_n| + |\zeta_n| + q_0\right)^{2(2p-1)} 
\left( 
|\Phi_n| + \big| |\zeta_n| - q_0 \big|
\right)^2
 \left(|\Phi_n|+|\zeta_n|\right)^2 
 \bigg)^{\frac 12}
\nn\\
&\leq
2 \sqrt 2 K 
\sup_{n\in\mathbb Z} 
\left[\left(|\Phi_n| + |\zeta_n| + q_0\right)^{2p-1} 
\left(|\Phi_n|+|\zeta_n|\right)\right]
\Big(
\sum_{n\in\mathbb Z}
\big( 
|\Phi_n|^2 + \big| |\zeta_n| - q_0 \big|^2
\big) 
\Big)^{\frac 12},
\nn
\end{align}
from which we readily obtain the inequality \eqref{into-ineq}. The inequality \eqref{contr-ineq} can be established in a similar way (see also the argument leading to the penultimate inequality in (3.38) of \cite{hkmms2024}). 
\end{proof}

Next, combining Lemma \ref{lip-l} with a contraction mapping argument, we establish local well-posedness for the modified gDNLS Cauchy problem in the sense of Hadamard. More precisely, we prove 
\begin{theorem}[Local well-posedness of modified gDNLS]\label{lwp-dnls-t}
Let $\zeta \in X^1(\mathbb Z)$. For any $p\geq 1$, the modified gDNLS  \eqref{mgdnls} with initial condition $\Phi(0) = (\Phi_n(0))_{n \in \mathbb Z} \in \ell^2$ specified by \eqref{mod-ic} and zero boundary conditions at infinity as described by~\eqref{vbcn} possesses a unique solution $\Phi \in \overline{B(0, \rho)} \subset C([0,T];\ell^2)$, where $B(0, \rho)$ denotes the open ball centered at the origin and with radius   given by
\begin{equation}\label{eq:rhoPhi}
\rho = \rho(T)
=
2
 \left[
  \no{\Phi(0)}_{\ell^2} + \sqrt{2\kappa} \no{\zeta'}_{\ell^2} \sqrt T
  +
2^{2p+\frac 32} |\gamma| K  \left(\left\| \zeta \right\|_{\ell^\infty} + \big\| |\zeta| - q_0 \big\|_{\ell^2}+q_0\right)^{2p+1} T
  \right],
\end{equation}
and the lifespan $T>0$ satisfies
\begin{equation}\label{contr-cond}
 2^{2p+\frac 52} |\gamma| K \left(\rho(T)+ \no{\zeta}_{\ell^\infty}+q_0\right)^{2p}  T \leq 1.
 \end{equation}
Furthermore, the solution depends continuously on the initial data.
\end{theorem}

\begin{remark}\label{mgl-r}
Theorem \ref{lwp-dnls-t} was established via a contraction mapping argument which requires that the lifespan $T$ satisfy the condition \eqref{contr-cond}. However, it is possible that the solution persists for longer times. In this regard, we refer to $T$ in \eqref{contr-cond} as the \textbf{minimum guaranteed lifespan} of the solution.
\end{remark}

\begin{proof}
The first step is to recast the Cauchy problem as an integral equation. For this purpose, we employ the finite Fourier transform (also known as discrete-time Fourier transform), which for $f\in \ell^2$ is defined by
\begin{equation}\label{dft-def}
\what f(\xi) := \sum_{n\in {\mathbb{Z}}} f_n e^{-i\xi n}, \quad \xi \in [0, 2\pi],
\end{equation}
with inversion formula 
\begin{equation}\label{dft-inv}
f_n=\frac{1}{2\pi}\int_{0}^{2\pi} e^{i \xi n} \what f(\xi) d\xi, \quad n\in\mathbb Z.
\end{equation}
Notice that
$$
\sum_{n\in\mathbb Z} f_{n+1} e^{-i\xi n} = e^{i\xi} \sum_{n\in\mathbb Z} f_{n+1} e^{-i\xi (n+1)} = e^{i\xi} \sum_{n\in\mathbb Z} f_n e^{-i\xi n} = e^{i\xi} \what f(\xi).
$$
In view of this property, applying \eqref{dft-def} to the gDNLS equation \eqref{mgdnls} after writing $\zeta_{n+1} + \zeta_{n-1}-2\zeta_n = \zeta'_{n+1}-\zeta_n'$, we obtain
$$
i \p_t \what \Phi(\xi, t) + \kappa \big(e^{i\xi} + e^{-i\xi} - 2\big) \what \Phi(\xi, t) + \kappa \big(e^{i\xi}-1\big) \what{\zeta'}(\xi) + \gamma \what{G(\Phi)}(\xi) =0,
$$
with $G(\Phi)$ as in Lemma \ref{lip-l}. Then, noting that $e^{i\xi} + e^{-i\xi} - 2 = 2\left(\cos \xi - 1\right) = -4\sin^2(\frac \xi 2)$ and integrating with respect to $t$, we find
\begin{align}
\what \Phi(\xi, t) 
&=
e^{-4i\kappa \sin^2 \left(\frac{\xi}{2} \right)t} \, \what{\Phi(0)}(\xi) - i\int_0^t 
e^{-4i\kappa \sin^2 \left(\frac{\xi}{2}\right)(t-\tau)}\left[\kappa \big(e^{i\xi}-1\big)\what{\zeta'}(\xi)+
\gamma  \what{G(\Phi)}(\xi, \tau)\right] d\tau
\label{fhat0}\\
&=
e^{-4i\kappa \sin^2 \left(\frac{\xi}{2} \right)t} \, \what{\Phi(0)}(\xi) 
+
i e^{i\frac \xi 2} \,  \frac{1-e^{-4i\kappa \sin^2 \left(\frac{\xi}{2} \right)t}}{2 \sin \left(\frac{\xi}{2} \right)} \, \what{\zeta'}(\xi)
- i \gamma \int_0^t e^{-4i\kappa \sin^2 \left(\frac{\xi}{2}\right)(t-\tau)} \what{G(\Phi)}(\xi, \tau) d\tau.
\nn
\end{align}
Hence, by means of the inversion formula \eqref{dft-inv}, we obtain the  integral equation 
\begin{equation}\label{eq:maplocal}
 \Phi(t)=\Lambda [\Phi](t)
\end{equation}
where 
\begin{equation}
\begin{aligned}
(\Lambda[\Phi])_n(t) &:= \frac{1}{2\pi}\int_0^{2\pi} e^{i\xi n} \bigg[e^{-4i\kappa \sin^2 \left(\frac{\xi}{2} \right)t} \, \what{\Phi(0)}(\xi) + i e^{i\frac \xi 2} \,  \frac{1-e^{-4i\kappa \sin^2 \left(\frac{\xi}{2} \right)t}}{2 \sin \left(\frac{\xi}{2} \right)} \, \what{\zeta'}(\xi) \bigg] d\xi 
\\
&\quad
- i\gamma \int_0^t \frac{1}{2\pi}\int_0^{2\pi} e^{i\xi n-4i\kappa \sin^2 \left(\frac{\xi}{2}\right)(t-\tau)} 
  \what{G(\Phi)}(\xi, \tau)  d\xi d\tau.
\end{aligned}
\end{equation}

In light of the above computation, we specify our notion of solution to the modified gDNLS Cauchy problem as the solution to the integral equation \eqref{eq:maplocal}. To this end, we will employ Banach's fixed point theorem to establish that the map $\Phi \mapsto \Lambda[\Phi]$ possesses a unique fixed point in an appropriate subset of the space $C([0,T_f];\ell^2)$ for some $T_f>0$ to be determined. Equivalently, this will imply a unique solution to the integral equation \eqref{eq:maplocal} and, therefore, to the modified gDNLS Cauchy problem.

In order to carry out the above plan, we need to show that $\Phi \mapsto \Lambda[\Phi]$ is a contraction in an appropriate subset of $C([0,T_f];\ell^2)$. In this connection, using the triangle inequality, Parseval's theorem and Minkowski's integral inequality, we have
\begin{align}\label{D-L20}
\left\| \Lambda[\Phi](t) \right\|_{\ell^2}
&\leq
\frac{1}{\sqrt{2\pi}} \, \bigg\|e^{-4i\kappa \sin^2 \left(\frac{\xi}{2} \right)t} \, \what{\Phi(0)}(\xi) + i e^{i\frac \xi 2} \, \frac{1-e^{-4i\kappa \sin^2 \left(\frac{\xi}{2} \right)t}}{2 \sin \left(\frac{\xi}{2} \right)} \, \what{\zeta'}(\xi)\bigg\|_{L^2(0, 2\pi)}
\nn\\
&\quad
+
\frac{|\gamma|}{\sqrt{2\pi}} \int_0^t \left\| e^{-4i\kappa \sin^2 \left(\frac{\xi}{2}\right)(t-\tau)}   \what{G(\Phi)}(\xi, \tau) \right\|_{L^2(0, 2\pi)} d\tau
\nn\\
&=
\no{\Phi(0)}_{\ell^2}
+
\frac{1}{\sqrt{2\pi}} \, \bigg\| \frac{1-e^{-4i\kappa \sin^2 \left(\frac{\xi}{2} \right)t}}{2 \sin \left(\frac{\xi}{2} \right)} \, \what{\zeta'}(\xi) \bigg\|_{L^2(0, 2\pi)}  
+
|\gamma| \int_0^t  \big\| G(\Phi)(\tau) \big\|_{\ell^2} d\tau, 
\end{align}
where we have also used the fact that the time exponential is unitary. 

Noting that $\left|\frac{1-e^{-i\kappa \theta^2 t}}{\theta}\right|^2 = 2 \kappa t \left|\frac{1-\cos(\kappa \theta^2 t)}{\kappa \theta^2 t}\right| \leq 2 \kappa t$ in view of the inequality $\left|\frac{1-\cos(\theta)}{\theta}\right| \leq 1$, $\theta \in \mathbb R$, we estimate the second term in \eqref{D-L20} as follows:
\begin{align}\label{eta-est}
\bigg\| \frac{1-e^{-4i\kappa \sin^2 \left(\frac{\xi}{2} \right)t}}{2 \sin \left(\frac{\xi}{2} \right)} \, \what{\zeta'}(\xi) \bigg\|_{L^2(0, 2\pi)}^2
&=
\int_0^{2\pi} \bigg|\frac{1-e^{-4i\kappa \sin^2 \left(\frac{\xi}{2} \right)t}}{2 \sin \left(\frac{\xi}{2} \right)}\bigg|^2 \left|\what{\zeta'}(\xi) \right|^2 d\xi
\nn\\
&\leq
2\kappa t \int_0^{2\pi} \big|\what{\zeta'}(\xi) \big|^2 d\xi
=
 2\kappa t \cdot 2\pi \no{\zeta'}_{\ell^2}^2. 
\end{align}
The $\ell^2$ norm of $G$ involved in the third term of \eqref{D-L20} can be handled via \eqref{into-ineq} and the embedding $\ell^2\subset \ell^\infty$ as follows:
\begin{align}
 \big\| G(\Phi) \big\|_{\ell^2} &\le 2\sqrt 2 \, K  \left( \left\| \Phi \right\|_{\ell^2} + \left\| \zeta \right\|_{\ell^\infty} + q_0\right)^{2p}\left(\left\| \Phi \right\|_{\ell^2} + \big\| |\zeta| - q_0 \big\|_{\ell^2}\right)
 \nn\\
 &\le 2\sqrt 2 \, K\left( \left\| \Phi \right\|_{\ell^2} + \left\| \zeta \right\|_{\ell^\infty} + \big\| |\zeta| - q_0 \big\|_{\ell^2}+q_0\right)^{2p+1}
\nn\\
 &\le 2\sqrt 2 \, K \cdot 2^{2p}\left( \left\| \Phi \right\|_{\ell^2}^{2p+1} + \left(\left\| \zeta \right\|_{\ell^\infty} + \big\| |\zeta| - q_0 \big\|_{\ell^2}+q_0\right)^{2p+1}\right),
 \nn
\end{align}
where we have used the inequality $(a+b)^\beta\le 2^{\beta-1}(a^\beta+b^\beta)$ for $a,b\ge 0$ and $\beta\ge1$, which follows from Jensen's inequality applied to the convex function $x^\beta$, $\beta\geq 1$.
In turn, we obtain
\begin{align}
 \left\| \Lambda[\Phi](t) \right\|_{\ell^2}
 &\leq
 \no{\Phi(0)}_{\ell^2} + \sqrt{2\kappa} \no{\zeta'}_{\ell^2} \sqrt t
\nn\\
&\quad+
2^{2p+\frac 32} |\gamma| K \left(\sup_{\tau \in [0, t]} \left\| \Phi(\tau) \right\|_{\ell^2}^{2p+1} + \left(\left\| \zeta \right\|_{\ell^\infty} + \big\| |\zeta| - q_0 \big\|_{\ell^2}+q_0\right)^{2p+1}\right)  t.
\label{eq:Lambda1}
\end{align}

Let $\rho=\rho(T)$ be given by \eqref{eq:rhoPhi} with $T>0$ to be determined. 
If $\Phi \in \overline{B(0, \rho)}$ then, by inequality~\eqref{eq:Lambda1}, in order for $\Lambda[\Phi] \in \overline{B(0, \rho)}$ it suffices to have
 $\frac{\rho(T)}{2}+ 2^{2p+\frac 32} 
 |\gamma| K  \rho(T)^{2p+1} \, T \le \rho(T)$ or, equivalently,
\begin{equation}\label{into-cond}
2^{2p+\frac 52} |\gamma| K  \rho(T)^{2p} \, T \leq 1.
\end{equation}
Furthermore, for any $\Phi, \Psi \in \overline{B(0, \rho)}$, using the estimate \eqref{contr-ineq} and the embedding $\ell^2 \subset \ell^\infty$, we have
\begin{align}
 \no{ \Lambda[\Phi](t)-\Lambda[\Psi](t)}_{\ell^2}
 &\leq 
 |\gamma| \int_0^t \no{G(\Phi(\tau)) -G(\Psi(\tau))}_{\ell^2}d\tau
 \nonumber\\
 &\leq 
 |\gamma| K\int_0^t \left(\no{\Phi(\tau)}_{\ell^2}+\no{\Psi(\tau)}_{\ell^2}+2\no{\zeta}_{\ell^\infty}+q_0\right)^{2p} \big\| \Phi(\tau)-\Psi(\tau)\big\|_{\ell^2} d\tau
 \nonumber\\
 &\leq 
2^{2p}  |\gamma| K  \left(\rho(T)+\no{\zeta}_{\ell^\infty}+q_0\right)^{2p}  T \sup_{t\in [0, T]} \big\| \Phi(t)-\Psi(t)\big\|_{\ell^2}.
\label{ldiff}
 \end{align}
Then, for $\Phi \mapsto \Lambda[\Phi]$ to be a contraction on $\overline{B(0, \rho)}$, it suffices to require that $T>0$ satisfy the following condition in addition to \eqref{into-cond}:
\begin{equation}
 2^{2p+1} |\gamma| K \left(\rho(T)+ \no{\zeta}_{\ell^\infty}+q_0\right)^{2p}  T \leq 1.\label{eq:Tf1}
\end{equation}
The conditions \eqref{into-cond} and \eqref{eq:Tf1} can be combined into the stronger condition \eqref{contr-cond}.
Hence, by Banach's fixed point theorem, for $T>0$ satisfying the condition \eqref{contr-cond} the map $\Phi \mapsto \Lambda[\Phi]$ possesses a unique fixed point in $\overline{B(0, \rho)}\subset C([0,T];\ell^2)$. 
Continuity with respect to the initial data follows by using inequality~\eqref{contr-ineq} along the lines of the argument that led to \eqref{ldiff}.
\end{proof}

\begin{remark}[Discrete versus continuous]
In the continuous case studied in \cite{hkmms2024}, the counterpart of estimate \eqref{eta-est} involves a smoothing effect, namely the $L^2(\mathbb R)$ norm of the derivative $\zeta'$ of the background $\zeta$ controls the $H_x^1(\mathbb R)$ norm (as opposed to the $L_x^2(\mathbb R)$ norm) of $\Lambda[\Phi](t)$. The absence of this  smoothing effect in the discrete case, as seen from \eqref{eta-est}, is a reflection of the fact that the discrete $H^1$ norm is controlled by the $\ell^2$ norm, namely 
\begin{equation}
\no{\nabla u}_{\ell^2}^2 =\sum_{n \in \mathbb Z}\left|\left(\nabla u \right)_n\right|^2=\sum_{n \in \mathbb Z} \left|u_{n+1}-u_n\right|^2\leq 4\no{u}_{\ell^2}^2.
\end{equation}
Hence, $H^1(\mathbb Z)$ and $\ell^2(\mathbb Z)$ are equal as sets and so the discrete  $H^1$ space does not imply higher regularity.
\end{remark}
\begin{remark}[Generalized Picard-Lindel\"of theorem and integral formulae]
It is worth noting that for DNLS-type nonlinear lattices, which can be viewed as systems of coupled ODEs in $\ell^2$, an existence proof could alternatively be constructed by applying the generalized Picard-Lindel\"of theorem \cite{zei85a} to the standard integral formula for the corresponding abstract ODE in sequence spaces. Here, we apply the fixed point argument to the integral equation~\eqref{eq:maplocal} derived from the discrete Fourier transform~\eqref{dft-def}, which is actually the variation of parameters integral formula (Duhamel's formula) in the Fourier space. We use~\eqref{eq:maplocal}  for the following reasons.  First, this approach facilitates a direct comparison with the continuous NLS limit and highlights the structural analogies explored in our previous work \cite{hkmms2024}, as well as the important differences with the continuous NLS limit. For instance, despite the similarities with the formulation of Theorem 3.1 in \cite{hkmms2024} for the continuous case, the 1D-Laplacian operator is not a bounded operator in $L^2(\mathbb{R})$, necessitating its proper Friedrichs extension on the suitable domain, in the Sobolev spaces setting. Second, although the standard integral formula for ODEs leads to an almost equivalent approach with the Duhamel formula \eqref{eq:maplocal} when  deriving the estimates on the distance between the  solutions, it turns out that the latter approach results in a subtle difference which makes it more convenient in the continuous limit (see Remark \ref{Dbest}). The same observation holds when considering the derivation of estimate \eqref{AES} on the asymptotic equivalence of solutions (see Remark \ref{DBest2}). Third, in the framework of the numerical analysis for semi-discrete schemes, the choice of the Duhamel formula in the Fourier space  highlights the analogies with spectral methods for NLS-type partial differential equations.  Adopting this framework can motivate further studies for ensuring consistency and stability when considering the continuous limit, as it allows for the study of frequency modes, potentially leading to information that is often less transparent in standard integral formulations.
\end{remark}

We conclude this subsection with a conservation law for the modified gDNLS equation \eqref{mgdnls}, for a functional that involves the $\ell^2$ norm of the solution and the background $\zeta\in X^1(\mathbb{Z})$. 

\begin{proposition}
\label{consl2}
Suppose $\Phi(0)= \left(\Phi_n(0)\right)_{n\in\mathbb Z} \in \ell^2$ and let $\Phi(t)\in C([0,T],\ell^2)$ be the unique solution of the Cauchy problem \eqref{mgdnls}-\eqref{vbcn}, as guaranteed by Theorem \ref{lwp-dnls-t}. Then, the functional 
\begin{equation}\label{modl2}
\mathcal{P}[\Phi(t)] := \frac{1}{2}\left\|\Phi(t)\right\|^2_{\ell^2}+\textnormal{Re}\sum_{n\in\mathbb Z}\Phi_n(t)\overline{\zeta_n},
\end{equation}
is conserved, namely
\begin{equation}\label{consl3}
\mathcal{P}[\Phi(t)]=\mathcal{P}[\Phi(0)], \quad t\in [0,T].	
\end{equation}
\end{proposition}

\begin{proof}
Multiplying equation \eqref{mgdnls} by $\overline{\Phi_n+\zeta_n}$, summing over $\mathbb{Z}$ and taking the imaginary part of the resulting expression, we have
\begin{equation}\label{Consd1}
\begin{aligned}
\text{Re}\sum_{n \in \mathbb Z} \frac{d \Phi_n}{dt} \, \overline{\left(\Phi_n+\zeta_n\right)} 
& +\kappa \, \text{Im}\sum_{n \in \mathbb Z}\left( \Delta( \Phi + \zeta)\right)_n \overline{(\Phi_n+\zeta_n)}
\\
&+\gamma\sum_{n\in\mathbb Z} \text{Im}\left[F(|\Phi_n+\zeta_n|^2) - F(q_0^2)\right] \left(\Phi_n+\zeta_n\right)\overline{\left(\Phi_n+\zeta_n\right)}=0.
\end{aligned}
\end{equation}
Since $\sum_{n\in\mathbb Z}\left(\Delta U\right)_n\overline{U_n} 
= 
\sum_{n\in\mathbb Z}
\left(U_{n+1}+U_{n-1}-2U_n\right) \overline{U_n}
= -\sum_{n\in\mathbb Z}|U_{n+1}-U_n|^2$, the second term vanishes. Moreover, the third term is also zero since $F$ is real-valued. Thus, \eqref{Consd1} takes the form
$$
\frac{1}{2}\frac{d}{dt}\left\|\Phi(t)\right\|^2_{\ell^2}
+ \frac{d}{dt}\text{Re}\sum_{n\in\mathbb Z}\Phi_n(t)\overline{\zeta_n}
= 0,
$$
which upon integration with respect to $t$ yields the claimed conservation law \eqref{consl3}.
\end{proof}

\begin{remark}
 Despite its simplicity, the conservation law \eqref{consl3} does not imply global existence of the local solutions of Theorem \ref{lwp-dnls-t} for the modified gDNLS Cauchy problem~\eqref{mgdnls}-\eqref{vbcn}. This limitation arises from the fact that $\zeta \in X^1(\mathbb{Z})$ (instead of $\ell^2$) and from the indefinite sign of the term $\textnormal{Re}\sum_{n\in \mathbb{Z}} \Phi_n(t)\overline{\zeta_n}$.
\end{remark}

\subsection{Local well-posedness of gAL} 

We present two variants of this result. The first version, given in Theorem~\ref{lwp-gal-t}, assumes that the background $\zeta$ belongs to $X^1(\mathbb{Z})$. Note that this assumption guarantees that $\zeta$ also belongs to the space $X^2(\mathbb{Z})$, defined by
\begin{equation}\label{X2-def}
X^2(\mathbb{Z}) := \left\{ \zeta \in \ell^\infty(\mathbb{Z}) : \zeta' \in \ell^2(\mathbb{Z}) \text{ and } \zeta'' = (\zeta_n'')_{n \in \mathbb{Z}} := (\zeta_n' - \zeta_{n-1}')_{n \in \mathbb{Z}} \in \ell^2(\mathbb{Z}) \right\},
\end{equation}
where $\zeta''$ denotes the discrete second derivative of $\zeta$. Indeed, we have the continuous embedding $X^1(\mathbb{Z}) \subset X^2(\mathbb{Z})$, since $\no{\zeta''}_{\ell^2} \leq 2\no{\zeta'}_{\ell^2}$. It is worth noting that, as with $\ell^p$ sequence spaces and in contrast to their continuous counterparts $H^s(\mathbb{R})$, the ordering of the exponents in the embedding is reversed. Let us recall that the Fourier space definition of the Sobolev spaces $H^s(\mathbb{R})$  in the continuous setting utilizes the norm 
\begin{equation*}
\no{f}_{H^s(\mathbb{R})} = \left( \int_{\mathbb{R}} (1 + \xi^2)^s |\mathcal F\{f\}(\xi)|^2 \, d\xi \right)^{\frac{1}{2}}.
\end{equation*}
Evidently, due to the presence of the weight $\left(1 + \xi^2\right)^s$, we have $H^s(\mathbb{R}) \subset H^r(\mathbb{R})$ for $s > r$. This justifies why, in the continuous setup, we have $X^p(\mathbb{R}) \subset X^q(\mathbb{R})$ when $p > q$, whereas in the discrete setup, the boundedness of the discrete operators leads to the inverse relationship.
Under the primary assumption $\zeta \in X^1(\mathbb{Z})$, the second variant of our local well-posedness result, given in Theorem~\ref{lwp-gal-t-2}, involves the norm $\no{\zeta''}_{\ell^2}$. This version is particularly useful for handling the case of large data when studying the minimal guaranteed lifespan of the gAL solutions in  Section \ref{lifespan-s}.

We begin with the proof of the first variant of the local well-posedness theory.
Similarly to the case of the gDNLS equation, the proof relies on the following lemma, which concerns the nonlinear operators arising in the gAL equation~\eqref{mgal} and provides the analogue of Lemma \ref{lip-l}.  

\begin{lemma}\label{lip-l-al}
Let $\zeta\in X^1(\mathbb{Z})$. For any $p \geq 1$, the operator 
\begin{equation}\label{Gal-def}
\phi \mapsto \mathcal G(\phi) := |\phi_n+\zeta_n|^{2p} \left(\phi_{n+1}+\phi_{n-1}+\zeta_{n+1}+\zeta_{n-1}\right) - 2q_0^{2p} \left(\phi_n+\zeta_n\right)
\end{equation}
satisfies the inequalities
\begin{align}
&\no{\mathcal G(\phi)}_{\ell^2}
\leq
16 p \left(\no{\phi}_{\ell^\infty}+\no{\zeta}_{\ell^\infty}+q_0\right)^{2p} \left(\no{\phi}_{\ell^2} + \no{|\zeta|-q_0}_{\ell^2}\right)
+
8 q_0^{2p} \no{\phi}_{\ell^2} + 4 q_0^{2p} \no{\zeta'}_{\ell^2},
\label{into-ineq-al}
\\
&\no{\mathcal G(\phi) - \mathcal G(\psi)}_{\ell^2}
\leq
2 
\left[\sqrt 2 \, q_0^{2p} 
+
2\left(2p + 1\right)
\left(\no{\phi}_{\ell^\infty} + \no{\psi}_{\ell^\infty} + 2\no{\zeta}_{\ell^\infty}
\right)^{2p}
\right]
\no{\phi- \psi}_{\ell^2}.
\label{contr-ineq-al}
\end{align}
\end{lemma}

\begin{proof}
Concerning inequality \eqref{into-ineq-al}, writing
\begin{align}
\mathcal G(\phi) 
&= 
|\phi_n+\zeta_n|^{2p} \left(\phi_{n+1}+\phi_{n-1}+\zeta_{n+1}+\zeta_{n-1}\right)
-
q_0^{2p} \left(\phi_{n+1}+\phi_{n-1}+\zeta_{n+1}+\zeta_{n-1}\right)
\nn\\
&\quad
+
q_0^{2p} \left(\phi_{n+1}+\phi_{n-1}+\zeta_{n+1}+\zeta_{n-1}\right)
 - 
 2q_0^{2p} \left(\phi_n+\zeta_n\right)
\nn\\
&= 
\big(|\phi_n+\zeta_n|^{2p} - q_0^{2p}\big) \left(\phi_{n+1}+\phi_{n-1}+\zeta_{n+1}+\zeta_{n-1}\right)
+
q_0^{2p} \left(\phi_{n+1}+\phi_{n-1}+\zeta_{n+1}+\zeta_{n-1} - 2\phi_n - 2\zeta_n\right)
\nn\\
&= 
\big(|\phi_n+\zeta_n|^{2p} - q_0^{2p}\big) \left(\phi_{n+1}+\phi_{n-1}+\zeta_{n+1}+\zeta_{n-1}\right)
+
q_0^{2p} \left(\phi_{n+1}+\phi_{n-1} - 2\phi_n +\zeta_{n+1}' - \zeta_n'\right),
\nn
\end{align}
and using the triangle inequality and then the inequality $(a+b)^2 \leq 2 (a^2 + b^2)$, we have
\begin{align}
\no{\mathcal G(\phi)}_{\ell^2}^2
&\leq
\sum_{n\in\mathbb Z} 
2\left|\phi_{n+1}+\phi_{n-1}+\zeta_{n+1}+\zeta_{n-1}\right|^2 
\big||\phi_n+\zeta_n|^{2p} - q_0^{2p}\big|^2
\nn\\
&\quad
+
2 q_0^{4p} 
\sum_{n\in\mathbb Z} 
\left|\phi_{n+1}+\phi_{n-1} - 2\phi_n +\zeta_{n+1}' - \zeta_n'\right|^2.
\label{gal-l2-0}
\end{align}
In order to handle the first sum, we note that, by the Mean Value Theorem, for any $0 \leq a \leq b$ and $p \geq 1$ we have $a^p-b^p = p c^{p-1} \left(a-b\right)$ for some $c\in (a, b)$. Hence, $a^p-b^p \leq p \, b^{p-1} \left(a-b\right)$. Similarly, if $0 \leq b\leq a$ then $b^p-a^p \leq p \, a^{p-1} \left(b-a\right)$. Hence, for any $a, b \geq 0$ and $p\geq 1$ we have $\left|a^p-b^p\right| \leq p \, \left(a^{p-1} + b^{p-1}\right) \left|a-b\right|
\leq p \, \left((a+b)^{p-1} + (a+b)^{p-1}\right) \left|a-b\right|
\leq 2 p \, \left(a + b \right)^{p-1} \left|a-b\right|$. Using this inequality with $a = |\phi_n+\zeta_n|^2$ and $b=q_0^2$, we obtain
\begin{equation}\label{pdiff-ineq0}
\big||\phi_n+\zeta_n|^{2p} - q_0^{2p}\big|
\leq
2 p \left(|\phi_n+\zeta_n|^2 + q_0^2\right)^{p-1} 
\left| |\phi_n+\zeta_n|^2 - q_0^2\right|,
\quad p \geq 1.
\end{equation}
Furthermore, since
$$
\left| |\phi_n+\zeta_n|^2 - q_0^2\right|
\leq
\left\{
\def\arraystretch{1.2}
\begin{array}{ll}
\left(|\phi_n|+|\zeta_n|\right)^2 - q_0^2, &|\phi_n+\zeta_n|\geq q_0,
\\
q_0^2 - \left| |\phi_n|-|\zeta_n|\right|^2, &|\phi_n+\zeta_n|\leq q_0,
\end{array}
\right.
$$
it follows that
$$
\left||\phi_n+\zeta_n|^2 - q_0^2\right|
\leq
\left|\left(|\phi_n|+|\zeta_n|\right)^2 - q_0^2\right|
+
\left|\left(|\phi_n|-|\zeta_n|\right)^2 - q_0^2\right|
\leq
2\left(|\phi_n|+|\zeta_n| + q_0\right) \left(|\phi_n|+\big| |\zeta_n| - q_0 \big|\right).
$$
Therefore, \eqref{pdiff-ineq0} becomes
\begin{equation}\label{pdiff-ineq}
\big||\phi_n+\zeta_n|^{2p} - q_0^{2p}\big|
\leq
4 p \left(|\phi_n|+|\zeta_n| + q_0\right)^{2p-1} \left(|\phi_n|+\big| |\zeta_n| - q_0 \big|\right),
\quad p \geq 1.
\end{equation}
In addition, we have
\begin{align}
\left|\phi_{n+1}+\phi_{n-1} - 2\phi_n +\zeta_{n+1}' - \zeta_n'\right|^2
&\leq
2 \left|\phi_{n+1}+\phi_{n-1} - 2\phi_n\right|^2
+
2\left|\zeta_{n+1}' - \zeta_n'\right|^2
\nn\\
&\leq
4 \left(2 |\phi_{n+1}|^2 + 2|\phi_{n-1}|^2 + 4\left|\phi_n\right|^2\right)
+
4 \left(\left|\zeta_{n+1}' \right|^2 + \left|\zeta_n'\right|^2\right).
\label{ab2-temp}
\end{align}
In view of \eqref{pdiff-ineq} and \eqref{ab2-temp}, \eqref{gal-l2-0} becomes
\begin{align*}
\no{\mathcal G(\phi)}_{\ell^2}^2
&\leq
64 p^2
\sum_{n\in\mathbb Z} 
\left|\phi_{n+1}+\phi_{n-1}+\zeta_{n+1}+\zeta_{n-1}\right|^2 
 \left(|\phi_n|+|\zeta_n| + q_0\right)^{2(2p-1)} \left(|\phi_n|^2+\big| |\zeta_n| - q_0 \big|^2\right)
\nn\\
&\quad
+
8 q_0^{4p} 
\sum_{n\in\mathbb Z} 
\left(2 |\phi_{n+1}|^2 + 2|\phi_{n-1}|^2 + 4\left|\phi_n\right|^2\right)
+
8 q_0^{4p} 
\sum_{n\in\mathbb Z} 
\left(\left|\zeta_{n+1}' \right|^2 + \left|\zeta_n'\right|^2\right)
\nn\\
&\leq
64 p^2 \left(2\no{\phi}_{\ell^\infty} + 2\no{\zeta}_{\ell^\infty}\right)^2 \left(\no{\phi}_{\ell^\infty}+\no{\zeta}_{\ell^\infty}+q_0\right)^{2(2p-1)} \left(\no{\phi}_{\ell^2}^2 + \no{|\zeta|-q_0}_{\ell^2}^2\right)
\nn\\
&\quad
+
64 q_0^{4p} \no{\phi}_{\ell^2}^2 + 16 q_0^{4p} \no{\zeta'}_{\ell^2}^2
\nn
\end{align*}
which can be rearranged to \eqref{into-ineq-al}.

Concerning the Lipschitz inequality \eqref{contr-ineq-al},  writing
\begin{align*}
\mathcal G(\phi) - \mathcal G(\psi) 
&= 
|\phi_n+\zeta_n|^{2p} \left(\phi_{n+1}+\phi_{n-1}+\zeta_{n+1}+\zeta_{n-1}\right) 
-
|\psi_n+\zeta_n|^{2p} \left(\psi_{n+1}+\psi_{n-1}+\zeta_{n+1}+\zeta_{n-1}\right) 
\nn\\
&\quad
- 
2q_0^{2p} \left(\phi_n+\zeta_n\right)
+
2q_0^{2p} \left(\psi_n+\zeta_n\right)
\nn\\
&= 
|\phi_n+\zeta_n|^{2p} \left(\phi_{n+1}+\phi_{n-1}\right)
-
|\psi_n+\zeta_n|^{2p} \left(\psi_{n+1}+\psi_{n-1}\right)
\nn\\
&\quad
+ 
\left(\zeta_{n+1}+\zeta_{n-1}\right) 
\left(
|\phi_n+\zeta_n|^{2p}
-
|\psi_n+\zeta_n|^{2p} \right) 
- 
2q_0^{2p} \left(\phi_n - \psi_n\right),
\end{align*}
we have
\begin{align*}
\no{\mathcal G(\phi) - \mathcal G(\psi)}_{\ell^2}^2
&\leq
2\big(2q_0^{2p}\big)^2 \no{\phi-\psi}_{\ell^2}^2
+
4\sum_{n\in\mathbb Z}
\left|
|\phi_n+\zeta_n|^{2p} \left(\phi_{n+1}+\phi_{n-1}\right)
-
|\psi_n+\zeta_n|^{2p} \left(\psi_{n+1}+\psi_{n-1}\right)
\right|^2
\nn\\
&\quad
+
4\sum_{n\in\mathbb Z}
\left|\zeta_{n+1}+\zeta_{n-1}\right|^2
\left|
|\phi_n+\zeta_n|^{2p} - |\psi_n+\zeta_n|^{2p}
\right|^2.
\end{align*}
Thus, combining the writing
\begin{align*}
&\quad
2\left[|\phi_n+\zeta_n|^{2p} \left(\phi_{n+1}+\phi_{n-1}\right)
-
|\psi_n+\zeta_n|^{2p} \left(\psi_{n+1}+\psi_{n-1}\right)
\right]
\nn\\
&=
\left(\phi_{n+1}+\phi_{n-1}\right)
\left(|\phi_n+\zeta_n|^{2p} 
-
|\psi_n+\zeta_n|^{2p}
\right)
+
|\psi_n+\zeta_n|^{2p} 
\left(\phi_{n+1}+\phi_{n-1}
-
\psi_{n+1}-\psi_{n-1}\right)
\nn\\
&\quad
+
|\phi_n+\zeta_n|^{2p} \left(\phi_{n+1}+\phi_{n-1} - \psi_{n+1} - \psi_{n-1}\right)
+
\left(\psi_{n+1}+\psi_{n-1}\right)
\left(|\phi_n+\zeta_n|^{2p} 
-
|\psi_n+\zeta_n|^{2p}
\right)
\end{align*}
with the inequality $(a+b+c+d)^2 \leq 4\left(a^2+b^2 + c^2+d^2\right)$ and the fact that, via the triangle inequality,
$\big|
\left|a+c\right| 
-
\left|b+c\right|
\big|
\leq
\left|
\left(a+c\right)
-
\left(b+c\right)
\right|
=
\left|a-b\right|$, 
we obtain
\begin{align*}
&
\no{\mathcal G(\phi) - \mathcal G(\psi)}_{\ell^2}^2
\leq
2\big(2q_0^{2p}\big)^2 \no{\phi-\psi}_{\ell^2}^2
+
4 \sum_{n\in\mathbb Z}
\left|\phi_{n+1}+\phi_{n-1}\right|^2
\left||\phi_n+\zeta_n|^{2p}
-
|\psi_n+\zeta_n|^{2p}
\right|^2
\nn\\
&
+
4\sum_{n\in\mathbb Z}
|\psi_n+\zeta_n|^{4p} 
\left|\phi_{n+1} - \psi_{n+1}
+
\phi_{n-1} - \psi_{n-1}
\right|^2
+
4\sum_{n\in\mathbb Z}
|\phi_n+\zeta_n|^{4p} 
\left|\phi_{n+1} - \psi_{n+1}
+
\phi_{n-1} - \psi_{n-1}\right|^2
\nn\\
&
+
4 \sum_{n\in\mathbb Z}
\left|\psi_{n+1}+\psi_{n-1}\right|^2
\left|
|\phi_n+\zeta_n|^{2p}
-
|\psi_n+\zeta_n|^{2p}
\right|^2
+
4\sum_{n\in\mathbb Z}
\left|\zeta_{n+1}+\zeta_{n-1}\right|^2
\left|
|\phi_n+\zeta_n|^{2p}
-
|\psi_n+\zeta_n|^{2p}
\right|^2.
\end{align*}
Hence, employing inequality \eqref{pdiff-ineq0} with $\psi_n+\zeta_n$ in place of $q_0$, namely 
$$
\left||\phi_n+\zeta_n|^{2p} - |\psi_n+\zeta_n|^{2p}\right|
\leq
2 p  \left(|\phi_n+\zeta_n|^2 + |\psi_n+\zeta_n|^2 \right)^{p-1} \left||\phi_n+\zeta_n|^2 - |\psi_n+\zeta_n|^2\right|, 
\quad p\geq 1, 
$$
and then using the triangle inequality to write 
$$
\left||\phi_n+\zeta_n|^2 - |\psi_n+\zeta_n|^2\right| 
\leq
\left(|\phi_n|+ |\psi_n|+ 2|\zeta_n|\right) \left|\phi_n-\psi_n\right|,
$$
we obtain
\begin{align*}
\no{\mathcal G(\phi) - \mathcal G(\psi)}_{\ell^2}^2
&\leq
2\big(2q_0^{2p}\big)^2 \no{\phi-\psi}_{\ell^2}^2
\nn\\
&\quad
+
4 \left(2\no{\phi}_{\ell^\infty}\right)^2
\cdot 4p^2 
\left(\no{\phi}_{\ell^\infty} +
\no{\psi}_{\ell^\infty} + 2\no{\zeta}_{\ell^\infty}\right)^{4(p-1)+2}
\no{\phi- \psi}_{\ell^2}^2
\nn\\
&\quad
+
4 \left(\no{\psi}_{\ell^\infty} + \no{\zeta}_{\ell^\infty}\right)^{4p}
\cdot 2\left(2\no{\phi-\psi}_{\ell^2}^2\right)
+
4 \left(\no{\phi}_{\ell^\infty} + \no{\zeta}_{\ell^\infty}\right)^{4p}
\cdot 2\left(2\no{\phi-\psi}_{\ell^2}^2\right)
\nn\\
&\quad
+
4 \left(2\no{\psi}_{\ell^\infty}\right)^2
\cdot 4p^2 
\left(\no{\phi}_{\ell^\infty} +
\no{\psi}_{\ell^\infty} + 2\no{\zeta}_{\ell^\infty}\right)^{4(p-1)+2}
\no{\phi- \psi}_{\ell^2}^2
\nn\\
&\quad
+
4 \left(2\no{\zeta}_{\ell^\infty}\right)^2 
\cdot 4p^2 
\left(\no{\phi}_{\ell^\infty} +
\no{\psi}_{\ell^\infty} + 2\no{\zeta}_{\ell^\infty}\right)^{4(p-1)+2}
\no{\phi- \psi}_{\ell^2}^2.
\end{align*}
Thus, in view of the inequality $\sqrt{a+b} \leq \sqrt a + \sqrt b$, we find
\begin{align*}
\no{\mathcal G(\phi) - \mathcal G(\psi)}_{\ell^2}
&\leq
2\sqrt 2 q_0^{2p}  \no{\phi-\psi}_{\ell^2}
+
8 p  \no{\phi}_{\ell^\infty}
\left(\no{\phi}_{\ell^\infty} + \no{\psi}_{\ell^\infty} + 2\no{\zeta}_{\ell^\infty}
\right)^{2p-1}
\no{\phi- \psi}_{\ell^2}
\nn\\
&\quad
+
4 \left(\no{\psi}_{\ell^\infty} + \no{\zeta}_{\ell^\infty}\right)^{2p}
\no{\phi-\psi}_{\ell^2}
+
4 \left(\no{\phi}_{\ell^\infty} + \no{\zeta}_{\ell^\infty}\right)^{2p}
\no{\phi-\psi}_{\ell^2}
\nn\\
&\quad
+
8p \no{\psi}_{\ell^\infty}
\left(\no{\phi}_{\ell^\infty} + \no{\psi}_{\ell^\infty} + 2\no{\zeta}_{\ell^\infty}
\right)^{2p-1}
\no{\phi- \psi}_{\ell^2}
\nn\\
&\quad
+
8p \no{\zeta}_{\ell^\infty} 
\left(\no{\phi}_{\ell^\infty} + \no{\psi}_{\ell^\infty} + 2\no{\zeta}_{\ell^\infty}
\right)^{2p-1}
\no{\phi-\psi}_{\ell^2},
\end{align*}
which can be rearranged to the desired Lipschitz inequality \eqref{contr-ineq-al}. 
\end{proof}

Proceeding analogously to modified gDNLS,  we combine Lemma \ref{lip-l-al} with a contraction mapping argument in order to establish the following Hadamard well-posedness result for the modified gAL equation in the case of background $\zeta \in X^1(\mathbb Z)$. 
\begin{theorem}[Local well-posedness of modified gAL] \label{lwp-gal-t}
Let $\zeta\in X^1(\mathbb{Z})$. For any $p\geq 1$, the modified gAL equation  \eqref{mgal} with initial condition $\phi(0)= (\phi_n(0))_{n \in \mathbb Z} \in \ell^2$ specified by \eqref{mod-ic} and zero boundary conditions at infinity as described by~\eqref{vbcn} possesses a unique solution $\phi \in \overline{B(0, \rho)} \subset C([0,T];\ell^2)$, where $B(0, \rho)$ denotes the open ball centered at the origin and with radius given by
\begin{equation}\label{rho-al}
\begin{aligned}
\rho = 
\rho(T)
=
2\Big[
&\no{\phi(0)}_{\ell^2} 
+ 
\sqrt{2\kappa} \no{\zeta'}_{\ell^2} \sqrt T
\\
&
+
|\mu|
\left(
2^{2(p+2)} p \left(\no{\zeta}_{\ell^\infty} + q_0 + \no{|\zeta|-q_0}_{\ell^2}\right)^{2p+1}
+
2 q_0^{2p} \no{\zeta'}_{\ell^2}
\right)
T\Big],
\end{aligned}
\end{equation}
and the lifespan $T>0$ satisfies
\begin{equation}\label{contr-cond-al}
2^{\frac 72} |\mu|
\left[q_0^{2p} 
+
2^{2p+\frac 12} \left(2p + 1\right)
\left(\rho(T) + \no{\zeta}_{\ell^\infty}
\right)^{2p}
\right] T \leq 1.
\end{equation}
Furthermore, the solution depends continuously on the initial data.
\end{theorem}

\begin{remark}\label{mgl-r-gal}
Similarly to Remark \ref{mgl-r}, the actual lifespan of the gAL solution emerging from  Theorem~\ref{lwp-gal-t} may be larger than the one satisfying the condition \eqref{contr-cond-al}. This is why we refer to $T$ in \eqref{contr-cond-al} as the \textbf{minimum guaranteed lifespan} of the solution.
\end{remark}

\begin{proof}
Taking the finite Fourier transform  \eqref{dft-def} of the modified gAL equation \eqref{mgal} and integrating with respect to $t$, we have
\begin{align}
\what \phi(\xi, t) 
&=
e^{-4i\kappa \sin^2 \left(\frac{\xi}{2} \right)t} \, \what{\phi(0)}(\xi) 
+
i e^{i\frac \xi 2} \,  \frac{1-e^{-4i\kappa \sin^2 \left(\frac{\xi}{2} \right)t}}{2 \sin \left(\frac{\xi}{2} \right)} \, \what{\zeta'}(\xi)
\nn\\
&\quad
- i\int_0^t e^{-4i\kappa \sin^2 \left(\frac{\xi}{2}\right)(t-\tau)}\left[\kappa \big(e^{i\xi}-1\big)\what{\zeta'}(\xi)+
\frac \mu 2  \, \what{\mathcal G(\phi)}(\xi, \tau)\right] d\tau.
\end{align}
Hence, inverting via \eqref{dft-inv}, we obtain the  integral equation 
\begin{equation}\label{eq:maplocal-al}
 \phi(t)= \Lambda [\phi](t),
\end{equation}
with 
\begin{equation}
\begin{aligned}
(\Lambda[\phi])_n(t) &:= \frac{1}{2\pi}\int_0^{2\pi} e^{i\xi n} \bigg[e^{-4i\kappa \sin^2 \left(\frac{\xi}{2} \right)t} \, \what{\phi(0)}(\xi) + i e^{i\frac \xi 2} \,  \frac{1-e^{-4i\kappa \sin^2 \left(\frac{\xi}{2} \right)t}}{2 \sin \left(\frac{\xi}{2} \right)} \, \what{\zeta'}(\xi) \bigg] d\xi 
\\
&\quad
- \frac{i\mu}{2} \int_0^t \frac{1}{2\pi}\int_0^{2\pi} e^{i\xi n-4i\kappa \sin^2 \left(\frac{\xi}{2}\right)(t-\tau)} 
  \what{\mathcal G(\phi)}(\xi, \tau)  d\xi d\tau.
\end{aligned}
\end{equation}
Using the triangle inequality, Parseval's theorem and Minkowski's integral inequality, we have
\begin{align}\label{D-L20-al}
\left\| \Lambda[\phi](t) \right\|_{\ell^2}
&\leq
\no{\phi(0)}_{\ell^2}
+
\frac{1}{\sqrt{2\pi}} \bigg\| \frac{1-e^{-4i\kappa \sin^2 \left(\frac{\xi}{2} \right)t}}{2 \sin \left(\frac{\xi}{2} \right)} \, \what{\zeta'}(\xi) \bigg\|_{L^2(0, 2\pi)}  
+
\frac{|\mu|}{2} \int_0^t  \big\| \mathcal G(\phi)(\tau) \big\|_{\ell^2} d\tau. 
\end{align}
The second term in \eqref{D-L20-al} was estimated in \eqref{eta-est}. 
Moreover, combining inequality \eqref{into-ineq-al} with the embedding $\ell^2\subset \ell^\infty$, we have 
\begin{align}
 \big\| \mathcal G(\phi) \big\|_{\ell^2} 
 &\leq 
16 p \left(\no{\phi}_{\ell^2}+\no{\zeta}_{\ell^\infty}+q_0\right)^{2p} \left(\no{\phi}_{\ell^2} + \no{|\zeta|-q_0}_{\ell^2}\right)
+
8 q_0^{2p} \no{\phi}_{\ell^2} + 4 q_0^{2p} \no{\zeta'}_{\ell^2}
\nn\\
&\leq
16 p \left(\no{\phi}_{\ell^2}+\no{\zeta}_{\ell^\infty}+q_0 + \no{|\zeta|-q_0}_{\ell^2}\right)^{2p+1}
+
8 q_0^{2p} \no{\phi}_{\ell^2} + 4 q_0^{2p} \no{\zeta'}_{\ell^2}
\nn\\
&\leq
16 p \left(2^{2p+1} \no{\phi}_{\ell^2}^{2p+1}+ 2^{2p+1} \left(\no{\zeta}_{\ell^\infty}+q_0 + \no{|\zeta|-q_0}_{\ell^2}\right)^{2p+1}\right)
+
8 q_0^{2p} \no{\phi}_{\ell^2} + 4 q_0^{2p} \no{\zeta'}_{\ell^2}
\nn\\
&= 
8 \left(2^{2(p+1)} p \no{\phi}_{\ell^2}^{2p+1} + q_0^{2p} \no{\phi}_{\ell^2}\right) 
+ 
2^{2p+5} p \left(\no{\zeta}_{\ell^\infty}+q_0 + \no{|\zeta|-q_0}_{\ell^2}\right)^{2p+1}
+
4 q_0^{2p} \no{\zeta'}_{\ell^2}.
\nn
\end{align}
In turn, we obtain
\begin{align}\label{eq:Lambda1-al}
\left\| \Lambda[\phi](t) \right\|_{\ell^2}
&\leq
\no{\phi(0)}_{\ell^2} + \sqrt{2\kappa} \no{\zeta'}_{\ell^2} \sqrt T
+
|\mu| 
\bigg[
4 \Big(2^{2(p+1)} p \sup_{t\in [0, T]} \no{\phi(t)}_{\ell^2}^{2p+1} + q_0^{2p} \sup_{t\in [0, T]} \no{\phi(t)}_{\ell^2}\Big) 
\nn\\
&\quad
+ 
2^{2(p+2)} p \left(\no{\zeta}_{\ell^\infty}+q_0 + \no{|\zeta|-q_0}_{\ell^2}\right)^{2p+1}
+
2 q_0^{2p} \no{\zeta'}_{\ell^2}
\bigg]
T.
\end{align}
Let $\rho=\rho(T)$ be defined by \eqref{rho-al}.
If $\phi \in \overline{B(0, \rho)}$ then, in view of inequality \eqref{eq:Lambda1-al}, for $\Lambda[\phi] \in \overline{B(0, \rho)}$ it suffices for $T>0$ to satisfy
$
\frac{\rho(T)}{2}+ 4|\mu| \big[2^{2(p+1)} p \, \rho(T)^{2p+1} + q_0^{2p} \rho(T)\big] T \leq \rho(T),
$
or, equivalently,
\begin{equation}\label{into-cond-al}
8|\mu| \left[2^{2(p+1)} p \, \rho(T)^{2p} + q_0^{2p} \right] T \leq 1.
\end{equation}
Furthermore, for any $\phi, \psi \in \overline{B(0, \rho)}$, the estimate \eqref{contr-ineq-al} and the embedding $\ell^2 \subset \ell^\infty$ yield
\begin{align}\label{ldiff-al}
 \no{ \Lambda[\phi](t)-\Lambda[\psi](t)}_{\ell^2}
 &\leq 
 \frac{|\mu|}{2} \int_0^t \no{\mathcal G(\phi(\tau)) -\mathcal G(\psi(\tau))}_{\ell^2}d\tau
 \nonumber\\
 &\leq 
|\mu|
\left[\sqrt 2 \, q_0^{2p} 
+
2^{2p+1} \left(2p + 1\right)
\left(\rho(T) + \no{\zeta}_{\ell^\infty}
\right)^{2p}
\right] T
\sup_{t\in [0, T]} \no{\phi(t) - \psi(t)}_{\ell^2}.
 \end{align}
Then, for $\phi \mapsto \Lambda[\phi]$ to be a contraction on $\overline{B(0, \rho)}$, it suffices for $T>0$ to satisfy 
\begin{equation}\label{contr-cond-al0}
2 |\mu| 
\left[\sqrt 2 \, q_0^{2p} 
+
2^{2p+1} \left(2p + 1\right)
\left(\rho(T) + \no{\zeta}_{\ell^\infty}
\right)^{2p}
\right] T \leq 1.
\end{equation}
Since $p\geq 1$, the   conditions \eqref{into-cond-al} and \eqref{contr-cond-al0}   can be combined into the stronger condition \eqref{contr-cond-al}. 
Hence, by Banach's fixed point theorem, for $T>0$ satisfying the condition~\eqref{contr-cond-al} the map $\phi \mapsto \Lambda[\phi]$ possesses a unique fixed point in $\overline{B(0, \rho)}\subset C([0,T];\ell^2)$. Continuity with respect to the initial data follows by using inequality~\eqref{contr-ineq-al} along the lines of the argument that led to \eqref{ldiff-al}.
\end{proof}

We conclude this section with the second variant of our local well-posedness theory, namely the analogue of  Theorem \ref{lwp-gal-t}  but now also involving the norm $\no{\zeta''}_{\ell^2}$ in the solution radius $\varrho$. The motivation for establishing this result stems from the lifespan study carried out in Section \ref{lifespan-s}. In particular, the fact that the solution radius $\varrho$ defined in Theorem \ref{lwp-gal-t-2} below involves the $\ell^2$ norm of $\zeta'$ only once, as the second occurrence of that norm in the original radius $\rho$ of Theorem \ref{lwp-gal-t} is replaced in $\varrho$ by the $\ell^2$ norm of  $\zeta''$,  this will be particularly useful for establishing Theorem~\ref{al-life-t} below in the case $\varepsilon > 1$, specifically regarding the minimal guaranteed lifespan of the solutions.
\begin{theorem}[Local well-posedness of gAL revisited] \label{lwp-gal-t-2}
Let $\zeta\in X^1(\mathbb Z)$. For any $p\geq 1$, the Cauchy problem  \eqref{mgal}, \eqref{mod-ic}, \eqref{vbcn} for the modified gAL equation possesses a unique solution $\phi \in \overline{B(0, \varrho)} \subset C([0,T];\ell^2)$, where $B(0, \varrho)$ denotes the open ball centered at the origin and with radius given by
\begin{equation}\label{rho-al-2}
\begin{aligned}
\varrho = \varrho(T)
=
2\Big[
&\no{\phi(0)}_{\ell^2} 
+ 
\sqrt{2\kappa} \no{\zeta'}_{\ell^2} \sqrt T
\\
&+
|\mu| 
\left(
2^{2(p+2)} p \left(\no{\zeta}_{\ell^\infty} + q_0 + \no{|\zeta|-q_0}_{\ell^2}\right)^{2p+1}
+
q_0^{2p} \no{\zeta''}_{\ell^2}
\right)
T\Big],
\end{aligned}
\end{equation}
and the  lifespan $T>0$ satisfies
\begin{equation}\label{contr-cond-al-2}
2^{\frac 72} |\mu| 
\left[q_0^{2p} 
+
2^{2p+\frac 12} \left(2p + 1\right)
\left(\varrho(T) + \no{\zeta}_{\ell^\infty}
\right)^{2p}
\right] T \leq 1.
 \end{equation}
Furthermore, the solution depends continuously on the initial data.
\end{theorem}

\begin{proof}
Returning to the proof of Lemma \ref{lip-l-al} and, more specifically, to the second term of the first inequality in \eqref{ab2-temp}, we take advantage of the fact that  $\zeta\in X^1(\mathbb Z) \subset X^2(\mathbb Z)$ in order to replace \eqref{ab2-temp} with the bound
$$
\left|\phi_{n+1}+\phi_{n-1} - 2\phi_n +\zeta_{n+1}' - \zeta_n'\right|^2
\leq
4 \left(2 |\phi_{n+1}|^2 + 2|\phi_{n-1}|^2 + 4\left|\phi_n\right|^2\right)
+
2 \left|\zeta_{n+1}'' \right|^2.
$$
In turn, \eqref{into-ineq-al} becomes
\begin{equation}\label{into-ineq-al-2}
\no{\mathcal G(\phi)}_{\ell^2}
\leq
16 p \left(\no{\phi}_{\ell^\infty}+\no{\zeta}_{\ell^\infty}+q_0\right)^{2p} \left(\no{\phi}_{\ell^2} + \no{|\zeta|-q_0}_{\ell^2}\right)
+
8 q_0^{2p} \no{\phi}_{\ell^2} + 2 q_0^{2p} \no{\zeta''}_{\ell^2},
\end{equation}
which allows us to carry out the proof of Theorem \ref{lwp-gal-t} from \eqref{D-L20-al} onward in order to obtain the claimed well-posedness result with the radius $\varrho$ given by \eqref{rho-al-2}.
\end{proof}

\section{Order of the minimum guaranteed solution lifespan}
\label{lifespan-s}

The local well-posedness results of Theorems \ref{lwp-dnls-t} and \ref{lwp-gal-t} established in the previous section enable a more detailed analysis of the minimum guaranteed lifespan of the solutions, allowing us to determine its order of magnitude in relation to the size of the initial data and the background.
For instance, a main result of the present section is a direct consequence of Theorem \ref{lwp-dnls-t} and establishes that, in the case of initial data of $\mathcal O(\varepsilon)$, the lifespan of the local solution of Theorem \ref{lwp-dnls-t} is of $\mathcal O(\varepsilon^{-2p})$ \textit{if and only if} the solution is of $\mathcal O(\varepsilon)$. 
We emphasize that $\varepsilon$ is \textit{not necessarily small}, so the result can be interpreted in two ways: small data imply large lifespan,  while large data result in small guaranteed lifespan. 

This section is organized as follows. First, we establish the above-described result for the  modified gDNLS equation \eqref{mgdnls}. Moreover, by employing the conservation law of Proposition \ref{consl2}, we show that for the \textit{finite-dimensional} modified gDNLS lattice arising from either Dirichlet or periodic boundary conditions, solutions exist globally  regardless of the size of the initial data or the exponent $p \geq 1$. This result is especially relevant for the numerical simulations of Section \ref{sim-s}, which rely on finite-dimensional approximations of the problem on a finite lattice with the aforementioned boundary conditions.
Then, we establish the corresponding result for the modified gAL equation \eqref{mgal}. 

Our first result concerns the lifespan of modified gDNLS on the infinite lattice and reads as follows.
\begin{theorem}\label{dnls-life-t}
Given $p\geq 1$ and $\varepsilon>0$, consider the Cauchy problem \eqref{mgdnls}-\eqref{vbcn} for the modified gDNLS equation with initial data $\Phi(0) \in \ell^2$ such that
\begin{equation}\label{dnls-ic}
 \left\| \Phi(0) \right\|_{\ell^2}=  A_0  \varepsilon, \quad A_0>0, 
\end{equation}
and a nonzero background described by $\zeta \in X^1(\mathbb Z)$ satisfying \eqref{zhi2} and such that
\begin{equation}\label{dnls-back}
q_0 = B\varepsilon, 
\quad 
\no{\zeta}_{\ell^\infty}=  B_0\varepsilon,
\quad 
\no{\zeta'}_{\ell^2}= B_1 \varepsilon^{p+1},
\quad 
\big\| |\zeta| - q_0 \big\|_{\ell^2}= B_2\varepsilon,
\end{equation}
for some constants $B, B_0, B_1, B_2>0$. 
\\[2mm]
\textnormal{(i)} 
Suppose the minimum guaranteed solution lifespan of Theorem \ref{lwp-dnls-t} is of the form
$
T = \dfrac{C}{\varepsilon^{2p}}
$
with $C>0$ satisfying the inequality that emerges from the combination of the definition  \eqref{eq:rhoPhi} with the condition~\eqref{contr-cond}.
Then, there exists a constant $A>0$, which depends on $C$ and the constants involved in \eqref{dnls-ic}-\eqref{dnls-back}, such that the solution $\Phi \in C([0, T]; \ell^2)$ to the modified gDNLS  Cauchy problem emerging from Theorem \ref{lwp-dnls-t} admits the size estimate
$
\displaystyle \sup_{t\in [0, T]}\left\|\Phi(t)\right\|_{\ell^2} \leq A \varepsilon.
$
\\[2mm]
\textnormal{(ii)} Conversely, if the radius associated with the local solution of Theorem \ref{lwp-dnls-t} is of the form $\rho = A \ve$ with $A>0$ satisfying  the inequality emerging from the combination of ~\eqref{eq:rhoPhi} with the condition~\eqref{contr-cond}, then the minimum guaranteed lifespan of the solution is of the form 
$
T = \dfrac{C}{\varepsilon^{2p}}
$
for an appropriate constant $C>0$ that depends on $A$ and the constants involved in \eqref{dnls-ic}-\eqref{dnls-back}.
\end{theorem}

\begin{proof}
(i) Combining the assumptions \eqref{dnls-ic}-\eqref{dnls-back} and the hypothesis for $T$ with the definition \eqref{eq:rhoPhi} for the solution radius $\rho$, we infer that $\rho = A\ve$ with
\begin{equation}
A
=
2 \left[A_0 + \sqrt{2\kappa} B_1  \sqrt C  +
2^{2p+\frac 32} |\gamma| K  \left(B_0 +  B_2 + B\right)^{2p+1} C\right].
\end{equation}
Thus, if $C>0$ is such that the condition \eqref{contr-cond} is satisfied, namely
\begin{equation}\label{C-cond}
 2^{2p+\frac 52} |\gamma| K \left(A + B_0 + B\right)^{2p}  C \leq 1, 
 \end{equation}
 then the local well-posedness result of Theorem \ref{lwp-dnls-t} readily implies the desired size estimate for the solution of modified gDNLS. Note that we can always choose an appropriate $C>0$ since the left side of \eqref{C-cond} tends to zero as $C\to 0^+$.
\\[2mm]
(ii) Combining the expression \eqref{eq:rhoPhi} with the hypothesis for the solution radius $\rho$, we have
\begin{equation}\label{T-eq}
A\ve = 2 \left[A_0 \ve + \sqrt{2\kappa} B_1 \ve^{p+1}  \sqrt T  +
2^{2p+\frac 32} |\gamma| K  \left(B_0 \ve +  B_2 \ve + B \ve\right)^{2p+1}  T\right] ,
\end{equation}
which is a quadratic equation for $\sqrt T$ which can be solved provided that
\begin{equation}\label{A-discr}
2A_0 - A \leq \frac{\kappa B_1^2}{2^{2p+\frac 32} |\gamma| K  \left(B_0 +  B_2 + B\right)^{2p+1}},
\end{equation}
(so that the discriminant of \eqref{T-eq} remains non-negative, ensuring real solutions) to yield
$T = \dfrac{C}{\ve^{2p}}$
with
\begin{equation}
C =  \left(\tfrac{-\sqrt{\kappa} B_1 + \sqrt{\kappa B_1^2 - 2^{2p+\frac 32} 
|\gamma| K  \left(B_0 +  B_2 + B\right)^{2p+1} \left(2A_0-A\right)}}{2^{2p+2} |\gamma| K  \left(B_0 +  B_2 + B\right)^{2p+1}}\right)^2.
\end{equation}
Thus, according to the local well-posedness of Theorem \ref{lwp-dnls-t}, a solution with lifespan $T$ exists provided that $A>0$ satisfies the condition \eqref{contr-cond}, namely
\begin{equation}\label{A-cond}
 2^{2p+\frac 52} |\gamma| K \left(A +B_0 + B\right)^{2p}  C \leq 1.
 \end{equation}
Note that the derivation of the condition \eqref{A-cond}  implicitly assumes the validity of the earlier condition~\eqref{A-discr}, namely both conditions are necessary for $A$. 
Furthermore, note that \eqref{A-cond} can always be satisfied by choosing $A$ sufficiently close to $2A_0$, so that the left side of~\eqref{A-cond}, which is controlled by $2A_0-A$ through $C$ (which vanishes when $A=2A_0$), is small enough. Although such a choice  may not be  optimal in general, as it leads to a small lifespan constant $C$, there may be other choices of $A$ that yield a larger $C$.
\end{proof}

\begin{remark}[Size of data vs. solution lifespan]
When the constants in the assumptions \eqref{dnls-ic} and \eqref{dnls-back}  are of $\mathcal{O}(1)$, Theorem \ref{dnls-life-t} can be interpreted in two ways.  On the one hand, for $\ve<1$ (i.e. ``small''), it implies that relatively small initial data and background result in relatively large lifespan of solutions. On the other hand, for $\ve>1$ (i.e. ``large''), it shows that large initial data and background lead to a small solution lifespan.
\end{remark}

Next, we turn our attention to the modified gDNLS equation \eqref{mgdnls} over a \textit{finite lattice} and supplemented with either \textit{homogeneous} Dirichlet or periodic boundary conditions.  
For this problem, we prove the \textit{global existence} of solutions \textit{unconditionally} with respect to the data. 
We note that the study of the finite lattice problem is important both for theoretical and for practical purposes, since this is the problem  used in the numerical simulations of Section \ref{sim-s}. 

More specifically, we consider an arbitrary number of $N+1$ oscillators placed equidistantly over the interval $\Omega = [-L,L]$  of length $2L$. We denote by $\kappa=h^{-2}$ the discretization parameter, where $h=2L/N$ is the lattice spacing, so that the oscillators are located at $x_n=-L+nh$, $n=0,1,2,\ldots,N$. We then supplement equation~\eqref{mgdnls}  either with the Dirichlet boundary conditions 
\begin{equation}\label{eq02}
	\Phi_0(t) =\Phi_{N}(t) = 0, \quad t \geq 0, 
\end{equation} 
or with the periodic boundary conditions 
\begin{equation}\label{eq02per}
	\Phi_n(t) =\Phi_{n+N}(t) = 0, \quad t \geq 0. 
\end{equation} 

In the case of the homogeneous Dirichlet boundary conditions \eqref{eq02}, we shall use the  finite-dimensional subspaces of $\ell^r$ defined by
\begin{equation}\label{lr0-def}
	\ell^r_{0} := \bigg\{U=(U_n)_{n\in\mathbb{Z}}\in\mathbb{R}: \ U_0=U_{N}=0, \ 
	\no{U}_{\ell^r_{0}}:=\bigg(h\sum_{n=1}^{N-1}|U_n|^r\bigg)^{\frac{1}{r}}<\infty\bigg\}, \quad 1\leq r\leq\infty,
\end{equation}
while in the case the of the periodic boundary conditions \eqref{eq02per} we shall work in the spaces of periodic sequences with  period $N$ defined by
\begin{equation}\label{lper-def}
	{\ell}^r_\text{per} :=\bigg\{U=(U_n)_{n\in\mathbb{Z}}\in\mathbb{R}: \ U_n=U_{n+N},\ 
	\no{U}_{\ell^r_{\text{per}}}:=\bigg(h\sum_{n=0}^{N-1}|U_n|^r\bigg)^{\frac{1}{r}}<\infty\bigg\}, \quad 1\leq r\leq\infty.
\end{equation}
Let us denote both of the above finite-dimensional spaces by $\mathcal{L}^r$. Note that the norms between the spaces $\mathcal{L}^r$ and $\mathcal{L}^q$  are  equivalent in view of the inequality
\begin{equation}
	\label{equi}
\no{U}_{\mathcal{L}^q}\leq \no{U}_{\mathcal{L}^r}\leq N^{\frac{q-r}{qr}}\no{U}_{\mathcal{L}^q},\quad 1\leq r\leq q \leq \infty. 
\end{equation}

In the above finite-dimensional setup,  we have the following result.
\begin{theorem}[Global existence on a finite lattice]
 \label{globexfin} Consider the modified gDNLS  equation~\eqref{mgdnls}   supplemented with initial data  $\Phi(0)\in \mathcal{L}^2$ and either the homogeneous Dirichlet conditions \eqref{eq02} or the periodic conditions \eqref{eq02per}. Then, the corresponding finite-lattice solutions exist globally in time. In particular, $\Phi \in C^1([0,\infty),\mathcal{L}^2)$ and is uniformly bounded with 
 \begin{equation}\label{eq07ge}
\no{\Phi(t)}_{\mathcal L^2}^2\leq   2\no{\Phi(0)}_{\mathcal L^2}^2 
+ 
4 N  \no{\zeta}_{\ell^{\infty}}^2
+
4 \no{\Phi(0)}_{\mathcal L^1} \no{\zeta}_{\ell^{\infty}}, \quad t > 0. 
\end{equation}
\end{theorem}

\begin{proof}
We only give the proof for $\ell^2_0$ as the argument for $\ell^2_\text{per}$ is similar.
The conservation law of Proposition \ref{consl2} is also valid in the case of  homogeneous Dirichlet boundary conditions \eqref{eq02}, i.e.
\begin{equation}
\label{eq04ge}
\frac 1h \no{\Phi(t)}_{\ell^2_0}^2 + 2\, \mathrm{Re}\sum_{n=1}^{N-1}\Phi_n(t)\overline{\zeta_n} 
=  
\frac 1h \no{\Phi(0)}_{\ell^2_0}^2 + 2\, \mathrm{Re}\sum_{n=1}^{N-1}\Phi_n(0)\overline{\zeta_n}.
\end{equation}
Thus, employing the triangle inequality, the definition \eqref{lr0-def} and the inequality \eqref{equi} with $r=1$, $q=2$, we obtain
\begin{align}
\label{eq05ge}
\no{\Phi(t)}_{\ell^2_0}^2
&\leq \no{\Phi(0)}_{\ell^2_0}^2 
+
2 h \, \bigg|\sum_{n=1}^{N-1}\Phi_n(t)\overline{\zeta_n}\bigg| + 2 h \, \bigg|\sum_{n=1}^{N-1}\Phi_n(0)\overline{\zeta_n}\bigg|
\nonumber\\
&\leq \no{\Phi(0)}_{\ell^2_0}^2+ 
2 \no{\Phi(t)}_{\ell^1_0}\no{\zeta}_{\ell^{\infty}} + 2 \no{\Phi(0)}_{\ell^1_0}\no{\zeta}_{\ell^{\infty}}\nonumber\\
&\leq
\no{\Phi(0)}_{\ell^2_0}^2+ 2 \sqrt{N}\no{\Phi(t)}_{\ell_0^2}\no{\zeta}_{\ell^{\infty}}+2 \no{\Phi(0)}_{\ell^1_0}\no{\zeta}_{\ell^{\infty}}
\nn
\end{align}
so that by Young's inequality
\begin{equation*}
\no{\Phi(t)}_{\ell^2_0}^2
\leq
\no{\Phi(0)}_{\ell^2_0}^2 +\frac{1}{2}\no{\Phi(t)}_{\ell^2_0}^2+2N\no{\zeta}_{\ell^{\infty}}^2 +
2\no{\Phi(0)}_{\ell^1_0}\no{\zeta}_{\ell^{\infty}}
\end{equation*}
which can be rearranged to the desired estimate \eqref{eq07ge}. In particular, this estimate provides a uniform bound in $t$ which ensures the global existence of solutions for all $t>0$. 
\end{proof}

\begin{remark}[Lack of global existence for gAL on a finite lattice]
Unlike the case of the modified gDNLS lattice \eqref{mgdnls}, a  global existence result for the finite-dimensional modified gAL lattice \eqref{mgal}, valid at least for a suitable size of initial data or for a suitable range of exponents, does not appear to be attainable at present. Such a result would rely on the derivation of an appropriate conservation law for the modified gAL system. We note that, even in the case of zero boundary conditions, the gAL lattice possesses the  non-trivial conserved quantity  \cite{PGKgAL,JMPClo}
\begin{equation*}
P_{gAL}(t)=\sum_{n\in\mathbb{Z}}|U_n|^2\,_2F_1\left(1,\tfrac{1}{p},1+\tfrac{1}{p};-|U_n|^{2p}\right),
\end{equation*}
where $_2F_1(a,b,c;z)$ is the Gauss hypergeometric function. 
In light of this conserved quantity for the gAL with zero boundary conditions, the derivation of a conservation law  for the modified gAL  lattice \eqref{mgal} is of independent interest and will be pursued elsewhere.
\end{remark}

We conclude this section with the analogue of Theorem~\ref{dnls-life-t} for the modified gAL equation \eqref{mgal}. In particular, the  local well-posedness results of Theorems~\ref{lwp-gal-t} and \ref{lwp-gal-t-2} imply the following result relating the lifespan of the modified gAL solutions with the size of the associated initial data and background. 

\begin{theorem}\label{al-life-t}
Given $p\geq 1$ and $\ve > 0$, consider the modified gAL Cauchy problem \eqref{mgal}, \eqref{mod-ic}, \eqref{vbcn} in the case of initial data $\phi(0) \in \ell^2$ such that
\begin{equation}\label{al-ic}
 \left\| \phi(0) \right\|_{\ell^2}=  A_0  \varepsilon, \quad A_0>0, 
\end{equation}
and a nonzero background described by $\zeta \in X^1(\mathbb Z)$ satisfying \eqref{zhi2} and such that
\begin{equation}\label{al-back}
q_0 = B\varepsilon, 
\quad 
\no{\zeta}_{\ell^\infty}=  B_0\varepsilon,
\quad 
\no{\zeta'}_{\ell^2}= B_1 \varepsilon^{p+1},
\quad 
\big\| |\zeta| - q_0 \big\|_{\ell^2}= B_2\varepsilon
\end{equation}
for some constants $B, B_0, B_1, B_2>0$. In addition, if $\ve > 1$ then assume that
\begin{equation}\label{al-back2}
\no{\zeta''}_{\ell^2}= B_3 \varepsilon, \quad B_3>0.
\end{equation}
\textnormal{(i)} 
Suppose that the minimum guaranteed solution lifespan is of the form
$
T = \dfrac{C}{\varepsilon^{2p}}
$
with $C>0$ satisfying the inequality that emerges from the combination of either  \eqref{rho-al} with \eqref{contr-cond-al} (when $\ve \leq 1$) or \eqref{rho-al-2} with~\eqref{contr-cond-al-2} (when $\ve > 1$).
Then, there exists a constant $A>0$, which depends on $C$ and the constants involved in \eqref{al-ic}-\eqref{al-back2}, such that the solution $\phi \in C([0, T]; \ell^2)$ to the modified gAL Cauchy problem guaranteed by either Theorem \ref{lwp-gal-t} (for $\ve\leq 1$) or Theorem \ref{lwp-gal-t-2} (for $\ve>1$) admits the size estimate
$
\displaystyle \sup_{t\in [0, T]}\left\|\phi(t)\right\|_{\ell^2} \leq A \varepsilon.
$
\\[2mm]
\textnormal{(ii)} Conversely, if the radii $\rho$ and $\varrho$ associated with the local solutions of Theorems \ref{lwp-gal-t} and \ref{lwp-gal-t-2} are of the form $\rho \leq A \ve$ and $\varrho = A \ve$ with $A>0$ satisfying, respectively,  the inequalities emerging from the combination of either \eqref{rho-al} with \eqref{contr-cond-al} (when $\ve \leq 1$) or \eqref{rho-al-2} with \eqref{contr-cond-al-2} (when $\ve > 1$), then the minimum guaranteed solution lifespan is of the form 
$
T = \dfrac{C}{\varepsilon^{2p}}
$
for an appropriate constant $C>0$ that depends on $A$ and the constants involved in~\eqref{al-ic}-\eqref{al-back2}.
\end{theorem}

\begin{proof}
We begin by noting that the assumptions on the norms $\no{\zeta'}_{\ell^2}$ and $\no{\zeta''}_{\ell^2}$ in \eqref{al-back} and \eqref{al-back2} can be made simultaneously only when $\ve\geq 1$, since for $\ve<1$ they would be inconsistent with the inequality $\no{\zeta''}_{\ell^2} \leq 2 \no{\zeta'}_{\ell^2}$ (which readily follows from the definition of $\zeta''$). Therefore, for $\ve \leq 1$ we use the well-posedness result of Theorem \ref{lwp-gal-t} (which does not involve $\no{\zeta''}_{\ell^2}$) and assume \eqref{al-ic}-\eqref{al-back}, while for $\ve > 1$ we employ Theorem~\ref{lwp-gal-t-2} (which does involve $\no{\zeta''}_{\ell^2}$) and additionally assume \eqref{al-back2}. 
\\[2mm]
(i) Suppose first that $\ve \leq 1$. From \eqref{rho-al} and the hypotheses \eqref{al-ic}-\eqref{al-back}, we have
\begin{align}
\rho 
=
2 \Big[
A_0 
+ 
\sqrt{2\kappa}  B_1 \sqrt C 
+
|\mu|
\left(
2^{2(p+2)} p \left(B_0 + B + B_2\right)^{2p+1} 
+
2 B^{2p} B_1 \ve^p
\right)
C  \Big] \ve.
\end{align}
Thus, since $\ve \leq 1$, it follows that $\rho \leq A \ve$ with 
\begin{equation}
A 
=
2 \Big[
A_0 
+ 
\sqrt{2\kappa}  B_1 \sqrt C 
+
|\mu|
\left(
2^{2(p+2)} p \left(B_0 + B + B_2\right)^{2p+1} 
+
2 B^{2p} B_1 
\right)
C  \Big].
\end{equation}
Hence, if $C$ is such that the condition \eqref{contr-cond-al} is satisfied, namely if
\begin{align}\label{C-cond-gal}
&\quad
2^{\frac 72} |\mu|
\left[q_0^{2p} 
+
2^{2p+\frac 12} \left(2p + 1\right)
\left(\rho(T) + \no{\zeta}_{\ell^\infty}
\right)^{2p}
\right] T
\nn\\
&\leq
2^{\frac 72} |\mu|
\left[B^{2p} \ve^{2p} 
+
2^{2p+\frac 12} \left(2p + 1\right)
\left(A \ve + B_0 \ve \right)^{2p}
\right] C \ve^{-2p}
\nn\\
&=
2^{\frac 72} |\mu|
\left[B^{2p} 
+
2^{2p+\frac 12} \left(2p + 1\right)
\left(A + B_0 \right)^{2p}
\right] C 
 \leq 1,
 \end{align}
then the local well-posedness result of Theorem \ref{lwp-gal-t} readily implies the desired size estimate for the solution of modified gAL. Note that we can always choose an appropriate $C>0$ since the left side of~\eqref{C-cond-gal} tends to zero as $C\to 0^+$.

Next, suppose that $\ve > 1$. From \eqref{rho-al-2} and the hypotheses \eqref{al-ic}-\eqref{al-back2}, we have
$$
\varrho 
=
2\Big[
A_0 \ve
+ 
\sqrt{2\kappa} B_1 \ve^{p+1} \sqrt C \ve^{-p}
+
|\mu| 
\left(
2^{2(p+2)} p \left(B_0 \ve + B \ve +B_2 \ve\right)^{2p+1}
+
B^{2p} \ve^{2p} B_3 \ve
\right)
C \ve^{-2p}\Big]
$$
so $\varrho = A\ve$ with
\begin{equation}\label{A-def}
A =
2\Big[
A_0 
+ 
\sqrt{2\kappa} B_1  \sqrt C 
+
|\mu| 
\left(
2^{2(p+2)} p \left(B_0  + B  +B_2 \right)^{2p+1}
+
B^{2p}   B_3 
\right)
C \Big].
\end{equation}
Hence, if $C$ is such that the condition \eqref{contr-cond-al-2} is satisfied, namely if $C$ satisfies \eqref{C-cond-gal} but with $A$ given by \eqref{A-def}, then the local well-posedness result of Theorem \ref{lwp-gal-t-2} readily implies the desired size estimate for the solution of modified gAL. 
\\[2mm]
(ii) Suppose first that $\ve \leq 1$. In view of \eqref{rho-al} and the hypotheses \eqref{al-ic}-\eqref{al-back}, we have
\begin{align}
\rho
&=
2\Big[
A_0 \ve 
+ 
\sqrt{2\kappa} B_1 \ve^{p+1}  \sqrt T
+
|\mu|
\left(
2^{2(p+2)} p \left(B_0 \ve  + B \ve + B_2 \ve\right)^{2p+1}
+
2 B^{2p} \ve^{2p} B_1 \ve^{p+1}
\right)
T\Big]
\nn\\
&\leq
2\Big[
A_0  
+ 
\sqrt{2\kappa} B_1 \ve^{p}  \sqrt T
+
|\mu|
\left(
2^{2(p+2)} p \left(B_0  + B + B_2\right)^{2p+1}
+
2 B^{2p}  B_1  
\right)
\ve^{2p}
T\Big] \ve.
\end{align}
Thus, since $\rho \leq A \ve$ by hypothesis, we can set
\begin{equation}
A 
=
2\Big[
A_0  
+ 
\sqrt{2\kappa} B_1 \ve^{p}  \sqrt T
+
|\mu|
\left(
2^{2(p+2)} p \left(B_0   + B  + B_2 \right)^{2p+1} 
+
2 B^{2p}   B_1 
\right)
\ve^{2p}
T\Big].
\end{equation}
This is a quadratic equation for $\sqrt T$ which can be solved provided that
\begin{equation}\label{A-discr-2}
2A_0 - A \leq \frac{\kappa B_1^2}{|\mu|
\left(
2^{2(p+2)} p \left(B_0   + B  + B_2 \right)^{2p+1} 
+
2 B^{2p}   B_1 
\right)}
\end{equation}
to yield $T = \dfrac{C}{\ve^{2p}}$ with 
\begin{equation}
C =  \left(\tfrac{-\sqrt{\kappa} B_1 + \sqrt{\kappa B_1^2 - |\mu|
\left(
2^{2(p+2)} p \left(B_0  + B + B_2\right)^{2p+1}
+
2 B^{2p}  B_1  
\right) \left(2A_0-A\right)}}{\sqrt 2 \, |\mu|
\left(
2^{2(p+2)} p \left(B_0  + B + B_2\right)^{2p+1}
+
2 B^{2p}  B_1  
\right)}\right)^2.
\end{equation}
Thus, according to the local well-posedness of Theorem \ref{lwp-gal-t}, a solution with lifespan $T$ exists provided that $A>0$ satisfies the condition \eqref{contr-cond-al}, namely 
\begin{equation}\label{A-cond-2}
2^{\frac 72} |\mu|
\left[B^{2p} + 2^{2p+\frac 12} \left(2p + 1\right)
\left(A + B_0 \right)^{2p}
\right] C  \leq 1.
\end{equation}
Note that the derivation of the condition \eqref{A-cond-2}  implicitly assumes the validity of the earlier condition~\eqref{A-discr-2}, namely both conditions are necessary for $A$. 

Next, suppose that $\ve > 1$. Then, since $\varrho = A \ve$, by  \eqref{rho-al-2} and the hypotheses \eqref{al-ic}-\eqref{al-back2} we have
\begin{equation}
A 
=
2\Big[
A_0 
+ 
\sqrt{2\kappa} B_1 \ve^p \sqrt T
+
|\mu| 
\left(
2^{2(p+2)} p \left(B_0 + B +B_2 \right)^{2p+1} 
+
B^{2p} B_3 
\right)
\ve^{2p}
T\Big].
\end{equation}
Following the same steps as for the case $\ve \leq 1$, we conclude that 
$T = \dfrac{C}{\ve^{2p}}$ with 
\begin{equation}
C =  \left(\tfrac{-\sqrt{\kappa} B_1 + \sqrt{\kappa B_1^2 - |\mu|
\left(
2^{2(p+2)} p \left(B_0  + B + B_2\right)^{2p+1}
+
B^{2p}  B_3  
\right) \left(2A_0-A\right)}}{\sqrt 2 \, |\mu|
\left(
2^{2(p+2)} p \left(B_0  + B + B_2\right)^{2p+1}
+
B^{2p}  B_3  
\right)}\right)^2
\end{equation}
provided that $A$ satisfies \eqref{A-cond-2} and
\begin{equation}\label{A-discr-3}
2A_0 - A \leq \frac{\kappa B_1^2}{|\mu|
\left(
2^{2(p+2)} p \left(B_0   + B  + B_2 \right)^{2p+1} 
+
B^{2p}   B_3 
\right)}.
\end{equation}
Note that the derivation of the condition \eqref{A-cond-2}  implicitly assumes the validity of the condition~\eqref{A-discr-3}, namely both conditions are necessary for $A$. 
\end{proof}

\begin{remark}\label{case-iii}
In the case of zero background $\zeta\equiv 0$, Theorems \ref{dnls-life-t} and \ref{al-life-t}  still hold true.
\end{remark}

\section{Distance between the \textnormal{g}AL and \textnormal{g}DNLS solutions}
\label{prox-s}

In this section, we establish estimates for the distance between solutions of the gAL equation~\eqref{gal} and the gDNLS equation \eqref{gdnls} that will lead to the precise version of Theorem \ref{prox-t-intro}. We begin with the modified systems \eqref{mgal} and \eqref{mgdnls}, and subsequently deduce the corresponding result for the original systems after accounting for the frequencies involved in the transformations \eqref{zhi4}. The result is general, treating the systems with different nonlinearity exponents. It relies on the estimates derived in Section~\ref{lwp-s} together with the analytical characterization of the minimal guaranteed solution lifespan for each system, obtained in Section \ref{lifespan-s}. The latter result allows us to consider both systems on a common, explicitly described, minimal interval of existence.
\begin{theorem}[Distance between gAL and gDNLS]
\label{al-dnls-t}
Given $p_1, p_2\geq 1$ and $\ve > 0$, consider the modified gAL equation \eqref{mgal} with $p=p_1$ and the modified gDNLS equation \eqref{mgdnls} with $p=p_2$, supplemented with the initial data   \eqref{mod-ic} and the nonzero boundary conditions \eqref{vbcn} satisfying the assumptions~\eqref{dnls-ic}~and~\eqref{al-ic}-\eqref{al-back2}. Let $T_c = \min\left\{T_1, T_2\right\}$ where $T_1 = \dfrac{M_1}{\ve^{2p_1}}$ and $T_2 = \dfrac{M_2}{\ve^{2p_2}}$ are, respectively, the lifespans of the modified gAL and the modified gDNLS solutions with constants $M_1, M_2>0$ such that
\begin{equation}\label{Cc-cond}
2^{\frac 72} |\mu|
\left[B^{2p_1} 
+
2^{2p_1+\frac 12} \left(2p_1 + 1\right)
\left(A_1 + B_0\right)^{2p_1}
\right] M_1 \leq 1,
\quad
2^{2p_2+\frac 52} |\gamma| K 
\left(A_2 + B_0 + B\right)^{2p_2}  M_2 \leq 1,
\end{equation}
where 
\begin{equation}\label{a1a2-def}
\begin{aligned}
A_1 &=
2\left[
A_0
+ 
\sqrt{2\kappa} B_1 \sqrt{M_1}
+
|\mu|
\left(
2^{2(p_1+2)} p_1 \left(B_0 + B + B_2\right)^{2p_1+1}
+
B^{2p_1} B_4
\right)
M_1\right],
\\
A_2 &= 2 \left[A_0 + \sqrt{2\kappa} B_1  \sqrt{M_2}  +
2^{2p_2+\frac 32} |\gamma| K  \left(B_0 +  B_2 + B\right)^{2p_2+1} M_2\right],
\end{aligned}
\end{equation}
with $B_4 = 2B_1$ if $\ve \leq 1$ and $B_4 = B_3$ if $\ve>1$.
If the initial data satisfy the distance condition
\begin{equation}\label{ic-close}
\no{\phi(0)-\Phi(0)}_{\ell^2} \equiv \no{u(0)-U(0)}_{\ell^2}  \leq C_0 \max\left\{\ve^{2p_1+1}, \ve^{2p_2+1}\right\},
 \end{equation}
for some constant $C_0>0$, then for each fixed $T \in (0, T_c]$ there exists  a constant $C >0$, which depends on $T$ and all of the above constants, such that the solutions of the gAL and gDNLS equations satisfy the distance bound
\begin{equation}\label{al-dnls-close}
\sup_{t \in [0,T]} \no{\phi(t)-\Phi(t)}_{\ell^2} 
\equiv
\sup_{t \in [0,T]} \no{e^{-i\mu q_0^{2p_1} t} \, u(t) - e^{-i\gamma F(q_0^2) t} \, U(t)}_{\ell^2}
\leq C \max\left\{\ve^{2p_1+1}, \ve^{2p_2+1}\right\}.
\end{equation}
\end{theorem}

\begin{proof}
Let 
\begin{equation}\label{delta-def}
\delta_n( t) := \phi_n(t) - \Phi_n( t)  \equiv e^{-i\mu q_0^{2p_1} t} \, u_n(t) - e^{-i\gamma F(q_0^2) t} \, U_n(t),
\end{equation}
with the second equality due to \eqref{zhi4} and illustrating the fact that the difference of solutions to the original gAL and gDNLS equations involves different phase factors multiplying each solution.
Subtracting \eqref{mgdnls} from \eqref{mgal} yields
\begin{equation}\label{eq:Deltanv}
i \frac{d \delta_n}{dt} + \kappa \left(\Delta \delta\right)_n = \gamma \, G(\Phi_n) - \frac 12 \mu \, \mathcal G(\phi_n), 
\end{equation}
with $G$ and $\mathcal G$ given by \eqref{G-def} and \eqref{Gal-def}  respectively. 
Similarly to \eqref{fhat0}, taking the finite Fourier transform~\eqref{dft-def}
of  equation~\eqref{eq:Deltanv} and then integrating with respect to $t$, we find
\begin{equation}\label{D-hat-nzbc1}
\what \delta(\xi, t) 
=e^{-4i\kappa \sin^2 \left(\frac{\xi}{2} \right)t} \, \what{\delta(0)}(\xi) - i\int_0^t 
e^{-4i\kappa \sin^2 \left(\frac{\xi}{2}\right)(t-\tau)}
\left[\gamma \, \what{G(\Phi)}(\xi, \tau) - \frac 12 \mu \, \what{\mathcal G(\phi)}(\xi, \tau)\right] d\tau.
\end{equation}
Hence, analogously to \eqref{D-L20}, 
\begin{equation}\label{D-L2}
\left\| \delta(t) \right\|_{\ell^2}
\leq
\no{\delta(0)}_{\ell^2}
+
\int_0^t \left[ |\gamma| \no{G(\Phi)(\tau)}_{\ell^2} + \frac 12 |\mu| \no{\mathcal G(\phi)(\tau)}_{\ell^2} \right] d\tau.
\end{equation}
Note that the boundary conditions \eqref{zhi2} and \eqref{vbcn} imply that $\lim_{|n|\to\infty} G(\Phi)(t) = \lim_{|n|\to\infty} \mathcal G(\phi)(t) = 0$. 
In light of the bounds \eqref{into-ineq} and \eqref{into-ineq-al-2} with $p=p_2$ and $p=p_1$, respectively, and also the embedding~\eqref{emb-ineq}, the  inequality~\eqref{D-L2} yields
\begin{align}\label{d-l2}
&\left\| \delta(t) \right\|_{\ell^2}
\leq
\no{\delta(0)}_{\ell^2}
+t \cdot\sup_{\tau \in [0,t]}
\Big\{
2\sqrt 2 \, |\gamma| K  \left( \left\| \Phi(\tau) \right\|_{\ell^2} + \left\| \zeta \right\|_{\ell^\infty} + q_0\right)^{2p_2}
\left(\left\| \Phi(\tau) \right\|_{\ell^2} + \no{|\zeta| - q_0}_{\ell^2}\right)
\nn\\
&
+ 
 |\mu|
\left[
8 p_1 \left(\no{\phi}_{\ell^2}+\no{\zeta}_{\ell^\infty}+q_0\right)^{2p_1} \left(\no{\phi}_{\ell^2} + \no{|\zeta|-q_0}_{\ell^2}\right)
+
4 q_0^{2p_1} \no{\phi}_{\ell^2} +  q_0^{2p_1} \no{\zeta''}_{\ell^2}
\right]
\Big\}.
\end{align}

According to part (i) in each of Theorems \ref{dnls-life-t}  and \ref{al-life-t}, 
$$
\sup_{t\in [0, T_c]}\left\|\phi(t)\right\|_{\ell^2} \leq A_1 \varepsilon,
\quad
\sup_{t\in [0, T_c]}\left\|\Phi(t)\right\|_{\ell^2} \leq A_2 \varepsilon.
$$
Thus, fixing $T \in (0, T_c]$ and using \eqref{ic-close}, estimate \eqref{d-l2} implies
\begin{equation}\label{nvH1b}
\begin{aligned}
\sup_{t\in [0,T]}\no{\delta(t)}_{\ell^2}
&\leq 
 C_0 \max\left\{\ve^{2p_1+1}, \ve^{2p_2+1}\right\}
 +
T 
\,  \Big\{
2\sqrt 2 \, |\gamma| K  \left(A_2 + B_0 + B\right)^{2p_2}
\left(A_2 + B_2\right) 
\ve^{2p_2+1}
\\
&\quad
+ 
|\mu| 
\left[
8 p_1 \left(A_1 +B_0 + B\right)^{2p_1} \left(A_1 + B_2\right)
+
4 B^{2p_1} A_1 +  B^{2p_1} B_4 
\right] \ve^{2p_1+1}
\Big\},
\end{aligned}
\end{equation}
which readily yields the bound \eqref{al-dnls-close} with constant 
\begin{equation}
\label{defconC}
\begin{aligned}
C = C_0 
+
T \,  \Big\{
&
|\mu|
\left[
8 p_1 \left(A_1 +B_0 + B\right)^{2p_1} \left(A_1 + B_2\right)
+
4 B^{2p_1} A_1 +  B^{2p_1} B_4 
\right]
\\
&
+ 
2\sqrt 2 \, |\gamma| K  \left(A_2 + B_0 + B\right)^{2p_2}
\left(A_2 + B_2\right)
\Big\},
\end{aligned}
\end{equation}
thereby completing the proof of the theorem.
\end{proof}

\begin{remark}
 \label{Dbest}
The difference between using the Duhamel formula \eqref{D-hat-nzbc1} and the standard integral formulation for abstract ODEs for deriving the proximity estimates in Theorem~\ref{al-dnls-t} is the following: Due to the boundedness of the linear operator $(\Delta_d \phi)_n = \kappa (\Delta \phi)_n$ on $\ell^2$ with norm $\no{\Delta_d}_{\mathcal{B}(\ell^2)} = 4/h^2$, a standard integral formulation for abstract ODEs would yield the first term of the right hand side of the estimates \eqref{D-L2} and \eqref{d-l2} to be  $\frac{4}{h^2}\no{\delta(0)}_{\ell^2}$, which becomes unbounded in the continuous limit $h \to 0$. This would affect considerably the distance estimate  when $\no{\delta(0)}_{\ell^2}\neq 0$.  On the other hand, in the discrete regime, where $h=\mathcal{O}(1)$ the associated distance estimate would be equivalent to \eqref{nvH1b}, with respect to the order of  $\ve$, while in the anticontinous limit $h\rightarrow\infty$ this term would vanish, yet keeping the order of the estimate \eqref{nvH1b}  with respect to $\ve$  intact.   Notably, in both approaches, this term vanishes  when $\no{\delta(0)}_{\ell^2} = 0$, which is the case we consider in our numerical studies of Section~\ref{sim-s}. 

Without the specific handling of nonlinear terms provided by the lifespans in Theorems~\ref{dnls-life-t} and~\ref{al-life-t}, a standard Gronwall inequality argument for $\delta(t)$, based on the Lipschitz properties of $\Delta_d$ and the nonlinear operators, would result in exponential growth estimates of the form $\mathcal{O}(e^{\varepsilon t})$ for the Duhamel formula or $\mathcal{O}(e^{\varepsilon t h^{-2}})$ for the standard ODE integral formulation.
\end{remark}

\begin{remark}
By the embedding \eqref{emb-ineq}, the bound \eqref{d-l2}
is also satisfied by any $\ell^r$ norm of the distance with $r\geq 2$ and, in particular, by the $\ell^\infty$ norm of the distance. \end{remark}

\begin{remark}
If $0 < \ve < 1$, then the bound \eqref{al-dnls-close} reduces to $\mathcal O(\ve^{2\min\{p_1, p_2\}+1})$ and Theorem \ref{al-dnls-t} provides a closeness result between the solutions of the gAL and gDNLS equations supplemented with the nonzero boundary conditions~\eqref{uU-bc}. On the other hand,  if $\ve\geq 1$ then the bound \eqref{al-dnls-close} is of $\mathcal O(\ve^{2\max\{p_1, p_2\}+1})$; yet, although in this case the result of Theorem \ref{al-dnls-t} does not correspond to closeness, it still provides control of the distance between the solutions of the two equations. 
\end{remark}

\begin{remark}\label{tc-r}
In the case of $0<\ve<1$, if $T = \mathcal O(T_c) = \mathcal O(\ve^{-2\min\{p_1, p_2\}})$ then the bound \eqref{al-dnls-close} becomes of $\mathcal O(\ve)$. This is consistent with the bound on the difference of solutions that one readily obtains by combining the estimates of Theorems \ref{dnls-life-t} and \ref{al-life-t} with the triangle inequality.
\end{remark}

\section{Asymptotic equivalence between \textnormal{g}AL and \textnormal{g}DNLS with power nonlinearity}
\label{as-s}

Consider the gAL and gDNLS equations with general power nonlinearity, namely equations \eqref{gal} and~\eqref{gdnls} with $F(x) = x^p$, $p\geq 1$. 
Suppose further that the two equations with respective solutions $u_n(t), U_n(t)$ are supplemented with the same initial condition, $u_n(0) = U_n(0) = v(x_n)$, for some suitable function $v$, and exist for sufficiently large time interval.
Here, we assume that the gAL and gDNLS equations are discretizations of the corresponding generalization of the continuous NLS equation for appropriately large $\kappa=h^{-2}$. We will estimate the difference between the solutions of these two equations emerging from the same initial condition with $0<h\leq 1$. Specifically, we will show that, at least for small times, their difference is of $\mathcal O(h^2t)$. Therefore, for the usual lattice with $h=1$, the two equations remain close for at least small times $t$, and the difference of their solutions will be $\mathcal O(t)$.
\begin{theorem}
\label{aseq}
Consider the gAL and gDNLS equations~\eqref{gal} and~\eqref{gdnls} with a power nonlinearity $F(x) = x^p$, $p\geq 1$, and the same initial condition, $u_n(0) = U_n(0) = v(x_n)$, for some suitable function $v$. 
There exists a time $T>0$ and a constant $C>0$ independent of $h$ such that the solutions $u_n$ and $U_n$ of the aforestated initial value problems satisfy the bound
$$
\no{u_n(t)-U_n(t)}_{\ell^\infty} \leq C h^2 t
$$
for all $0<h\leq1$ and $t\in [0,T]$.
\end{theorem}

\begin{proof}
In what follows, we use the symbol $\lesssim$ to denote $\leq C$ where $C>0$ is a constant independent of $h$. Moreover, we assume that the solutions are uniformly bounded in $h$. More precisely, we assume that $|u_n|+|U_n|<M$ for some $M>0$ independent of $h$ and for all $t>0$. 

The Taylor expansions with integral remainder
$$
u(x_n+h,t)=\sum_{k=0}^3 \frac{h^k}{k!} u^{(k)}(x_n,t)+I_1,\qquad  u(x_n-h,t)=\sum_{k=0}^3 \frac{(-h)^k}{k!}u^{(k)}(x_n,t)+I_2,
$$
where 
$$I_1=\int_{x_n}^{x_n+h} \frac{(x_{n}+h-z)^3}{6}u^{(4)}(z,t)~dz,\qquad I_2=-\int_{x_n-h}^{x_n} \frac{(x_{n}-h-z)^3}{6}u^{(4)}(z,t)~dz,$$
can be added to yield the formula
\begin{equation}\label{eq:fdeq1}
\frac{u(x_n+h,t)+u(x_n-h,t)}{2}=u(x_n,t)+\frac{h^2}{2}u_{xx}(x_n,t)+\frac{I_1+I_2}{2}.
\end{equation}
Additionally, using the Taylor expansions with integral remainder
$$
u(x_n+h,t)=\sum_{k=0}^5 \frac{h^k}{k!} u^{(k)}(x_n,t)+I_3, \qquad u(x_n-h,t)=\sum_{k=0}^5 \frac{(-h)^k}{k!} u^{(k)}(x_n,t)+I_4,
$$
where 
$$I_3=\int_{x_n}^{x_n+h} \frac{(x_{n}+h-z)^5}{5!}u^{(6)}(z,t)~dz,\qquad I_4=-\int_{x_n-h}^{x_n} \frac{(x_{n}-h-z)^5}{5!}u^{(6)}(z,t)~dz,$$
we obtain
\begin{equation}\label{eq:fdeq2}
\frac{u(x_n+h,t)-2u(x_n,t)+u(x_n-h,t)}{h^2}=u_{xx}(x_n,t)+\frac{h^2}{12} u^{(4)}(x_n,t)+\frac{I_3+I_4}{h^2}.
\end{equation}
Writing $u_n(t)=u(x_n,t)$ and substituting (\ref{eq:fdeq1}) and (\ref{eq:fdeq2}) into the gAL equation \eqref{gal} yields
\begin{align*}
0&=i \frac{du_n}{dt}+\frac{1}{h^2}(\Delta u)_n+\frac{1}{2}\mu F(|u_n|^2)(u_{n+1}+u_{n-1})\\
&=i u_t(x_n,t) + u_{xx}(x_n,t) +\mu F(|u(x_n,t)|^2)u(x_n,t)
\\
&\quad
+h^2\left[ \frac{\mu}{2} F(|u(x_n,t)|^2) u_{xx}(x_n,t) +\frac{1}{12}u^{(4)}(x_n,t)\right]+C_1(x_n,t)
\end{align*}
where
\begin{align*}
C_1(x,t)&=\frac{\mu |u|^2}{12}\left[\int_x^{x+h}(x+h-z)^{3}u^{(4)}(z)~dz-\int_{x-h}^{x}(x-h-z)^{3}u^{(4)}(z)~dz\right]
\\
&\quad + \frac{1}{120 h^2}\left[\int_x^{x+h}(x+h-z)^5u^{(6)}(z)~dz-\int_{x-h}^{x}(x-h-z)^5u^{(6)}(z)~dz\right].
\end{align*}
In other words, the solution $u_n$ to the AL equation is $u_n(t)=u(x_n,t)$ where $u$ is the solution of the initial value problem
\begin{equation}\label{eq:ivp1}
\begin{aligned}
&i u_t + u_{xx} +\mu F(|u|^2)u+h^2\left[ \frac{\mu}{2} F(|u|^2) u_{xx} +\frac{1}{12}u^{(4)}\right]+C_1(x,t)=0,\\
& u(x,0)=v(x).
\end{aligned}
\end{equation}
Such equations resulting from finite difference approximations are often called \textit{effective equations} of the discrete approximations.

Along the same lines, we can see that the solution $U_n$ of the gDNLS equation \eqref{gdnls} is $U_n(t)=U(x_n,t)$ where $U$ is the solution of the initial value problem
\begin{equation}\label{eq:ivp2}
\begin{aligned}
&i U_t + U_{xx} +\mu F(|U|^2)U+\frac{h^2}{12}U^{(4)}+C_2(x,t)=0,\\
&U(x,0)=v(x),
\end{aligned}
\end{equation}
with
$$C_2(x,t)=\frac{1}{120 h^2}\left[\int_x^{x+h}(x+h-z)^5U^{(6)}(z)~dz-\int_{x-h}^{x}(x-h-z)^5U^{(6)}(z)~dz\right].$$
Note that both functions $C_1$ and $C_2$ contain the remainders of the Taylor expansions used to derive these effective equations, and both are $\mathcal O(h^4)$ for bounded solutions. In particular, there is a constant $C>0$ independent of $h$ such that the difference $D(x,t) =C_1(x,t)-C_2(x,t)$  satisfies $|D(x,t)|\leq Ch^4$.

Let us further assume that there is a $T_c>0$ such that the solutions of the effective equations exist, are unique and satisfy $\no{\p_x^6 u}_{L^\infty(\mathbb R)}+\no{\p_x^6 U}_{L^\infty(\mathbb R)}\leq M$ for all $t\in [0,T_c]$. 
To estimate the difference between the solutions of these equations, we first subtract the equations~\eqref{eq:ivp1} and \eqref{eq:ivp2} to obtain the equation for the difference $e:=u-U$:
\begin{equation}\label{eq:dif}
i e_t + e_{xx} +\mu\left(F(|u|^2)u-F(|U|^2)U\right)+\frac{h^2}{12} e^{(4)}+h^2\frac{\mu}{2} F(|u|^2) u_{xx}+D(x,t)=0.
\end{equation}
Also after differentiation with respect to $x$, we obtain
\begin{equation}\label{eq:dif2}
i e_{xt} + e_{xxx} +\mu \left[F(|u|^2)u-F(|U|^2)U\right]_x +\frac{h^2}{12} e^{(5)}+h^2\frac{\mu}{2} (F(|u|^2) u_{xx})_x + D_x(x,t)=0.
\end{equation}
It is easy to deduce, yet by the definition of $D(x,t)$, that for bounded solutions there exists a constant $C>0$ independent of $h$ such that $|D_x(x,t)|\leq C h^3$.

Note that for the power nonlinearity $F(x) = x^p$, $p\geq 1$, the inequalities (3.38) and (3.56) of \cite{hkmms2024} imply that there are constants $A_1(p,u,U)>0$ and $A_2(p,u,U)>0$ such that
\begin{align}\label{eq:ineqMD1}
\no{F(|u|^2)u-F(|U|^2)U}_{L^2}&\leq A_1(p,u,U)\no{u-U}_{L^2},
\\
\label{eq:ineqMD2}
\no{\left(F(|u|^2)u-F(|U|^2)U\right)_x}_{L^2}&\leq A_2(p,u,U)\no{u-U}_{H^1}.
\end{align}

Multiplication of \eqref{eq:dif} with $\bar{e}$ and integration over $\mathbb{R}$ yields the following equation for the $L^2(\mathbb R)$ norm of $e$ and its derivatives:
\begin{equation}\label{eq:dif3}
\begin{aligned}
\frac{i}{2}\frac{d}{dt}\no{e}_{L^2}^2- \no{e_x}_{L^2}^2&+\int_{-\infty}^\infty \left(\mu (F(|u|^2)u-F(|U^2|)U)\bar{e}\right)dx+\frac{h^2}{12}\no{e_{xx}}_{L^2}^2\\
&
+\int_{-\infty}^\infty \left(h^2\frac{\mu}{2}  F(|u|^2) u_{xx} \bar{e}+D(x,t)\bar{e}\right)dx=0.
\end{aligned}
\end{equation}
Taking imaginary parts, we have
\begin{equation}\label{eq:dif4}
\frac{1}{2}\frac{d}{dt}\no{e}_{L^2}^2+\mu \Im \int_{-\infty}^\infty (F(|u|^2)u-F(|U^2|)U)\bar{e}~dx+\Im\int_{-\infty}^\infty \left(h^2\frac{\mu}{2} F(|u|^2) u_{xx} \bar{e}+ D(x,t)\bar{e}\right)dx=0.
\end{equation}
Using the bounds for $U$, $u$ and $D$, and (\ref{eq:ineqMD1}) we infer that there a constant $C>0$ such that
\begin{equation}\label{eq:ineq1}
\frac{d}{dt}\no{e}_{L^2}^2\leq C( h^2\no{e}_{L^2}+\no{e}_{L^2}^2).
\end{equation}

Similarly, multiplying 
\eqref{eq:dif2} with $\bar{e}_x$ and integrating we obtain
\begin{equation}\label{eq:dif5}
\begin{aligned}
\frac{i}{2}\frac{d}{dt}\no{e_x}_{L^2}^2-\no{e_{xx}}_{L^2}^2&+\mu \int_{-\infty}^\infty \left( [(F(|u|^2)u-F(|U^2|)U)]_x\bar{e}_x\right) dx +\frac{h^2}{12} \no{e_{xxx}}_{L^2}^2\\
&+\int_{-\infty}^\infty \left(h^2\frac{\mu}{2}(F(|u|^2)u_{xx})_x\bar{e}_x+D_x(x,t)\bar{e}_x\right) dx=0.
\end{aligned}
\end{equation}
The imaginary part of the last equation is
\begin{equation}
\begin{aligned}
\frac{1}{2}\frac{d}{dt}\no{e_x}_{L^2}^2& +\mu \Im\int_{-\infty}^\infty \left( [(F(|u|^2)u-F(|U^2|)U)]_x\bar{e}_x \right) dx\\
&+\Im\int_{-\infty}^\infty \left(h^2\frac{\mu}{2}(F(|u|^2)u_{xx})_x\bar{e}_x+D_x(x,t)\bar{e}_x\right) dx=0,
\end{aligned}
\end{equation}
which, with the help (\ref{eq:ineqMD2}), leads to the inequality
\begin{equation}\label{eq:ineq2}
\frac{d}{dt}\no{e_x}_{L^2}^2\leq C( h^2\no{e_x}_{L^2} +\no{e}_{H_1}^2).
\end{equation}

Inequalities \eqref{eq:ineq1} and \eqref{eq:ineq2} along with Lemma 1 of \cite{ab1991} combine to yield 
\begin{equation}
    \label{eq:ineq}
\frac{d}{dt}\no{e}^2_{H^1}\leq C( h^2\no{e}_{H^1}+\no{e}^2_{H^1})\leq C(h^2\no{e}_{H^1}+\frac{1}{h^2}\no{e}^3_{H^1}).
\end{equation}
Then, since $e(0) = 0$, Lemma 2 from \cite{ab1991} implies the existence of a time $0<T\leq T_c$ such that
\begin{equation}\label{eq:ineq3}
\no{e}_{H^1}\leq C h^2t
\end{equation}
for all $t\in[0,T]$, where the constants involved in the inequality are independent of $h$. Using the Sobolev embedding theorem, we deduce that $\no{e}_{L^{\infty}} \leq C h^2t$, which completes the proof.
\end{proof}

 \begin{remark}
 [Theorem \ref{al-dnls-t} vs Theorem \ref{aseq}]\label{DBest2}  
 The estimate of Theorem~\ref{aseq} cannot be derived directly from Theorem~\ref{al-dnls-t}, despite the fact that the integral equation \eqref{D-hat-nzbc1} incorporates the discretization parameter $\kappa = 1/h^2$ within the phase term $e^{-4i\kappa \sin^2 (\xi/2)t}$. Indeed, note that the discrete Laplacian $\Delta_d \phi = \kappa (\Delta \phi)_n$ is a bounded linear operator on $\ell^2$ and generates the propagator $e^{-i\Delta_d t} \in C^{\infty}(\mathbb{R}, \mathcal{B}(\ell^{2}))$, where $\mathcal{B}(\ell^2)$ denotes the space of bounded linear operators on $\ell^2$. This propagator is unitary, as reflected by the unimodular term $e^{-4i\kappa \sin^2 (\xi/2)t}$ in the Fourier space, resulting in the inequality  \eqref{D-L2} and the subsequent estimate \eqref{d-l2} and eliminating the discretization parameter in the  $\ell^2$ estimation of Theorem~\ref{al-dnls-t}.  This behavior also persists in the standard Picard integral approach when $\no{\delta(0)}_{\ell^2}=0$, as explained in Remark \ref{Dbest}, as well as in the (even worse) scenario via Gronwall's inequality, which leads to an exponentially growing estimate involving $e^{Ct/h^2}$, which diverges in the continuous limit.
 
 Consequently, Theorem~\ref{al-dnls-t} provides a proximity result based on the size of the initial and background data of $\mathcal{O}(\max\left\{\ve^{2p_1+1}, \ve^{2p_2+1}\right\} t)$ and does not imply the $\mathcal{O}(h^2t)$ estimate of Theorem~\ref{aseq}. The latter is therefore essential: by using the relevant effective equations \eqref{eq:ivp1} and \eqref{eq:ivp2} for the gAL and gDNLS, respectively, we shift our attention to the parameter $\kappa = 1/h^2$ and are able to explicitly extract the truncation error before the estimation step. This allows us to rigorously characterize the asymptotic equivalence of the systems as $h\rightarrow 0$ through the resulting estimate of $\mathcal{O}(h^2t)$.
\end{remark}

\section{Numerical analysis and simulations}
\label{sim-s}

We now illustrate the theoretical results proved in the previous sections via different numerical experiments on a nonzero background. We thereby demonstrate the proximity (over reasonable time intervals) between the dynamics of the gAL and gDNLS lattices as predicted by Theorem~\ref{al-dnls-t}, and also the relation between the size of the respective solution lifespans and the initial data and background, as established by Theorems \ref{dnls-life-t} and \ref{al-life-t}.  The focus of our simulations is on the case of a power nonlinearity, i.e. on the gAL equation \eqref{gal} and the gDNLS equation~\eqref{gdnls} with $F(x) = x^p$, $p\geq 1$.

\subsection{Numerical scheme}
\label{SecScheme}
We begin with a description of the numerical scheme used in our numerical simulations. Let $\mathcal{I}=\{x_n=nh, \quad \text{for}\quad n=0,1,\dots,N\}$ with $N=2L/h$ and $h\leq 1$.  Recalling that $\kappa=h^{-2}$ and setting $\mu=1$ in \eqref{gal} for simplicity, we consider the periodic initial-boundary value problem for the gAL equation over a finite lattice, namely
\begin{equation}\label{eq:al2}
\begin{aligned}
&i\frac{d u_n}{dt} + \frac{1}{h^2} \left(\Delta u\right)_n +\frac{1}{2}\,\left|u_{n}\right|^{2p}\left(u_{n+1}+u_{n-1}\right)=0, \quad n=0,1,\dots,N, \ t>0,
\\
&u_n(0) = f(n), \quad n=0,1,\dots,N,
\\
&u_{N+k}(t) = u_k(t), 
\end{aligned}
\end{equation}
where the initial condition also satisfies the periodic boundary conditions, namely $f(N+k) = f(k)$, for some suitable integer $k$ such that  the infinite lattice is approximated while keeping the boundary effects at bay. 
Similarly, the periodic initial-boundary value problem for the gDNLS equation \eqref{gdnls} with $F(x) = x^p$, $p\geq 1$, takes the form 
\begin{equation}\label{eq:dnls2}
\begin{aligned}
&i\frac{d U_n}{dt} + \frac{1}{h^2} \left(\Delta U\right)_n +\left|U_{n}\right|^{2p}U_{n}=0, \quad n=0,1,\dots,N, \ t>0,
\\
&U_n(0) = f(n), \quad n=0,1,\dots,N,
\\
&U_{N+k}(t) = U_k(t).
\end{aligned}
\end{equation}

For any $p\geq 1$, the solutions of the gAL lattice \eqref{eq:al2} and the gDNLS lattice \eqref{eq:dnls2} respectively conserve the functionals
\begin{equation}\label{eq:energyal}
E_{\text{AL}}(t) := \frac{1}{2}h\sum_{n=0}^{N-1}\bar{u}_n\left(u_{n+1}+u_{n-1}\right),
\quad
E_{\text{DNLS}}(t) := h\sum_{n=0}^{N-1}|U_n|^2, 
\end{equation}
i.e. $\dfrac{d}{dt}E_{\text{AL}}(t) = \dfrac{d}{dt}E_{\text{DNLS}}(t)=0$.
The conservation of these \textit{quadratic functionals} naturally motivates the use of numerical schemes that preserve them too. Ensuring  conservation to a sufficient degree  can serve as an indicator of the accuracy of the computed numerical solutions. 
For the numerical solution of the gAL and gDNLS equations we employ the fourth-order Gauss-Legendre Runge-Kutta method, which respects the conservation of the functionals  \eqref{eq:energyal}.   We implement the method in  two stages, as described by the Butcher \textit{tableau} \cite{b2003} shown in \eqref{eq:Butcher}, in order to integrate numerically the ordinary differential equations that correspond to the real and imaginary parts of the  equations \eqref{eq:al2} and \eqref{eq:dnls2}.
\begin{equation}\label{eq:Butcher}
\begin{array}{c|c}
{\bf c} & {\bf A} \\
\hline 
\rule{0pt}{3ex}  & {\bf b}^T 
\end{array}\quad =\quad 
\begin{array}{c|cc}
\frac{1}{2}-\frac{1}{6}\sqrt{3} & \frac{1}{4} & \frac{1}{4}-\frac{1}{6}\sqrt{3} \\
\rule{0pt}{3ex} \frac{1}{2}+\frac{1}{6}\sqrt{3} & \frac{1}{4}+\frac{1}{6}\sqrt{3} & \frac{1}{4} \\ [3pt]
 \hline
\rule{0pt}{3ex}  & \frac{1}{2} & \frac{1}{2}
\end{array}.
\end{equation}
Specifically, if $u_n=r_n+is_n$ then equation (\ref{eq:al2}) yields the system
\begin{equation}\label{eq:alsys1}
\begin{aligned}
        &\frac{d}{dt}r_n+\frac{1}{h^2}(\Delta s)_n+\frac{1}{2}(r_n^2+s_n^2)^{p}(s_{n+1}+s_{n-1})=0,\\
        &\frac{d}{dt}s_n-\frac{1}{h^2}(\Delta r)_n-\frac{1}{2}(r_n^2+s_n^2)^{p}(r_{n+1}+s_{r-1})=0,\\
\end{aligned}
\end{equation}
and if $U_n=R_n+i~S_n$ then equation (\ref{eq:dnls2}) results in the system
\begin{equation}\label{eq:dnlssys1}
\begin{aligned}
        &\frac{d}{dt}R_n+\frac{1}{h^2}(\Delta S)_n+(R_n^2+S_n^2)^{p}S_n=0,\\
        &\frac{d}{dt}S_n-\frac{1}{h^2}(\Delta R)_n-(R_n^2+S_n^2)^{p}R_n=0.
\end{aligned}
\end{equation}
For the numerical approximation of the solutions of the nonlinear systems arising from the time discretization of (\ref{eq:alsys1}) and (\ref{eq:dnlssys1}), we used the Jacobian-free complex-step Newton method as described in \cite{m2025}. This method has the same quadratic order of convergence as the classical Newton method~\cite{m2025}. It is worth noting that, in order to detect a possible blow-up of a solution, we regard the numerical blow-up time as the instance at which the numerical solution produces an overflow. Although more sophisticated numerical methods involving adaptive time-stepping methodology have been developed for the study of blow-up phenomena (e.g. see  \cite{ADKM2003,KK2018} and the references therein), we found that the present numerical approach provides sufficient information for the purposes of the present work. Moreover, the results were corroborated using the classical explicit Runge-Kutta method of fourth order. Finally, for the numerical simulations, the specific choice of the number of lattice sites $N$ does not substantially influence the results. Nevertheless, we typically used $N = \mathcal{O}(10^2)$ or $N = \mathcal{O}(10^4)$ in order to suppress boundary effects. Likewise, whether $N$ is odd or even has no impact on the numerical outcomes and, in particular, on the blow-up investigation.

\subsection{The case $p=1$: integrable AL versus non-integrable DNLS}
\label{subsecintegrable}

We begin the presentation of the numerical results with a study  focusing on the potential proximity between the integrable AL and the non-integrable DNLS systems \eqref{al} and \eqref{dnls}, which correspond to the case $p=1$ in the gAL and gDNLS equations \eqref{gal} and \eqref{gdnls}.
Specifically, we set $h = p =1$ in equations \eqref{eq:al2} and \eqref{eq:dnls2} and consider the initial conditions
\begin{equation}\label{eq:num}
u_n(0)=U_n(0)=q_0 \left(1+i\sech(n)\right),
\end{equation}
which correspond to $\phi_n(0)=\Phi_n(0)=i q_0 \sech(n)$ for the modified equations \eqref{mgal} and \eqref{mgdnls}, i.e. a localized perturbation of a constant background of amplitude $q_0$.

In our first experiment, we set the background amplitude to $q_0=0.1$ which, by assumptions~\eqref{dnls-ic} and \eqref{al-ic}, yields $\varepsilon \approx \no{\phi(0)}_{\ell^2} \approx 0.14$. The spatiotemporal evolution of the densities $|u_n(t)|^2$ and $|U_n(t)|^2$  over the lattice $[-300, 300]$ and the time interval $[0,800]$ is shown in Figure \ref{fig1AL-DNLSdensity}. The top panel corresponds to the AL equation and the middle panel to the DNLS equation. The two patterns appear almost indistinguishable, \textit{illustrating the persistence of the main features of the universal behavior associated with modulational instability \cite{blmt2018}, even in the non-integrable DNLS system, and over a remarkably long time interval}. This observation provides strong numerical evidence for the proximity between the AL and DNLS dynamics in the regime of small initial data.

\begin{figure}[h!]
 	\begin{center}
    \begin{tabular}{cc}
 		\includegraphics[width=0.4\textwidth]{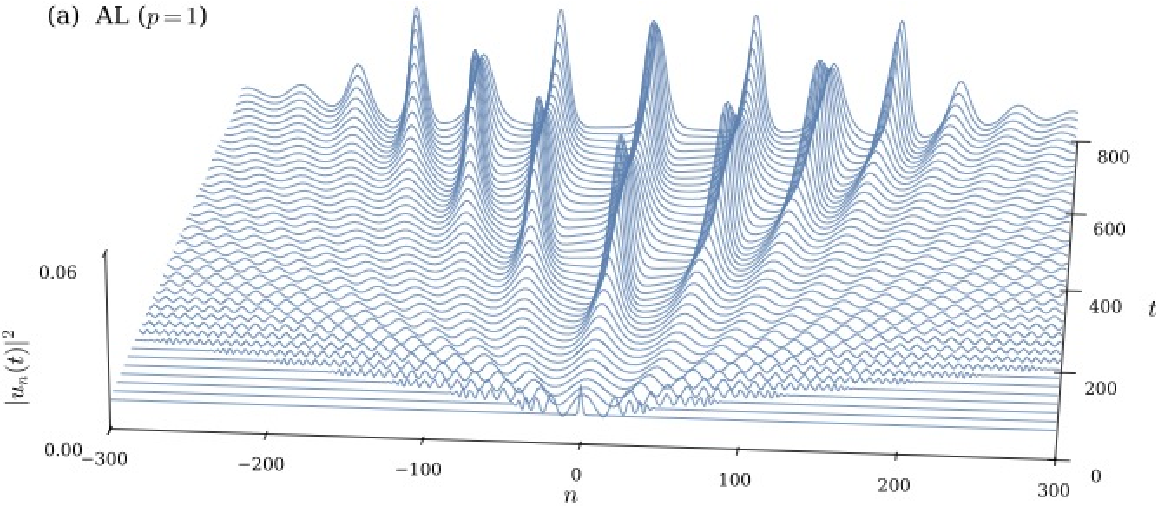} &
        \includegraphics[width=0.4\textwidth]{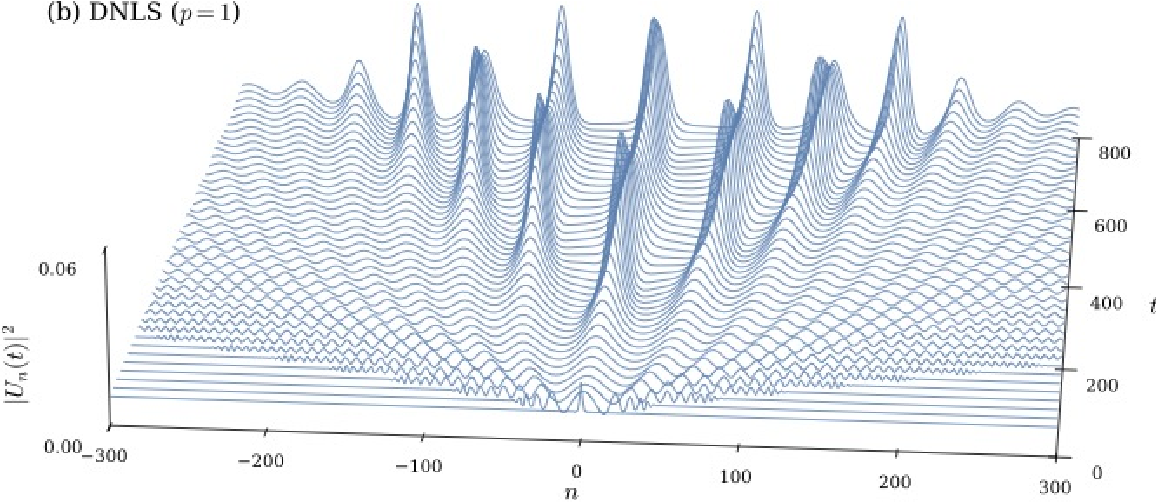}
    \end{tabular}
    \includegraphics[width=0.7\textwidth]{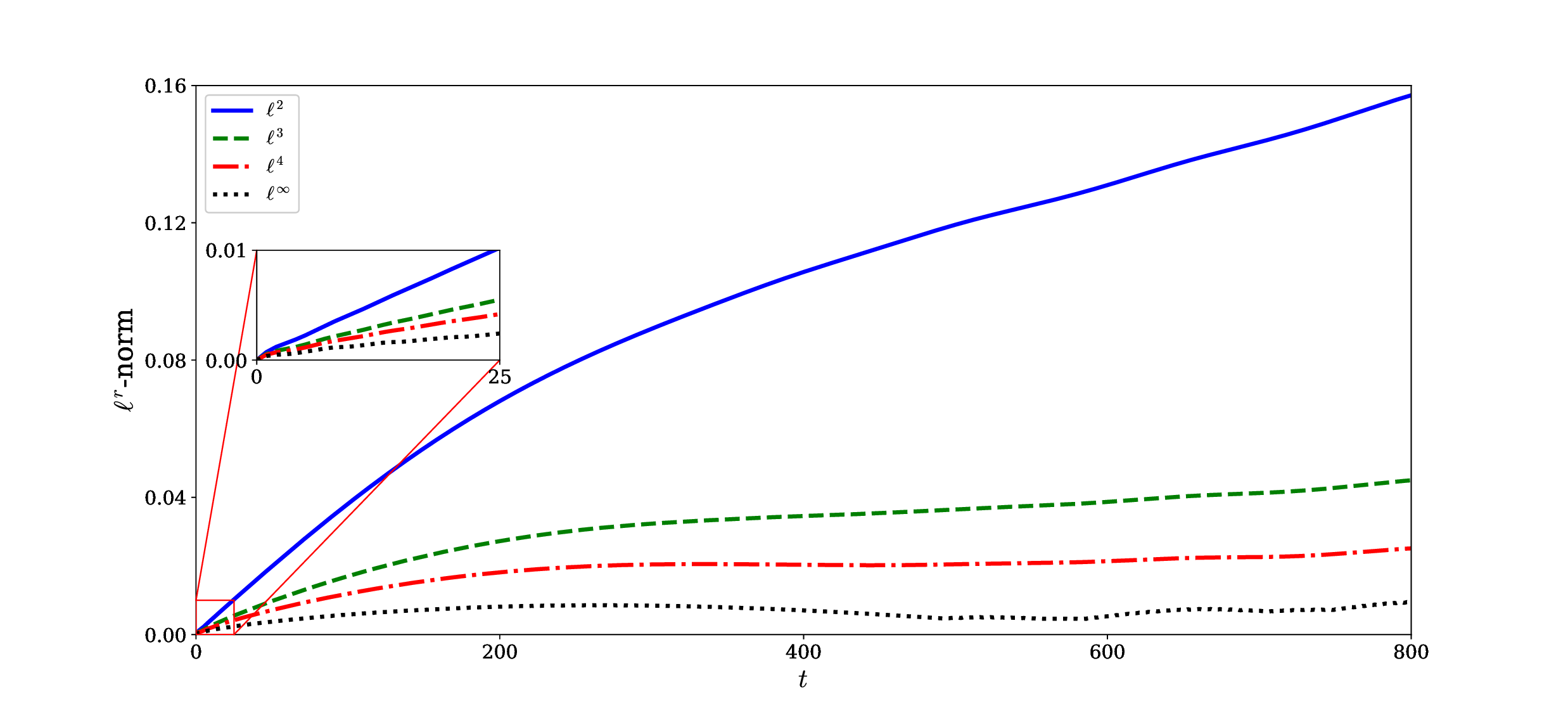}	
 	\end{center}
 	\caption{\textit{Top}: (a) Spatiotemporal evolution of the density $|u_n(t)|^2$ of the solution to the AL lattice \eqref{eq:al2} with $h=1$ and the initial condition \eqref{eq:num} with $q_0=0.1$, and (b) the same evolution for the DNLS lattice \eqref{eq:dnls2}. \textit{Bottom}: Time evolution of the distance $\delta(t)$ given by \eqref{delta-def} in $\ell^r$ for $r=2,3,4,\infty$.}
 	\label{fig1AL-DNLSdensity}
 \end{figure}

The remarkable robustness of the dynamics, even in the transition from the integrable AL to the non-integrable DNLS, is further illustrated by the bottom panel of Figure \ref{fig1AL-DNLSdensity}. This panel depicts the time evolution of the  norm $\no{\delta(t)}_{\ell^r}$, $r=2,3,4,\infty$, of the distance $\delta(t)$ given by \eqref{delta-def}.
The numerical results are in excellent agreement with the proximity theory of Theorem \ref{al-dnls-t}, even for considerably long times. This fact can be further emphasized by noting that, in the present setting of $p_1=p_2=1$,  solutions of both systems exist globally. From the definition of the constant $C$ in~\eqref{defconC}, and noting that $C_0=0$ therein (since the initial data coincide for both systems), Theorem~\ref{al-dnls-t} and, in particular, inequality \eqref{d-l2}, indicates at most linear growth of the distance $\no{\delta(t)}_{\ell^2}$, namely
\begin{equation}
\label{distnum1}
\no{\delta(t)}_{\ell^2}\lesssim \ve^3 t.
\end{equation}
Consequently, the distance remains of order $\mathcal{O}(\varepsilon)$ for times $t\approx \mathcal{O}\left(\varepsilon^{-2}\right)$, in agreement with Remark~\ref{tc-r}. In other words, when the dynamics of both systems is triggered by the same initial data, the solution of the non-integrable DNLS equation remains within a distance of $\mathcal{O}(\varepsilon)$ from the AL solution over times of $\mathcal{O}\left(\varepsilon^{-2}\right)$. Since in our numerical example $\varepsilon=\mathcal{O}(10^{-1})$, the theory predicts that the dynamics should remain within $\mathcal{O}(10^{-1})$ up to times of order $\mathcal{O}(10^{2})$. Observing that the evolution of $\no{\delta(t)}_{\ell^2}$ represented by the solid (blue) curve in the bottom panel of Figure \ref{fig1AL-DNLSdensity} exhibits an almost linear growth, we conclude that  proximity between the AL and DNLS systems indeed persists up to the end of the numerical integration at $t=800$, hence even exceeding the theoretical expectations.

The excellent agreement between theory and numerical results is also evident at shorter times. Specifically, according to Theorem \ref{al-dnls-t}, for times $t\approx \mathcal{O}(1)$ the DNLS solution should  certainly remain within $\mathcal{O}(\varepsilon^3)$ of the AL solution. This is confirmed numerically in the inset of the bottom panel of Figure \ref{fig1AL-DNLSdensity}, which further illustrates the embedding properties of the $\ell^r$-spaces (recall~\eqref{emb-ineq}). 

For increased $0<\varepsilon<1$ compared to the above value of $\varepsilon \approx 0.14$, the divergence between the two solutions is expected to occur sooner, in accordance with the linear growth estimate \eqref{distnum1}. Such behavior occurs when the dynamics is triggered from the initial condition \eqref{eq:num} with $q_0=0.4$, which yields a larger $\varepsilon \approx \no{\phi(0)}_{\ell^2} \approx 0.56$. The bottom panel of Figure \ref{fig2AL-DNLSdensity}   illustrates the corresponding time evolution of the norms for the difference $\delta(t)$. Notably, the top and middle panels of that figure, which depict the spatiotemporal evolution of the densities $|u_n(t)|^2$ and $|U_n(t)|^2$ of the AL and DNLS solutions, respectively, reveal an interesting phenomenon: although the distance $\delta(t)$, measured in various norms, increases significantly more rapidly compared to the case $q_0=0.1$ of Figure \ref{fig1AL-DNLSdensity}, the patterns of the two solutions still remain almost indistinguishable. The faster growth of the norms is a consequence of the different  propagation speeds of the outgoing spatiotemporal oscillations of small amplitude emitted from the central cores of  the densities $|u_n(t)|^2$ and $|U_n(t)|^2$, an effect that becomes more pronounced in the case $q_0=0.4$ compared to $q_0=0.1$.
It is noted that, over the time interval $[0,800]$, the DNLS solution with $q_0=0.1$ attains a maximum value $\max_{n,t}|U_n|^2 \approx 0.053684$, whereas the corresponding AL solution reaches $\max_{n,t}|u_n|^2 \approx 0.053781$. For $q_0 = 0.4$, the corresponding maximum values over the time interval $[0,100]$ are $\max_{n,t} |U_n|^2 \approx 1.085558$ and $\max_{n,t}|u_n|^2\approx 0.977670$, respectively.

\begin{figure}[h!]
 	\begin{center}
    \begin{tabular}{cc}
 		\includegraphics[width=0.4\textwidth]{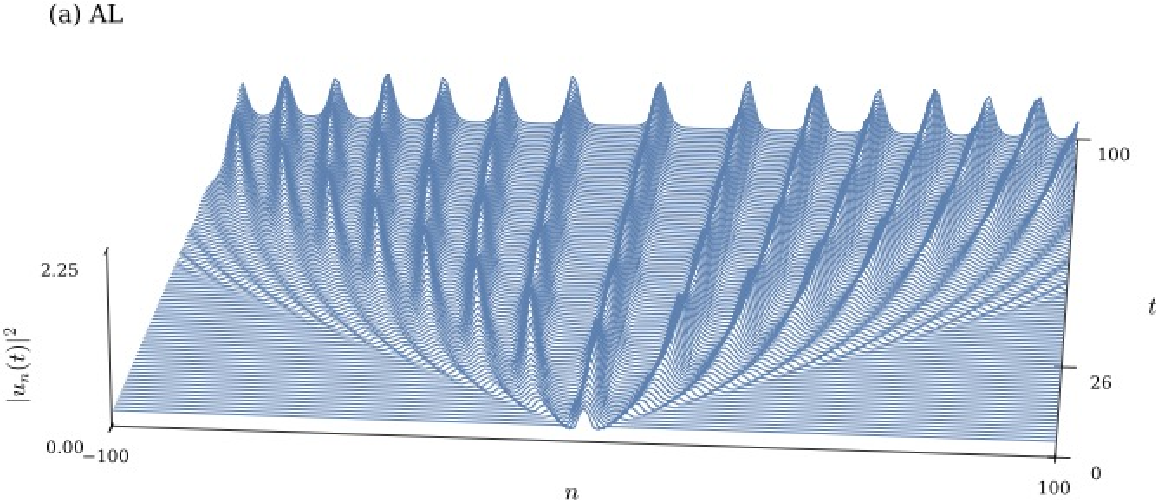}&
        \includegraphics[width=0.4\textwidth]{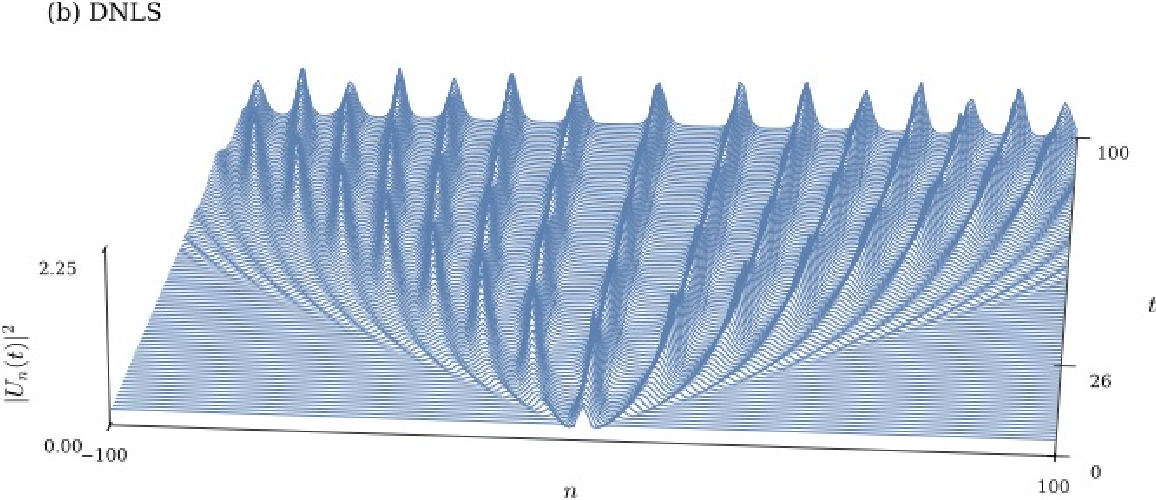}
    \end{tabular}
    \includegraphics[width=0.7\textwidth]{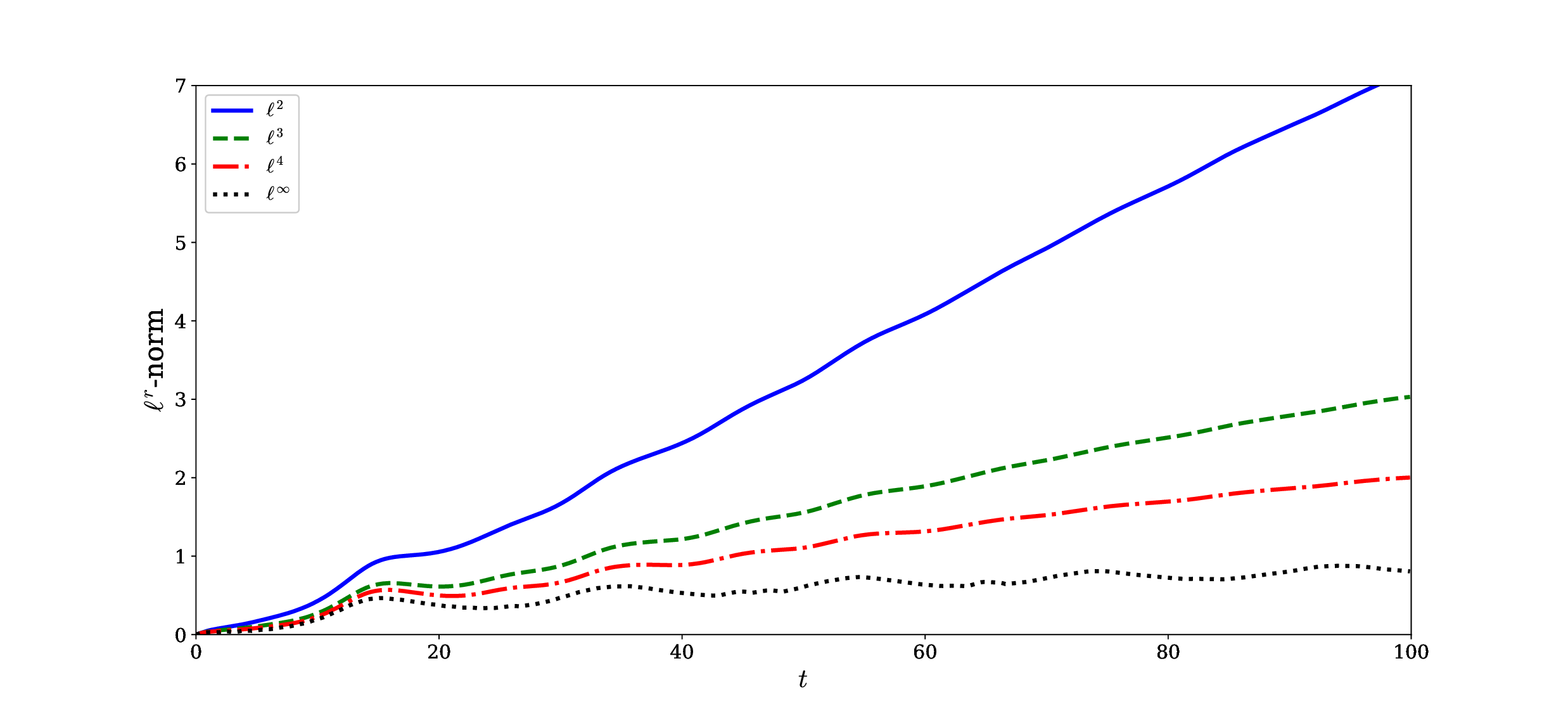}	
 	\end{center}
 	\caption{(a) Spatiotemporal evolution of the density $|u_n(t)|^2$ of the solution to the AL lattice \eqref{eq:al2} with $h=1$ and the initial condition \eqref{eq:num} with $q_0=0.4$. \textit{Middle}: The same evolution for the DNLS lattice \eqref{eq:dnls2}. \textit{Bottom}: Time evolution of the distance $\delta(t)$ given by \eqref{delta-def} in $\ell^r$ for $r=2,3,4,\infty$.}
 	\label{fig2AL-DNLSdensity}
 \end{figure}

Finally, we compare the proximity results obtained in this work for the discrete AL and DNLS systems with the corresponding results of \cite{hkmms2024} in the setting of continuous NLS models. In the continuous setting with nonzero boundary conditions, the divergence between the dynamics of the integrable cubic NLS \eqref{nls} and its non-integrable generalization  \eqref{gnlsp}  occurs much sooner when $p>1$.  Although the main features of the modulational instability pattern remain largely unchanged, the faster growth of the distance in various norms is due to subtle differences in the oscillations and the steeper gradients that develop far from the core. Notably, this divergence arises even for small values of the background $q_0$, affecting the evolution of the distance in the $H^1$ Sobolev norm.

In contrast, in the discrete setting, the divergence remains moderate even over very long time intervals. We identify two main reasons for this difference between the discrete and continuous models: (i) the control of the discrete $H^1$ norm by the $\ell^2$ norm, implying that narrow oscillations that may emerge in the non-integrable system do not dramatically affect the evolution of the $\ell^2$ distance; (ii) the order of the nonlinearities, since in the continuous case the cubic integrable NLS is compared against a non-integrable NLS with $p>1$, which amplifies the growth of differences for larger amplitudes, whereas in the integrable AL and non-integrable DNLS systems considered here the nonlinearity is cubic in both cases. The role of the nonlinearity order is further illustrated in the continuous case by the non-integrable NLS equation with saturable nonlinearity, which is effectively of cubic order (see equation (1.10) and relevant discussion in pages 157-158 of \cite{hkmms2024}), making the saturable model a structurally more stable non-integrable counterpart with respect to the dynamics of the integrable NLS equation.

\subsection{The case $p=2$: collapse for gAL and quasi-collapse for gDNLS}
\label{subsecblow}
After examining the integrable versus the non-integrable setting in the case $p=1$ of cubic nonlinearities, we proceed to  study the dynamics of the gAL and gDNLS systems for $p>1$, where integrability is lost in both cases. As before, we use the initial conditions \eqref{eq:num} for both systems, keeping $h=1$ fixed.

The dynamics of the regime $p>1$ presents a significant departure from those of the case $p=1$. In particular, as numerical simulations reveal, for $p=2$ the gAL lattice exhibits \textit{finite-time blow-up}. Before presenting our numerical studies, we recall  the main implications of our theoretical analysis. Theorems \ref{dnls-life-t} and \ref{al-life-t} guarantee a minimum lifespan for the solutions of the gDNLS and gAL lattices, respectively, while Theorem \ref{globexfin} rules out the possibility of finite-time collapse for the gDNLS in the case of a finite lattice considered in the numerical simulations. In addition, Theorem~\ref{al-dnls-t} provides quantitative proximity estimates between the two systems. 

We illustrate the above theoretical implications by first considering the dynamics for $p=2$ and $q_0=0.4$. The corresponding results are presented in Figures \ref{Fig3AL-DNLSCollapse} and \ref{Fig4AL-DNLSCollapse}.
Figure \ref{Fig3AL-DNLSCollapse} shows the dynamics of the spatiotemporal patterns for both systems for  $t\in [0,26]$, with gAL in the top panel and gDNLS in the bottom panel. Over that time interval, the dynamics is nearly indistinguishable, as further confirmed by the evolution of the distance norms $\no{\delta(t)}_{\ell^r}$ shown in the bottom panel.

  \begin{figure}[h!]
 	\begin{center}
    \begin{tabular}{cc}
 		\includegraphics[width=0.4\textwidth]{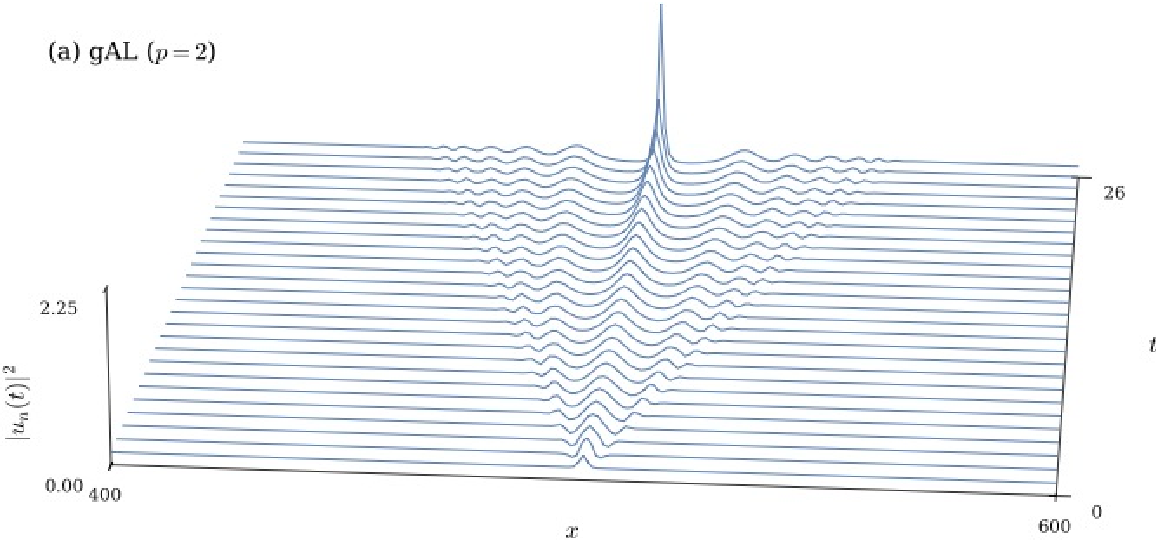}&
 		\includegraphics[width=0.4\textwidth]{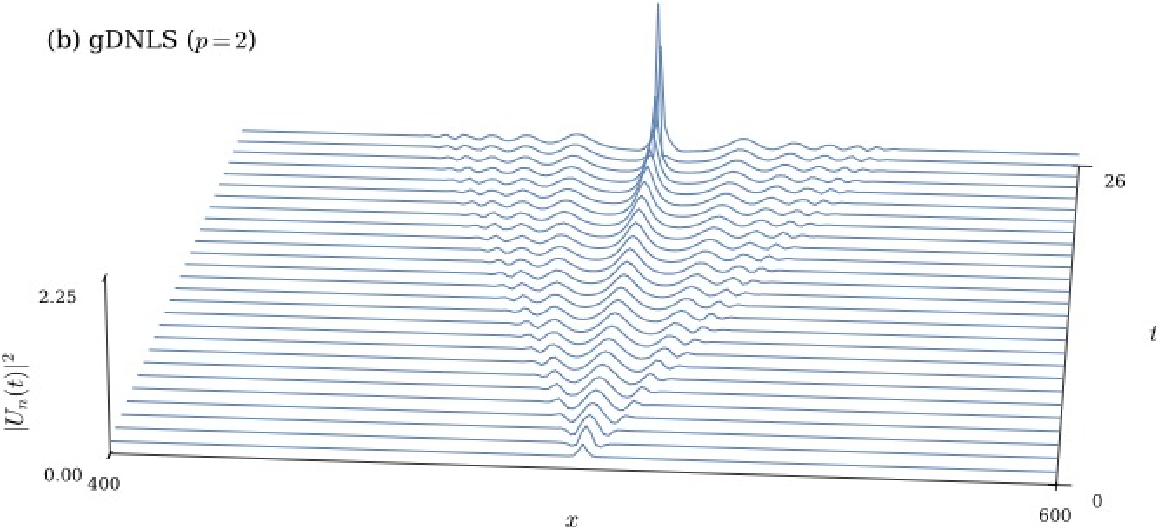}
        \end{tabular}
        \includegraphics[width=0.7\textwidth]{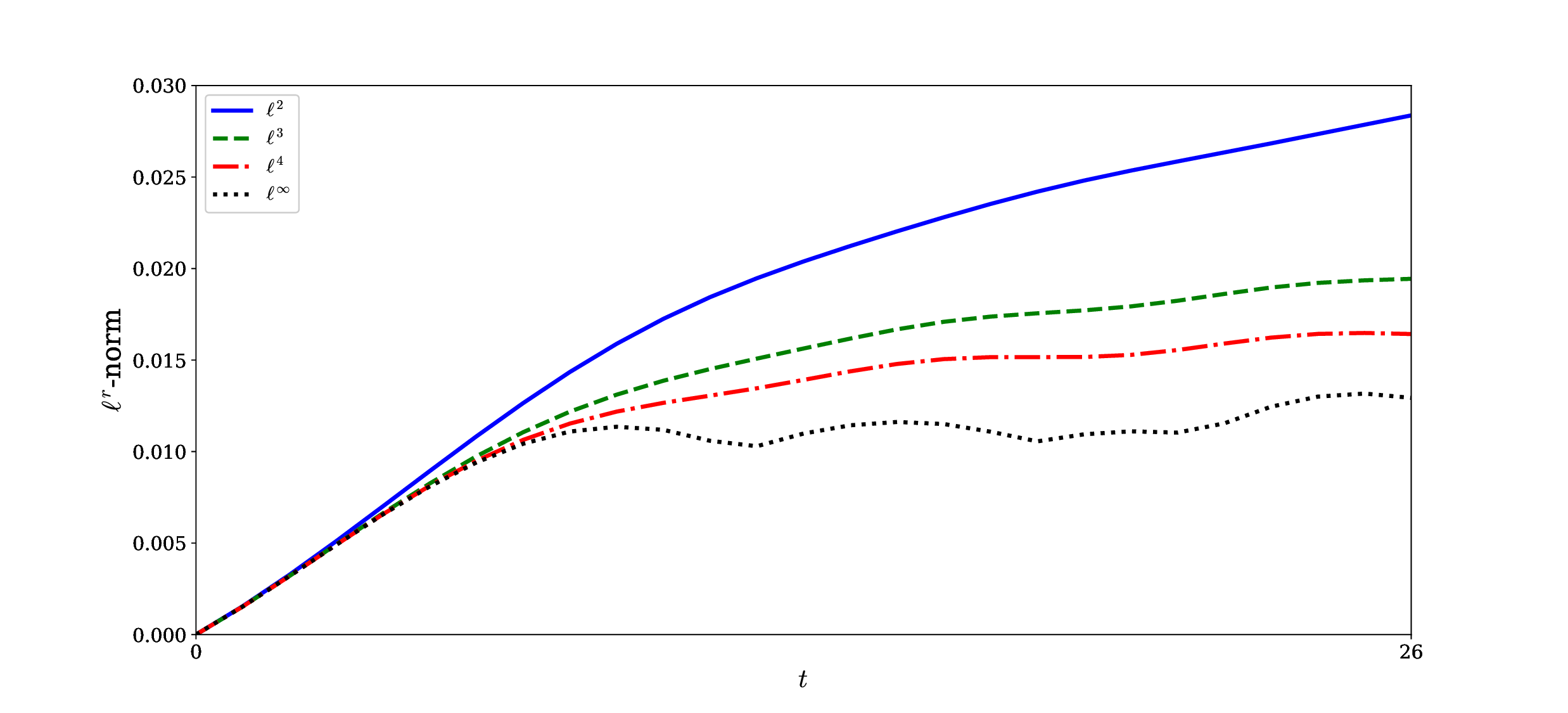}
 	\end{center}
 	\caption{\textit{Top}: Spatiotemporal evolution over $t\in [0,26]$ of the density $|u_n(t)|^2$  of the solution to the gAL equation \eqref{eq:al2} with $p=2$, $h=1$ and the initial condition~\eqref{eq:num} with $q_0=0.4$. \textit{Middle}: The same evolution for the gDNLS equation~\eqref{eq:dnls2} with $p=2$, $h=1$. \textit{Bottom}: Time evolution of the distance $\delta(t)$ given by \eqref{delta-def} in $\ell^r$ for $r=2,3,4,\infty$.}
 	\label{Fig3AL-DNLSCollapse}
 \end{figure}
 
 \textit{Beyond $t\approx 26$, however, the two systems exhibit dramatically different features.} The gAL system undergoes finite-time blow-up at $t\approx 26.31$, forming a sharp point singularity at the center. In contrast, the gDNLS solution, while closely following the gAL profile near $t\approx 26$, \textit{continues to exist}, fully consistent with the global result of Theorem \ref{globexfin}, as shown in Figure \ref{Fig4AL-DNLSCollapse}. Indeed, for this experiment, gDNLS was integrated up to $t=1000$ confirming the striking difference in long-term dynamics between the two systems. For a better illustration of the persistent modulational instability structure, we show only the dynamics for $t\in [0,100]$).  Importantly, for gDNLS, the profile of the solution at $t=26$ and beyond exhibits characteristics of a \textit{quasi-collapse}, as manifested by the emergence of the narrow, spike-like wave-forms that are preserved within the cone of the modulational instability pattern.

\begin{figure}[h!]
 	\begin{center}
        \includegraphics[width=0.6\textwidth]{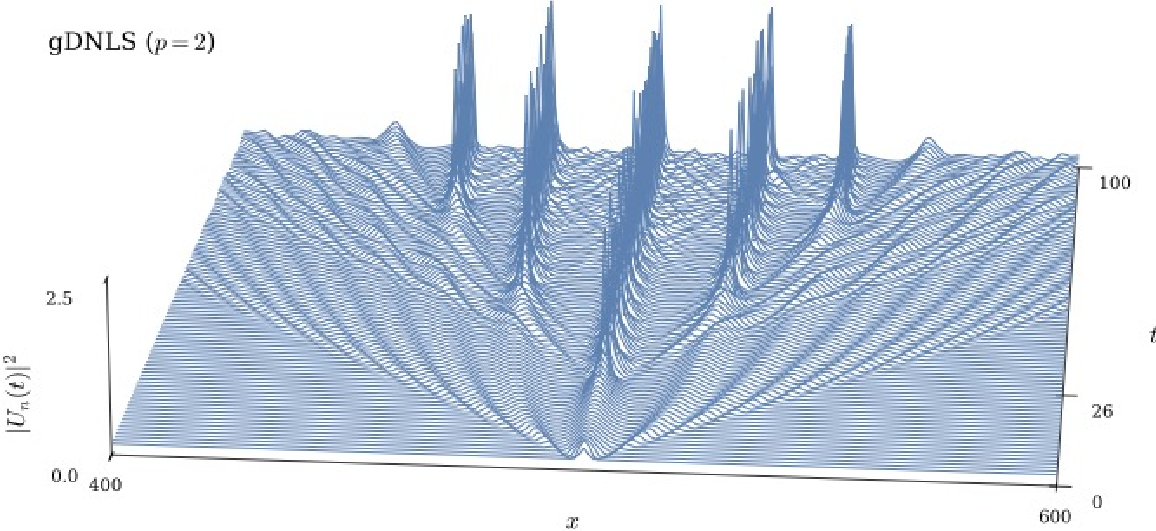}
 	\end{center}
 	\caption{Spatiotemporal evolution over $t\in [0,100]$ of the density $|U_n(t)|^2$  of the solution to the gDNLS lattice \eqref{eq:dnls2} with $p=2$, $h=1$ and the initial condition \eqref{eq:num} with $q_0=0.4$.}
 	\label{Fig4AL-DNLSCollapse}
 \end{figure}

For the non-integrable gAL equation with $p=2$,  
\textit{finite-time blow-up arises even for small values of the background $q_0$}, in stark contrast with the globally existing solutions of the integrable AL equation with $p=1$.  In particular, although the lifespan of the gAL solutions is considerably larger for smaller $q_0$, finite-time collapse is nevertheless observed. 
Conversely, the time of existence decreases with $q_0$, in full agreement with the theoretical result of $\mathcal{O}(\varepsilon^{-4})$ proved in Theorem~\ref{al-life-t}, as summarized in Table~\ref{Tblowvsq0}. 

\begin{table}[h!]
\begin{tabular}{c|ccccccc}
$q_0$ & 0.1 & 0.12 & 0.14 & 0.16 & 0.18 & 0.19 & 0.20 \\\hline
$T$ & 24435 & 10773 & 5345 & 2889 & 1666 & 1292 & 1001 \\
\end{tabular}
\caption{Approximate blow-up times for gAL with $p=2$ as $q_0$ varies.}
\label{Tblowvsq0}
\end{table}

\begin{figure}[tbh!]
 	\begin{center}
 		\begin{tabular}{cc}
 		\includegraphics[width=0.47\textwidth]{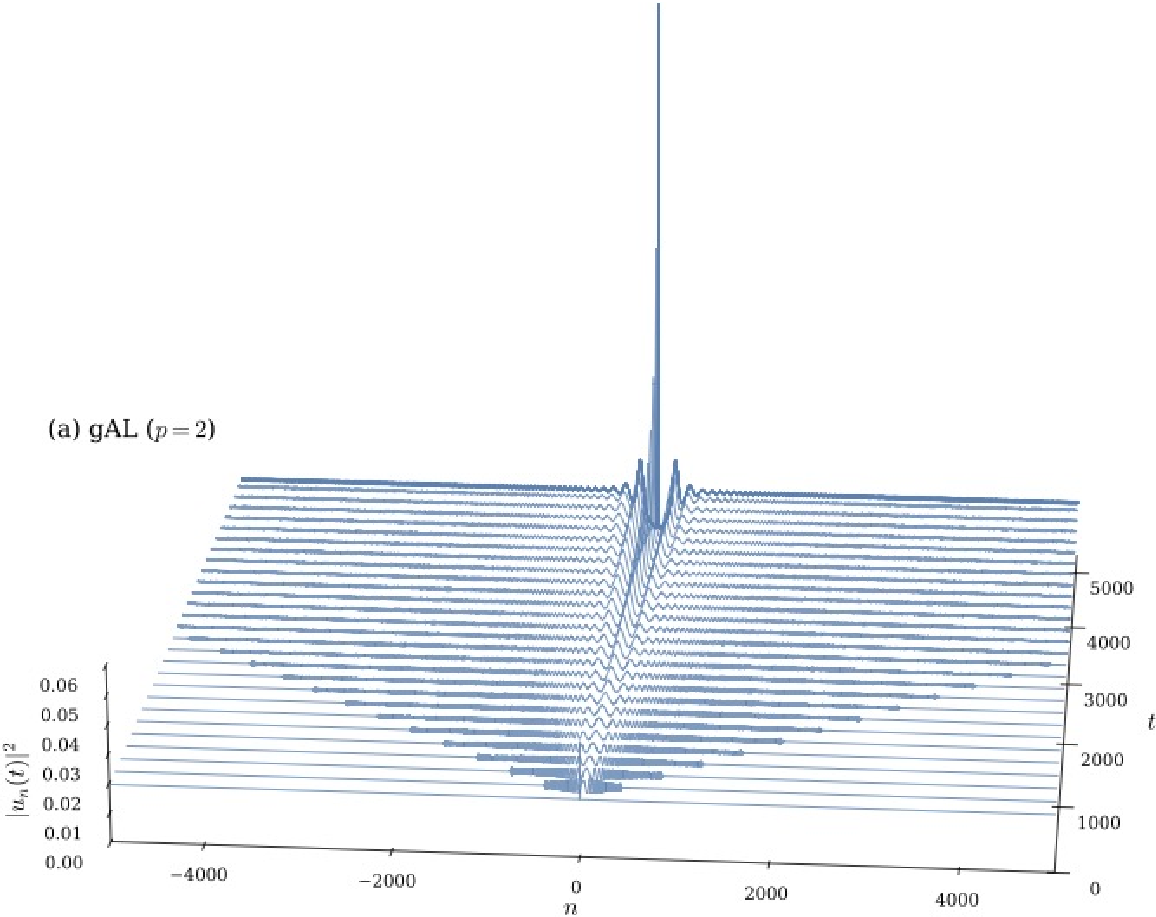}&
        \includegraphics[width=0.47\textwidth]{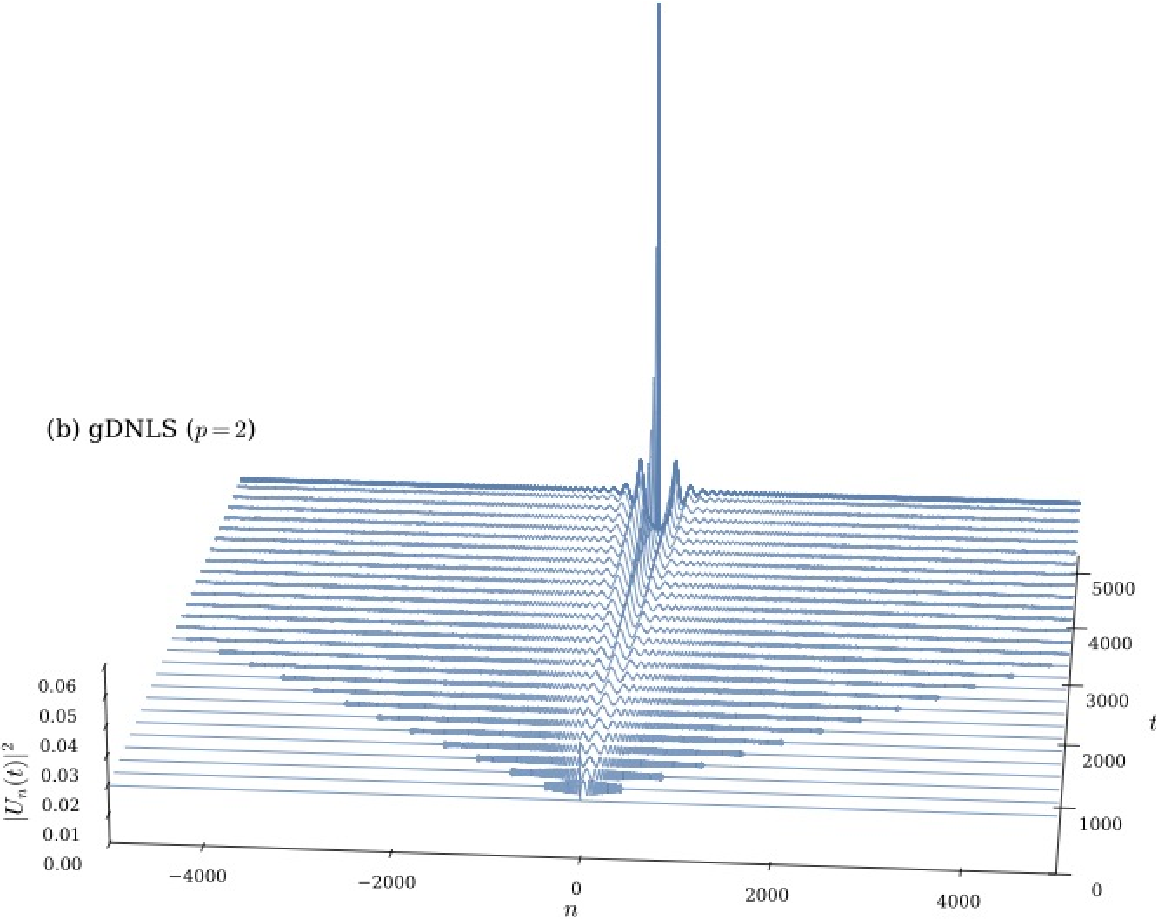}\\
   		\end{tabular}
        \includegraphics[width=0.7\textwidth]{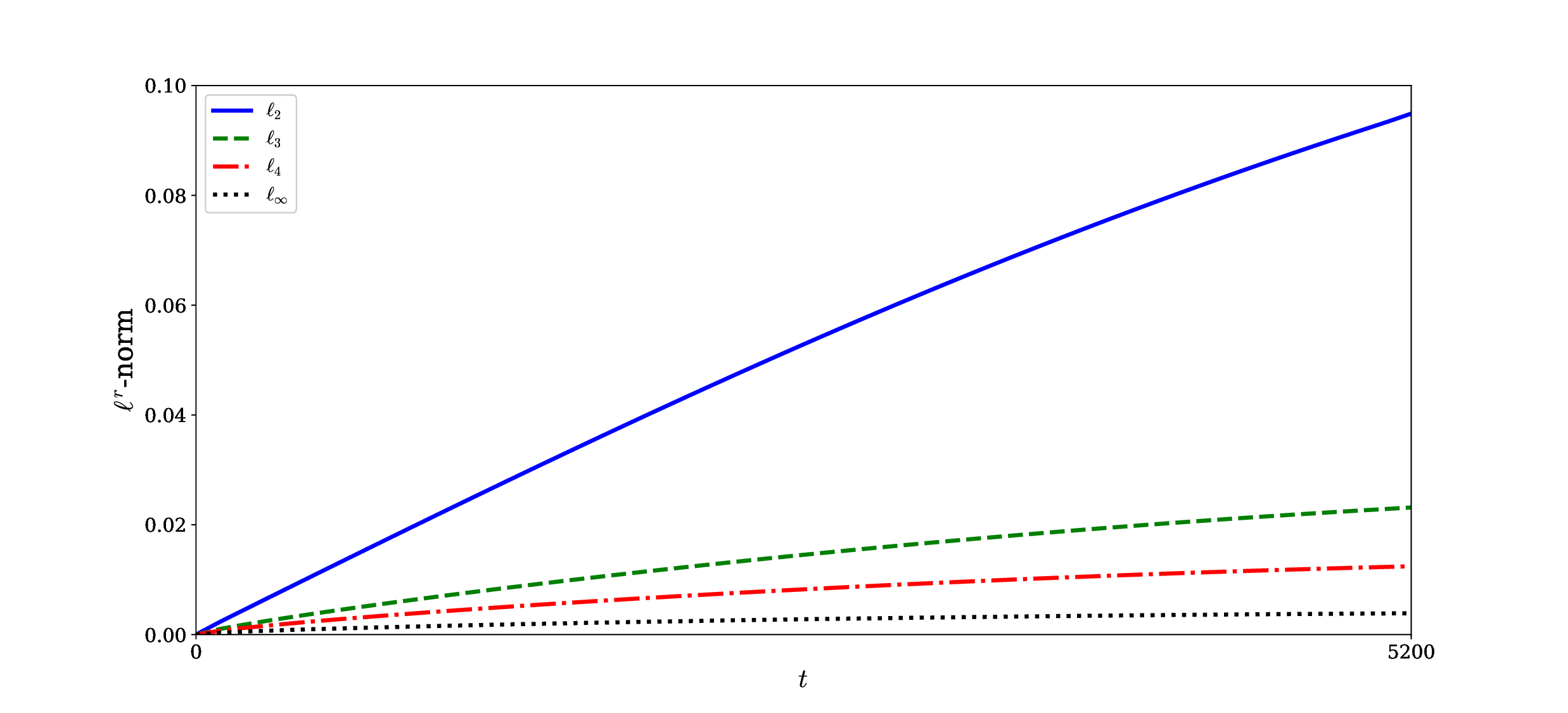}
 	\end{center}
 	\caption{\textit{Top}: (a) Spatiotemporal evolution of the density $|u_n(t)|^2$ of the solution to the gAL lattice~\eqref{eq:al2} and (b) gDNLS lattice~\eqref{eq:dnls2}  with $p=2$, $h=1$ and the initial condition \eqref{eq:num} with $q_0=0.14$, for $t\in [0, 5345]$. \textit{Bottom}: Time evolution of the distance $\delta(t)$ given by \eqref{delta-def} in $\ell^r$ for $r=2,3,4,\infty$.}
 	\label{fig4aAL-DNLSdensity}
 \end{figure}

 \begin{figure}[tbh!]
 	\begin{center}
            \includegraphics[width=0.6\textwidth]{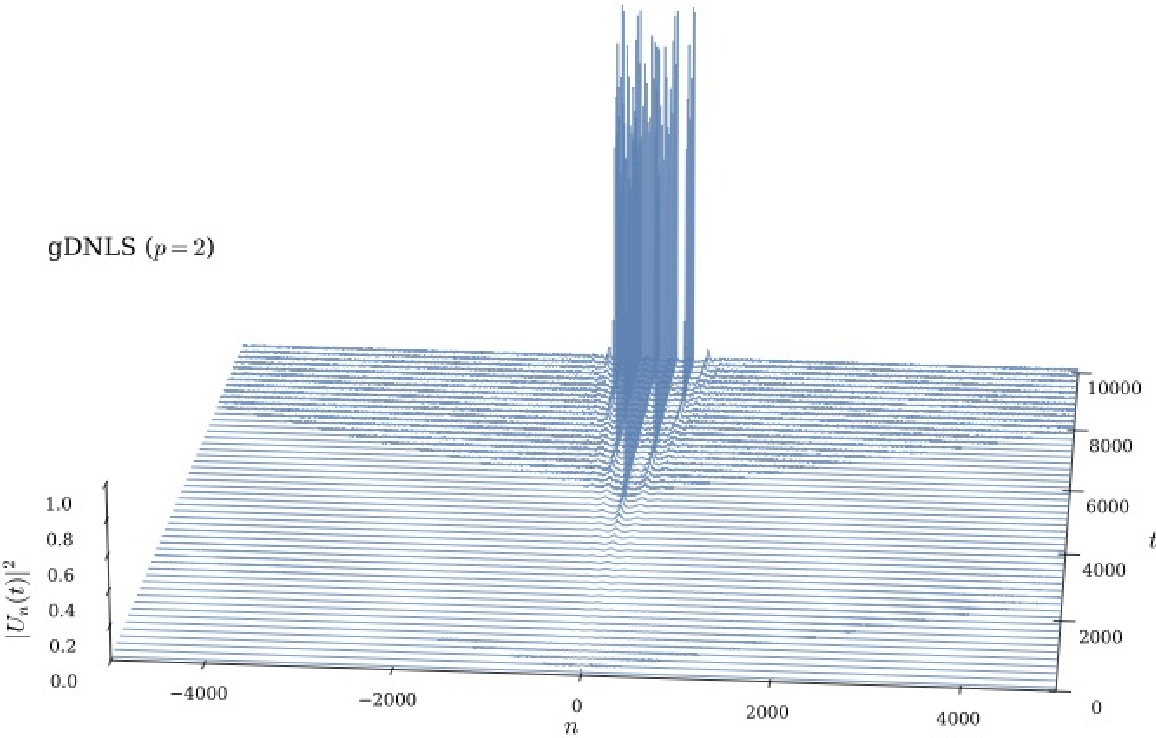}
 	\end{center}
 	\caption{The same evolution for the gDNLS lattice~\eqref{eq:dnls2} as in Fig. \ref{fig4aAL-DNLSdensity} with $q_0=0.14$, $p=2$, $h=1$, for $t\in [0,10000]$. }
 	\label{fig4aAL-DNLSdensity2}
 \end{figure}

 Motivated by Table \ref{Tblowvsq0}, we next present a second example illustrating the blow-up dynamics for small background amplitude, specifically for $q_0 = 0.14$.
The corresponding dynamics for both systems is shown in Figure \ref{fig4aAL-DNLSdensity}, depicting the long-term evolution of the solutions.
This experiment highlights several key features.

First, we again observe an excellent agreement between the two systems up to the blow-up time $t \approx 5345$ of the gAL, as evidenced by the comparison of their spatiotemporal dynamics in the top panels of Figure \ref{fig4aAL-DNLSdensity}.
The difference between the two solutions, measured in the $\ell_p$ norms, remains small up to $t \approx 5200$, as shown in the bottom panel of the same figure.
For larger times, and up to the onset of blow-up in the gAL solution, the dynamics begins to deviate more noticeably.
Importantly, the time of collapse for the gAL remains in excellent consistency with the analytical lifespan estimates established in Theorem~\ref{al-life-t}.
Finally, we note the difference in the scaling of the $t$-axis between Figure~\ref{fig4aAL-DNLSdensity}, which displays the gAL dynamics up to the blow-up time, and Figure \ref{fig4aAL-DNLSdensity2}, which illustrates the gDNLS evolution beyond the gAL blow-up and up to $t \approx 10000$.
Second, the smaller amplitude of $q_0=0.14$ in comparison to the case of $q_0=0.4$ shown in Figures~\ref{Fig3AL-DNLSCollapse} and~\ref{Fig4AL-DNLSCollapse} results in a significant delay in the emergence of the point singularity for the gAL equation, as well as in the appearance of the spike-like waveforms in the gDNLS equation associated with quasi-collapse.

Summarizing the numerical results for  $q_0=0.4$ and $q_0=0.14$, we see that  both experiments strongly confirm the predictive value of the theory established in the previous sections. 
In both cases, the dynamics of the solutions remains proximal over the common interval of existence, whose length is dictated by the minimum guaranteed lifespan of the gAL solution. This lifespan agrees with the theoretically predicted order of $\mathcal{O}(\varepsilon^{-4})$ and, crucially, it coincides with the time of blow-up for gAL and the onset of quasi-collapse for gDNLS. 
In the latter case, the solutions persist globally on the finite lattice, in agreement with Theorem \ref{globexfin} for the finite lattice. 
It is noted that, for $q_0=0.14$, the solution of  gDNLS attains a maximum value of $\max_{n,t}|U_n(t)|^2\approx 2.31436$ over the time interval $[0,10000]$, whereas for $q_0=0.4$ the corresponding maximum value over the interval $[0,100]$ is $\max_{n,t}|U_n(t)|^2\approx 2.70145$. 

At this point, it is worth emphasizing that numerical studies of blow-up phenomena are particularly challenging. Even with advanced numerical schemes, detecting the precise blow-up time remains a highly demanding task. To gauge the accuracy of the scheme described at the beginning of the present section, we performed simulations over a long time interval $t\in [0,1000]$ during which the gDNLS solutions remained well-defined while the gAL solution exhibited blow-up as previously described. A uniform timestep of $\Delta t = 0.01$ was employed. Furthermore, for $q_0 = 0.4$, the conserved quantities remained nearly constant throughout the experiment, namely $E_{\text{gAL}}(t) \approx 160.27229380103$ and $E_{\text{gDNLS}}(t) \approx 160.32065349145$, with the displayed digits preserved over the entire simulation. The Jacobian-free complex-step Newton method converged typically within three iterations and at most four, using a tolerance of $10^{-10}$ for both systems.
\\[2mm]
\textbf{Blow-up for the gAL equation with $p=2$ in the case of zero boundary conditions.}
As stated in Remark \ref{case-iii}, Theorem \ref{al-life-t} also applies in the case of zero boundary conditions, i.e. $q_0=0$, thus offering the possibility of finite-time blow-up. Our numerical experiments support this scenario.
For example, for $p=2$ and initial data of the form $u_n(0) = iA \sech(n)$ with $A=1.2$, we observed the formation of a sharp point singularity at the center of the  spatiotemporal pattern. The amplitude of that solution grew exponentially fast, leading to an almost instantaneous blow-up at $t \approx 1.88$. The order of the minimum guaranteed lifespan given in Theorem \ref{al-life-t} is $\mathcal{O}(\no{u(0)}_{\ell^2}^{-4})$. For $A=1.2$, one computes $\no{u(0)}_{\ell^2}^{-4} \approx 0.12$. Increasing the amplitude to $A=2$ produces a blow-up of the same type, now at $t \approx 0.088$, while in this latter case $\no{u(0)}_{\ell^2}^{-4} \approx 0.016$. Figure \ref{fig:bupzerobg} illustrates the point blow-up patterns for the densities of the solutions associated with these two initial amplitudes.
\begin{figure}[h!]
 	\begin{center}
 		\begin{tabular}{cc}
 		\includegraphics[width=0.48\textwidth]{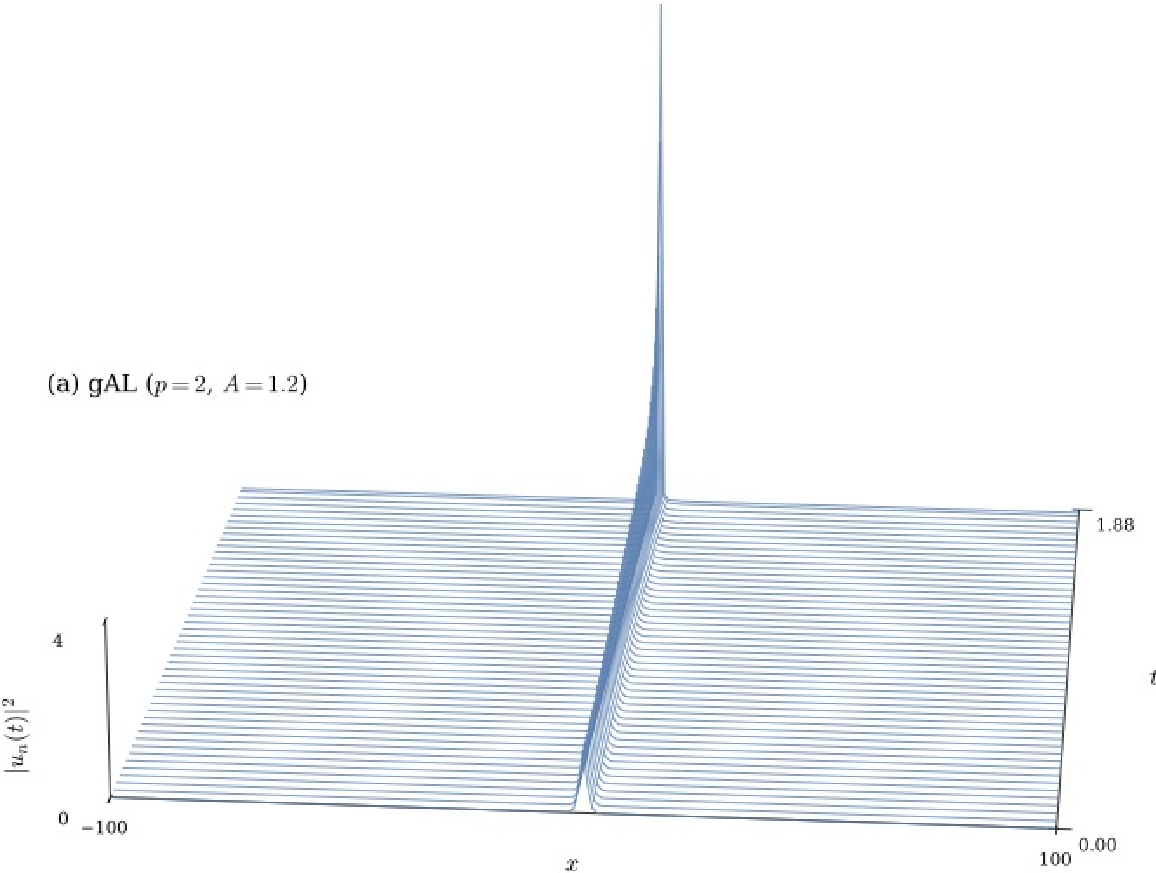} &
            \includegraphics[width=0.48\textwidth]{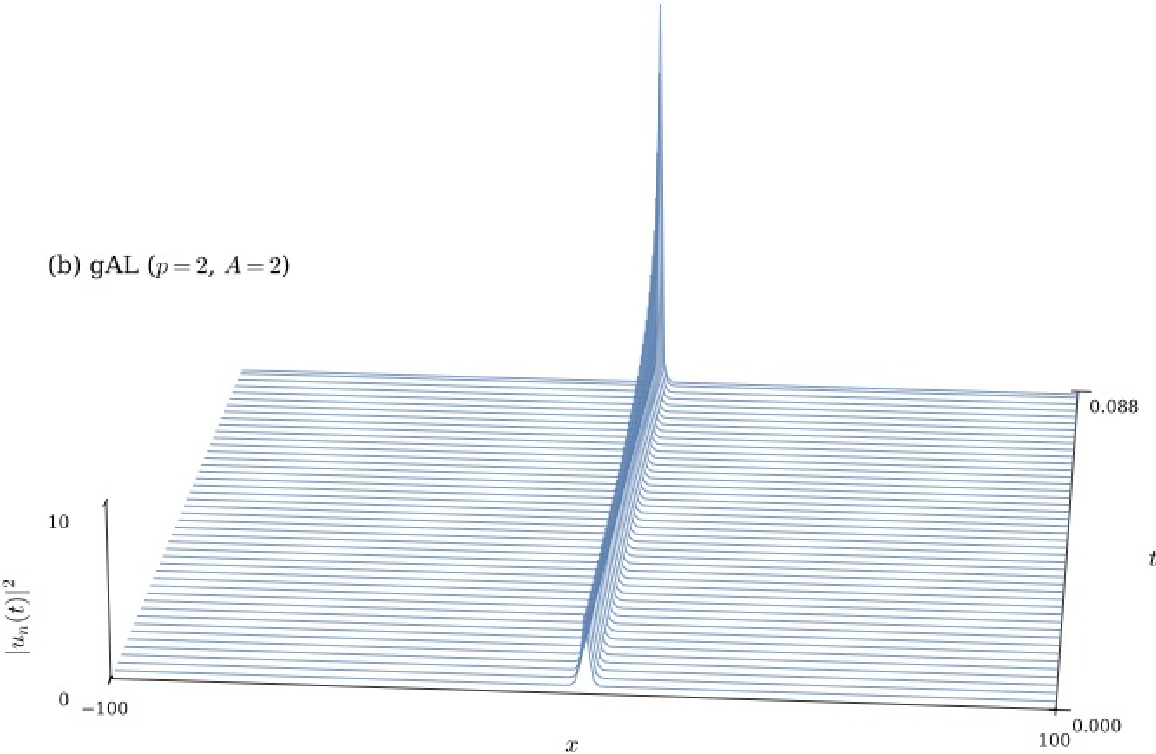}
 		\end{tabular}
 	\end{center}
 	\caption{Blow-up of solutions of gAL with $p=2$ with zero background. (a) $A=1.2$, (b) $A=2$}
 	\label{fig:bupzerobg}
 \end{figure}
Hence, as in the case of nonzero boundary conditions, the combined theoretical and numerical results of the present work lead to a positive answer to the  question raised in \cite{PGKgAL} regarding the possibility of a finite-time blow-up for solutions of the gAL equation on a zero background.
\\[2mm]
\textbf{The role of the nonlinearity exponent.}
Comparing the case $p=1$ illustrated in Figure \ref{fig1AL-DNLSdensity} for $q_0=0.1$ with the case $p=2$  shown in Figure \ref{fig4AL-DNLSdensity} for the same background amplitude, we identify the following crucial differences. 
In the former case, as noted earlier, both AL and DNLS reveal  non-trivial long-time asymptotics associated with the universal pattern of modulational instability. 
By contrast, in the latter case, the gAL solution undergoes delayed finite-time blow-up while the gDNLS solution exhibits delayed quasi-collapse. This comparison further underlines the significance of Theorem~\ref{al-dnls-close} particularly in settings where gAL and gDNLS  possess different nonlinearity exponents. Indeed, Theorem \ref{al-dnls-close} guarantees proximity between gAL with exponent $p_1$ and gDNLS with exponent $p_2$ for timescales of order
$$
T_c=\min\{\mathcal{O}(\ve^{-2p_1}), \mathcal{O}(\ve^{-2p_2})\}.
$$
Thus, when comparing AL with $p_1=1$ against gDNLS with $p_2=2$ (or vice versa), the expected timescale of proximity is $T_c = \mathcal{O}(\varepsilon^{-2})$, which is shorter than for the case $p_1=p_2=2$, where $T_c = \mathcal{O}(\varepsilon^{-4})$. This reduced proximity window for mismatched exponents becomes apparent when the AL dynamics with $p_1=1$ shown in the top panel of Figure \ref{fig1AL-DNLSdensity} is compared against the gDNLS dynamics with $p_2=2$ in the middle panel of Figure \ref{fig4AL-DNLSdensity} or, conversely, the gDNLS dynamics with $p_2=1$ in the middle panel of Figure \ref{fig1AL-DNLSdensity} against the gAL dynamics with $p_1=2$ in the top panel of Figure \ref{fig4AL-DNLSdensity}. Indeed, when $p=2$, the dynamics remains close to those for $p=1$ up to $t \approx 100$, 
which is approximately the time before the amplitude of the oscillations inside the modulational instability cone becomes significant in either gDNLS or gAL.

\begin{figure}[h!]
 	\begin{center}
 		\begin{tabular}{cc}
 		\includegraphics[width=0.4\textwidth]{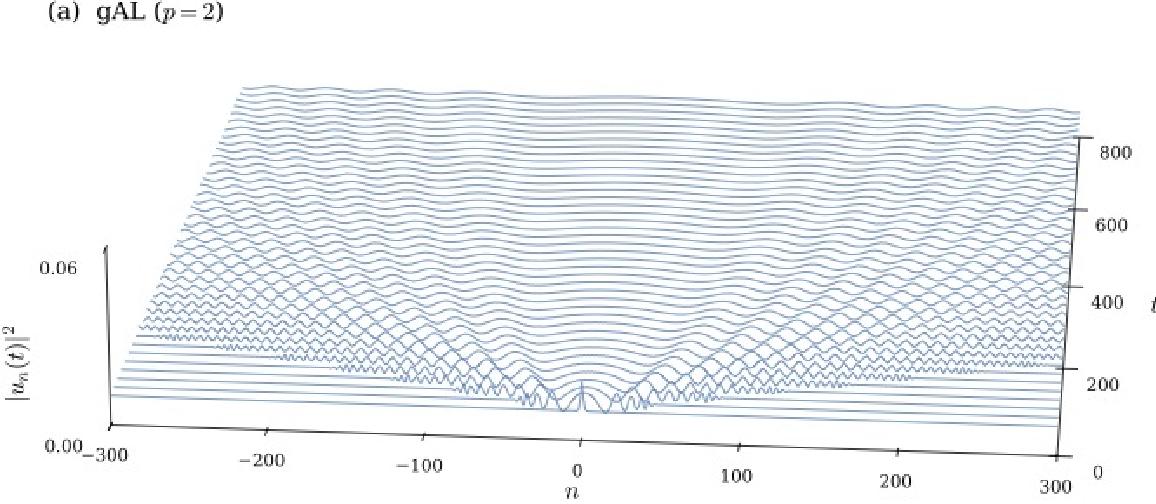} & 
            \includegraphics[width=0.4\textwidth]{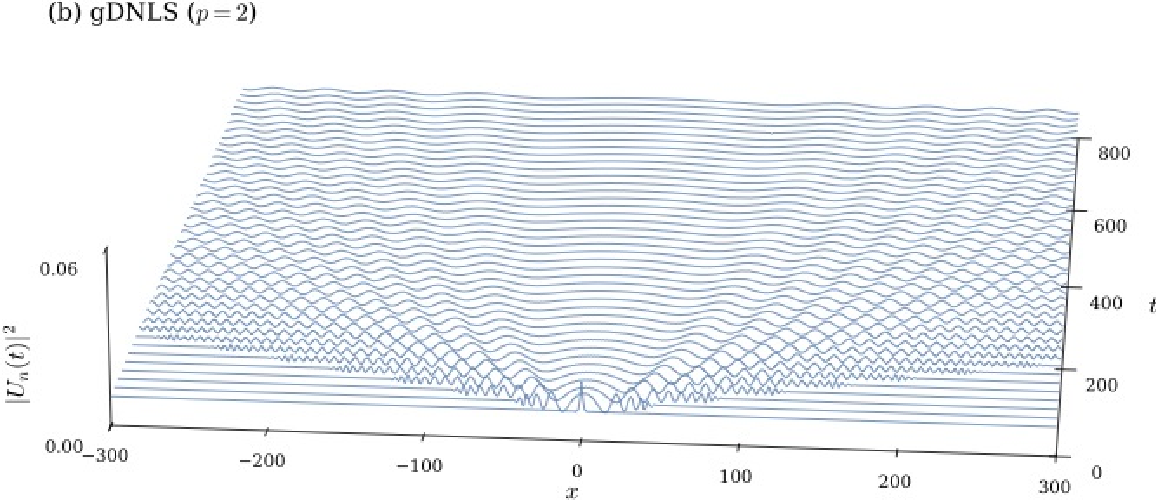}
        \end{tabular}
             \includegraphics[width=0.7\textwidth]{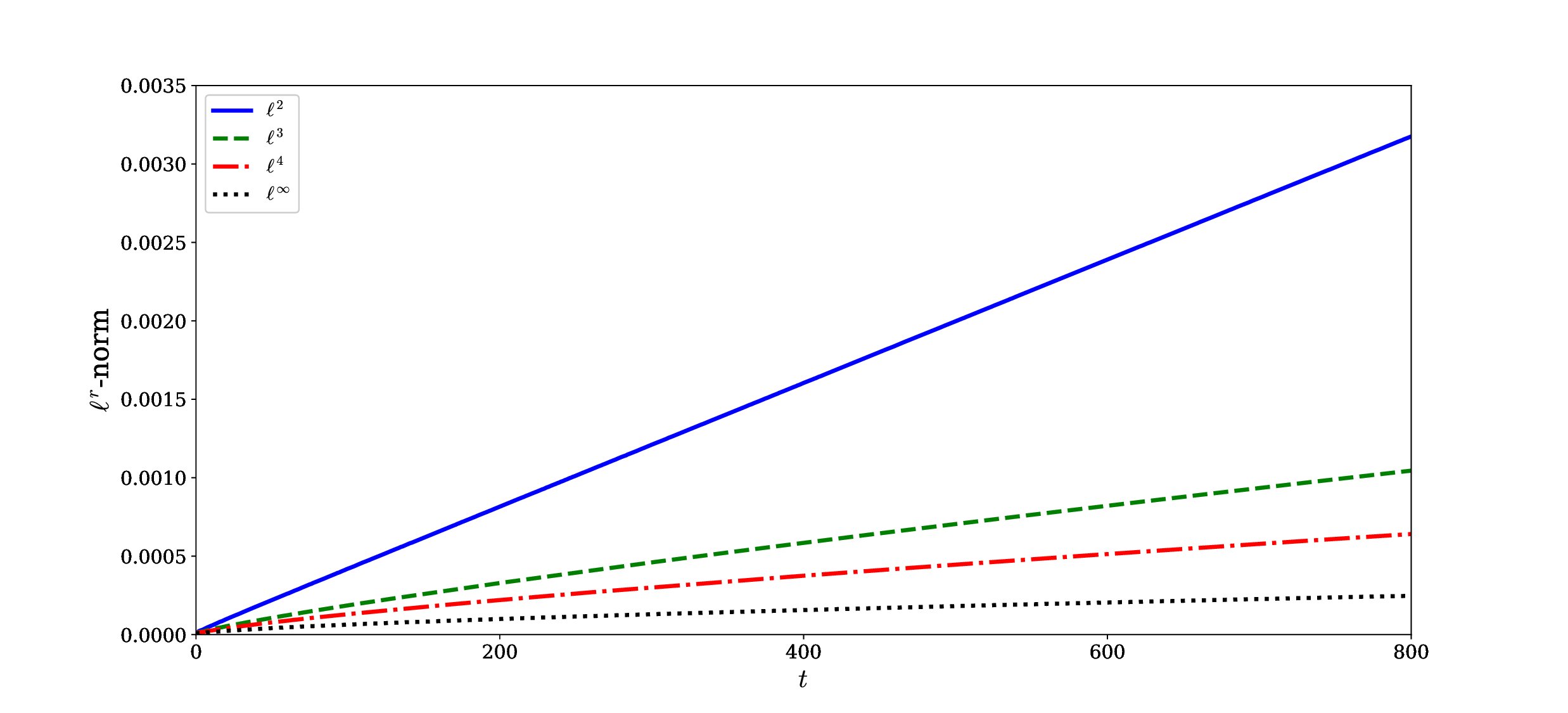}	
 	\end{center}
 	\caption{\textit{Top}: (a) Spatiotemporal evolution of the density $|u_n(t)|^2$ of the solution to the gAL lattice \eqref{eq:al2}  with $p=2$, $h=1$ and the initial condition \eqref{eq:num} with $q_0=0.1$ and (b) the same evolution for the gDNLS lattice \eqref{eq:dnls2} with $p=2$, $h=1$. \textit{Bottom}: Time evolution of the distance $\delta(t)$ given by \eqref{delta-def} in $\ell^r$ with  $r=2,3,4,\infty$. }
 	\label{fig4AL-DNLSdensity}
 \end{figure}
 
It is also worth noting that, as am alternative to Theorem~\ref{al-dnls-close} (which contrasts  gAL against gDNLS even in the case of different exponents and provides the above justification for the differences in the dynamics and their order of proximity observed numerically in Figures \ref{fig1AL-DNLSdensity} and \ref{fig4AL-DNLSdensity}),  one could instead consider for the same purpose the distance between solutions of two gAL systems with different exponents, or likewise between solutions of two gDNLS systems with different exponents. Such an investigation is of independent interest and is currently the subject of ongoing work.

\subsection{Non-constant background}
Finally, we investigate the case of a non-constant background $\zeta_n$ described by the general boundary conditions \eqref{uU-bc}-\eqref{zhi2}. For our numerical study, we consider the specific example 
\begin{equation}
\label{exzetan}
\zeta_n = A_1 e^{-\beta n^2} + q_0, \quad A_1, \beta > 0,
\end{equation}
which satisfies the conditions \eqref{zhi1}-\eqref{zhi2} and can be implemented in the numerical scheme via the initial condition 
\begin{equation}
\label{eq:nonconstinit}
u_n(0) = U_n(0) = \left( A_1 e^{-\beta n^2} + q_0 \right) \left( 1 + i \operatorname{sech}(n) \right),
\end{equation}
in analogy with the initial condition \eqref{eq:num} used in the case of a constant background. Equation \eqref{exzetan} represents the discrete analogue of the example (4.40)  in \cite{hkmms2024}. In all subsequent numerical investigations, we fix the parameters $\beta = 0.005$ and $q_0 = 0.1$ while varying the Gaussian amplitude $A_1$, and consider once again the cases $p=1$ and $p=2$ for the nonlinearity exponents.
\\[2mm]
\textit{The case $p=1$.}
Figure \ref{fig:nonconstbgb} illustrates the dynamics for $A_1 = 0.15$, which yields $\no{\zeta}_{\ell^{\infty}} = 0.25$, $\no{\zeta'}_{\ell^2} \approx 0.04$ and $\||\zeta| - q_0\|_{\ell^2} \approx 0.63$. We have $\varepsilon=\no{\phi(0)}_{\ell^2} = \no{\Phi(0)}_{\ell^2} \approx 0.35$ and, given that $q_0 = 0.1$, the conditions \eqref{dnls-ic}-\eqref{dnls-back} and \eqref{al-ic}-\eqref{al-back} are satisfied. In particular, note that for $p=1$ we have $\no{\zeta'}_{\ell^2} < \varepsilon^2=\varepsilon^{p+1}$, satisfying strongly the required assumption on $\no{\zeta'}_{\ell^2}$.  The system is integrated over the lattice $[-5000, 5000]$ for $t\in [0,300]$, and the plots are given over $n\in [-200, 200]$ for better visualization of the patterns.

The top panels (a) and (b) of Figure \ref{fig:nonconstbgb} illustrate the spatiotemporal evolution of the densities for the AL and DNLS numerical solutions, respectively, demonstrating the qualitative proximity of the corresponding patterns. The bottom panel depicts the evolution of the distance norms and highlights once more the excellent quantitative agreement with the theoretical predictions for the timescales of the proximal dynamics. Specifically, the distance $\no{\delta(t)}_{\ell^2}$ remains of $\mathcal{O}(\varepsilon^3)$ for times $t \approx \mathcal{O}(1)$ and of $\mathcal{O}(\varepsilon)$ for times $t \approx \mathcal{O}(\varepsilon^{-2})$, even exceeding theoretical expectations, as observed in the case of a constant background. An interesting difference with the case of the constant background is the emergence of a distinct breather-like waveform at the center, within the cone of oscillations. The maximum amplitude of these breathing modes is moderate when compared to that of the initial condition. For the AL system, we observe a maximum amplitude of $\max_{n,t}|u_n(t)|^2 \approx 0.30$, while for the DNLS system $\max_{n,t}|U_n(t)|^2 \approx 0.39$ within the interval of time integration.

The aforementioned breather-like waveforms exhibit an even more interesting pattern as $A_1$ increases. Figure~\ref{fig:nonconstbg} illustrates the dynamics for $A_1 = 0.3$. For this parameter value, we obtain $\no{\zeta}_{\ell^{\infty}} = 0.4$, $\no{\zeta'}_{\ell^2} \approx 0.08$ and $\||\zeta| - q_0\|_{\ell^2} \approx 1.26$. Here,  $\varepsilon = \no{\phi(0)}_{\ell^2} = \no{\Phi(0)}_{\ell^2} \approx 0.56$. Given that $q_0 = 0.1$, the conditions \eqref{dnls-ic}-\eqref{dnls-back} and \eqref{al-ic}-\eqref{al-back} are again satisfied. Notably, $\no{\zeta'}_{\ell^2} < \varepsilon^2 = \varepsilon^{p+1}$, which strongly satisfies the required assumption for $\no{\zeta'}_{\ell^2}$.

As expected from the explicit dependence of the estimate \eqref{d-l2} on the norms of the initial and background data, the increased value of $A_1 = 0.3$ (compared with $A_1 = 0.15$) results in a faster divergence rate. This remains in excellent agreement with theoretical predictions, as confirmed by the evolution of the norms in the bottom panel of Figure~\ref{fig:nonconstbg}. As shown in the top panels (a) and (b), which compare the spatiotemporal evolution of the densities, the waveforms remain proximal until the manifestation of the first two events (i.e. for $t \approx 70$), thus still exceeding the theoretical predictions. Notably, the emergent breather-like waveforms possess extreme amplitudes compared to the initial condition (exceeding the latter by a factor of three), which is a characteristic reminiscent of rogue waves. Indeed, for the AL system we observe a maximum amplitude of $\max_{n,t}|u_n(t)|^2 \approx 2.17$, while for the DNLS system we find $\max_{n,t}|U_n(t)|^2\approx 2.55$.

In the case of the AL lattice, the breathing waveform is reminiscent of the interplay between a discrete Peregrine soliton profile and a second-order rogue wave for the higher peaks, resembling a ``two-period'' Kuznetsov-Ma breather. A similar structure has been observed in the continuous case for the integrable cubic NLS equation, as seen in Figure 4.10 of  \cite{hkmms2024}. 
In the case of the DNLS lattice, beyond the time regime of proximal dynamics with the AL lattice, the rogue-wave-like oscillations exhibit a considerably smaller period. This difference explains the oscillatory behavior in the evolution of the norms: minima are observed when the highest peaks of the AL and DNLS coincide, while maxima arise due to the discrepancy between the highest peaks of the DNLS and the lower-amplitude Peregrine soliton peaks of the AL waveform.

A preliminary explanation for the emergence of these rogue-wave-reminiscent breathers is the localization of the background induced by $\zeta_n$. In contrast to the constant background case, this localization may enhance the effects of modulational instability. However, this phenomenon warrants further mathematical and numerical investigations that will be pursued elsewhere.

\begin{figure}[h!]
 	\begin{center}
 		\begin{tabular}{cc}
 		\includegraphics[width=0.45\textwidth]{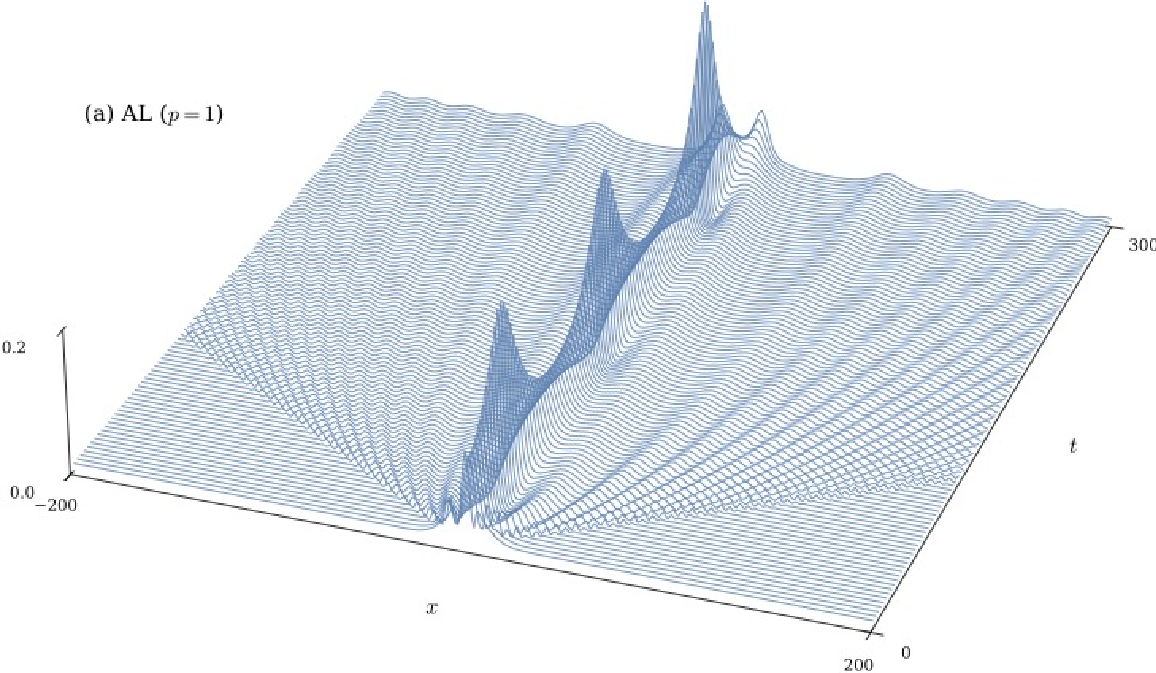}&
            \includegraphics[width=0.45\textwidth]{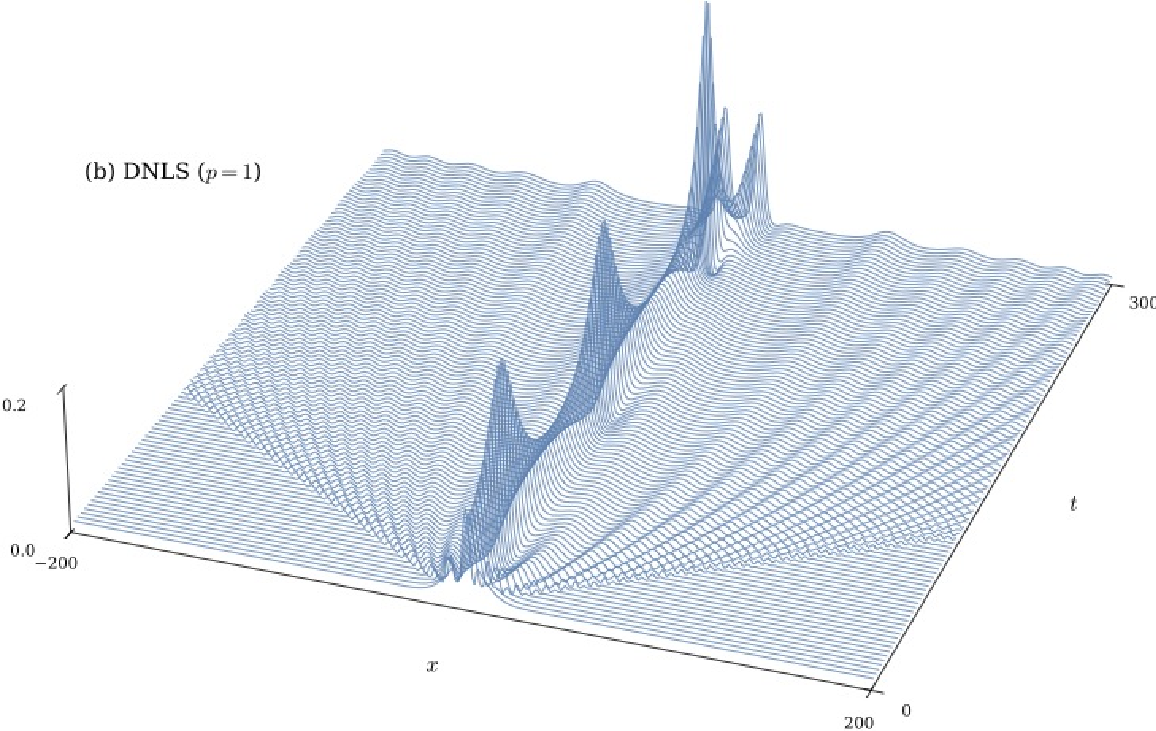}
 		\end{tabular}
             \includegraphics[width=0.7\textwidth]{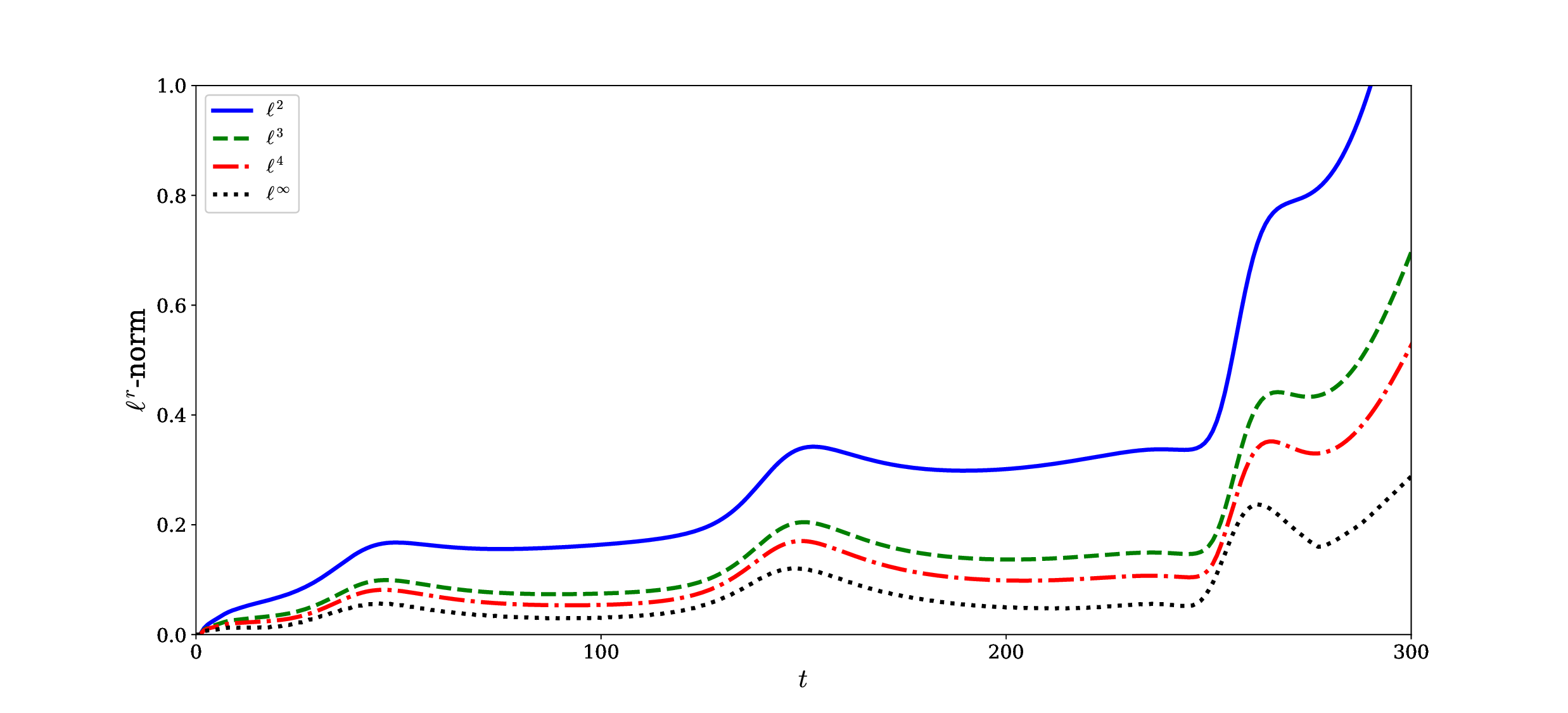}	
 	\end{center}
 	\caption{\textit{Top}: (a) Spatiotemporal evolution of the density $|u_n(t)|^2$ of the solution to the AL lattice \eqref{eq:al2}  with $p=1$, $h=1$ and the initial condition \eqref{eq:nonconstinit} with $q_0=0.1$, $\beta=0.005$, $A_1=0.15$. (b) The same evolution for the DNLS lattice \eqref{eq:dnls2} with $p=1$, $h=1$. \textit{Bottom}: Time evolution of the distance $\delta(t)$ given by \eqref{delta-def} in $\ell^r$ with  $r=2,3,4,\infty$. }
 	\label{fig:nonconstbgb}
 \end{figure}

\begin{figure}[h!]
 	\begin{center}
 		\begin{tabular}{cc}
 		\includegraphics[width=0.45\textwidth]{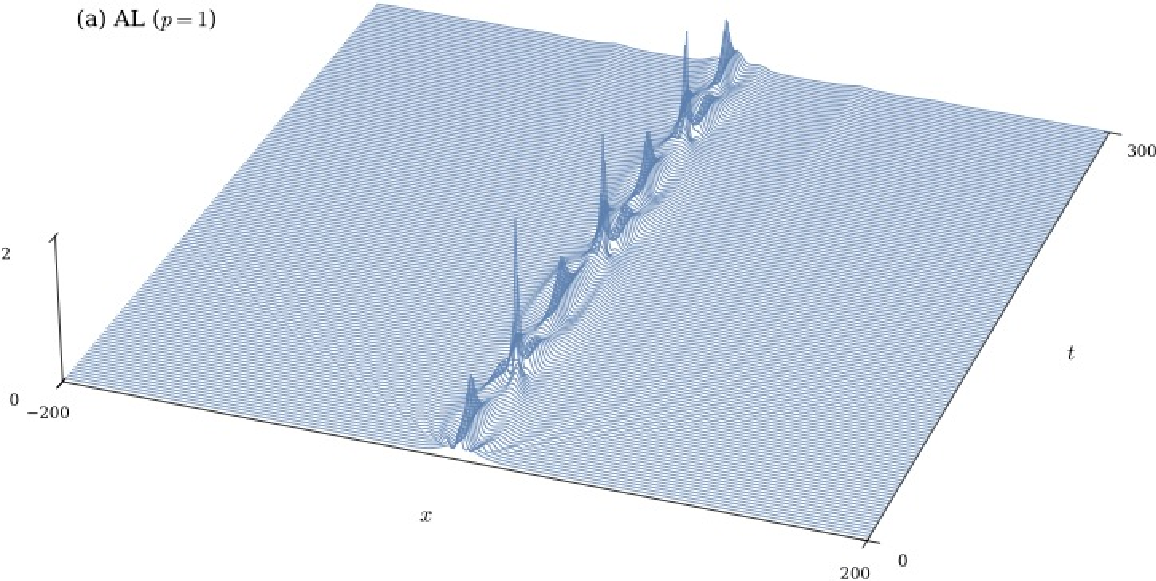}&
            \includegraphics[width=0.45\textwidth]{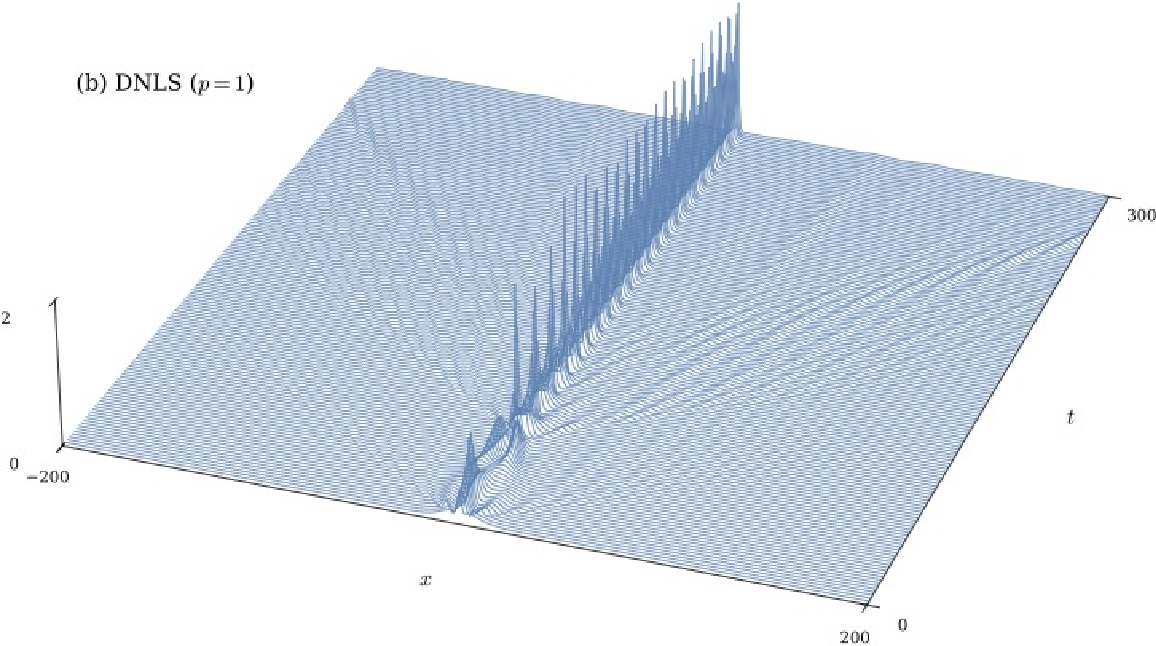}
 		\end{tabular}
             \includegraphics[width=0.7\textwidth]{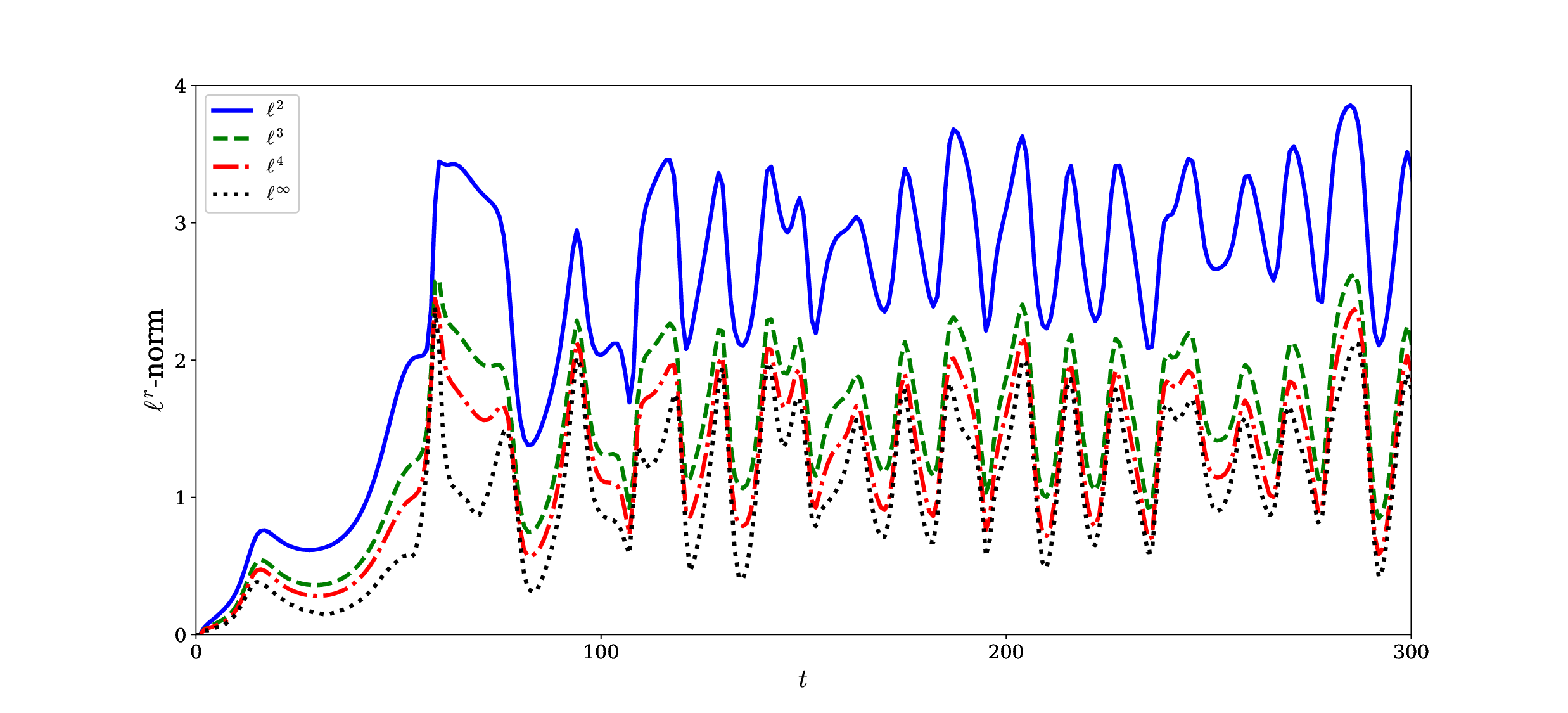}	
 	\end{center}
 	\caption{\textit{Top}: (a) Spatiotemporal evolution of the density $|u_n(t)|^2$ of the solution to the AL lattice \eqref{eq:al2}  with $p=1$, $h=1$ and the initial condition \eqref{eq:nonconstinit} with $q_0=0.1$, $\beta=0.005$, $A_1=0.3$. (b) The same evolution for the DNLS lattice \eqref{eq:dnls2} with $p=1$, $h=1$. \textit{Bottom}: Time evolution of the distance $\delta(t)$ given by \eqref{delta-def} in $\ell^r$ with  $r=2,3,4,\infty$. }
 	\label{fig:nonconstbg}
 \end{figure}
 \begin{figure}[h!]
 	\begin{center}
 		\begin{tabular}{cc}
 		\includegraphics[width=0.4\textwidth]{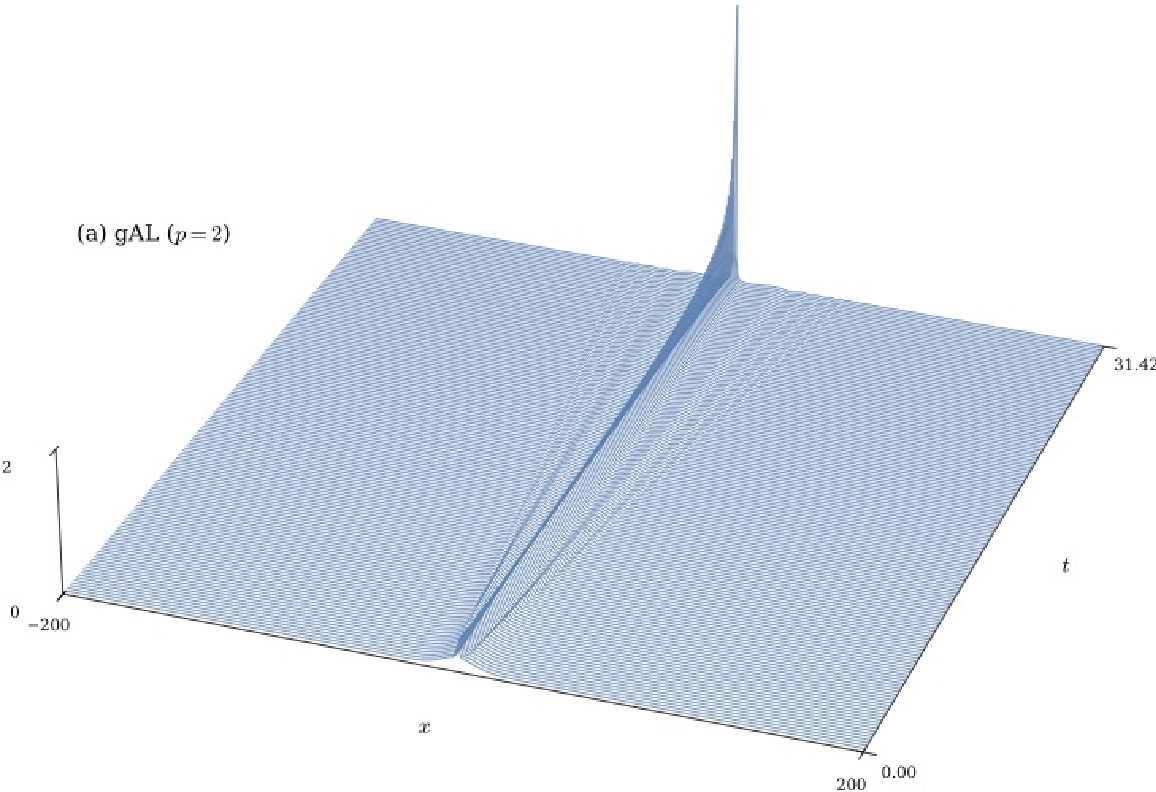}&
            \includegraphics[width=0.4\textwidth]{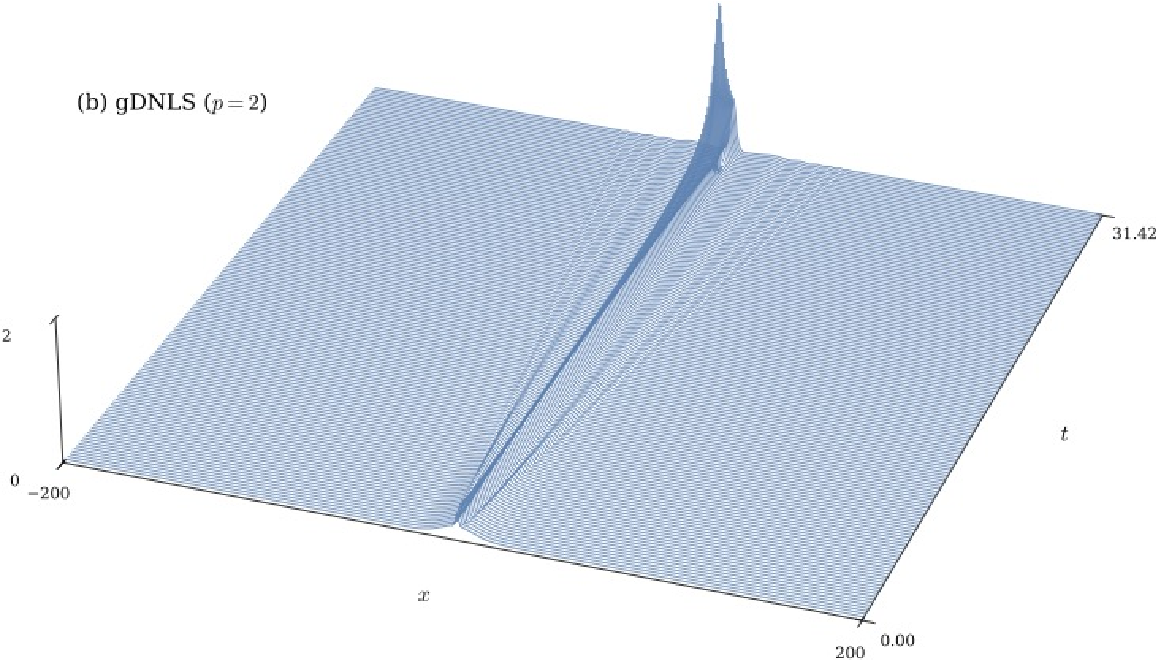}
 		\end{tabular}
             \includegraphics[width=0.7\textwidth]{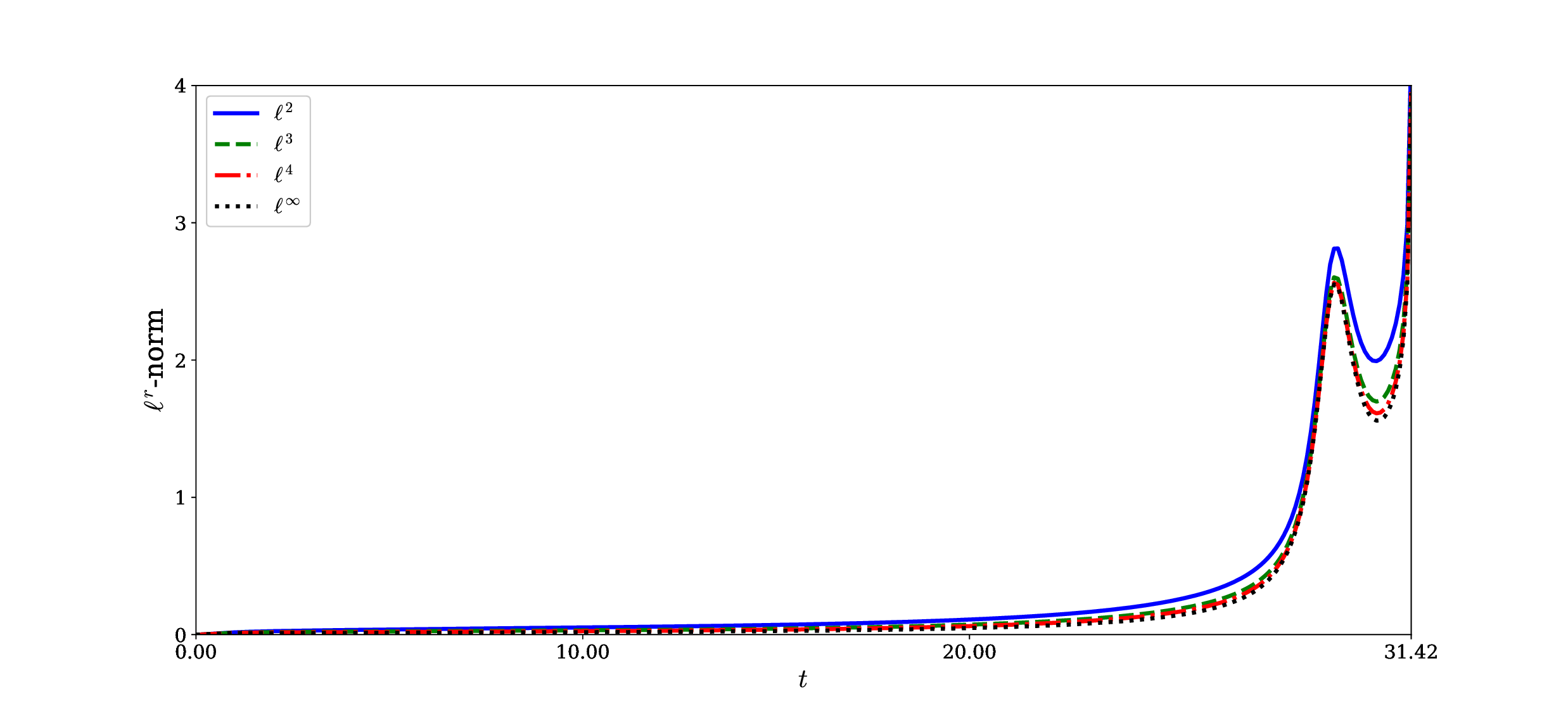}	
 	\end{center}
 	\caption{\textit{Top}: (a) Spatiotemporal evolution of the density $|u_n(t)|^2$ of the solution to the gAL lattice \eqref{eq:al2}  with $p=2$, $h=1$ and the initial condition \eqref{eq:nonconstinit} with $q_0=0.1$, $\beta=0.005$, $A_1=0.3$. (b) The same evolution for the gDNLS lattice \eqref{eq:dnls2} with $p=2$, $h=1$. \textit{Bottom}: Time evolution of the distance $\delta(t)$ given by \eqref{delta-def} in $\ell^r$ with  $r=2,3,4,\infty$. }
 	\label{fig:nonconstbg2}
 \end{figure}

\begin{figure}[h!]
 	\begin{center}
            \includegraphics[width=0.4\textwidth]{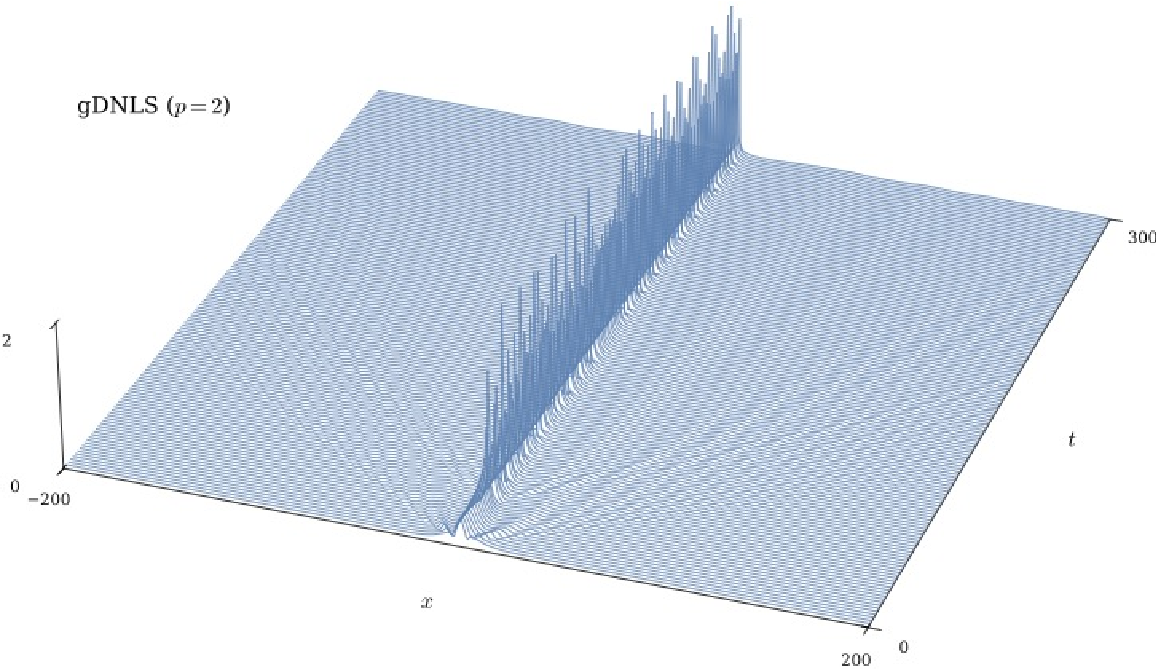}
 	\end{center}
 	\caption{Continuation of the spatiotemporal evolution of Figure \ref{fig:nonconstbg2}~(b) of the density $|U_n(t)|^2$ of the solution to the gDNLS lattice \eqref{eq:dnls2} with $p=2$, $h=1$ and the initial condition \eqref{eq:nonconstinit} with  $q_0=0.1$, $\beta=0.005$, $A_1=0.3$ up to $T=300$. }
 	\label{fig:nonconstbg3}
 \end{figure}
 
\noindent
\textit{The case $p=2$.} Next, we investigate the dynamics for the initial condition \eqref{exzetan} in the case $p=2$, with $A_1=0.3$ and all other parameters chosen as before. In this instance, as expected, the solution of the gAL equation exhibits finite-time blow-up. The top panels of Figure~\ref{fig:nonconstbg2} illustrate the spatiotemporal evolution of the densities of the solutions for gAL (left panel) and gDNLS (right panel), respectively, up to $T \approx 31.42$, which corresponds to the time at which the gAL solution collapses. The bottom panel depicts the evolution of the distance norms $\no{\delta(t)}_{\ell^r}$. 

The numerical results illustrate once again the proximal dynamics within the common interval of existence, as defined by the maximal existence time of the gAL solution. The observed timescales are in excellent agreement with Theorem~\ref{al-dnls-t}, consistent with our findings for the case of the constant background. Furthermore, as predicted by the global existence result of Theorem~\ref{globexfin}, the corresponding gDNLS solution remains bounded throughout the entire simulation interval $t \in [0, 300]$, as shown in Figure~\ref{fig:nonconstbg3}. This figure demonstrates the quasi-collapse phenomenon, where the solution attains a peak amplitude of $\max_{n,t}|U_n(t)|^2 \approx 2.50$.

\begin{remark}\label{BS}  
All simulations in this section considered nonzero boundary conditions of the form 
\begin{equation}\label{zhi22}
\lim_{n \to \pm \infty} \zeta_n = q_0> 0,	
\end{equation}
Although our theoretical analysis does cover the case of $\zeta_- = -\zeta_+$  in the boundary conditions \eqref{uU-bc}-\eqref{zhi2}, the numerical investigation of this scenario requires non-trivial modifications to the numerical methods used (e.g. the implementation of non-periodic boundary conditions), which extend beyond the scope of the present work and will be addressed elsewhere. For a recent theoretical and numerical study of the NLS equation with nonzero boundary conditions and different asymptotic phases of the background, we refer the reader to \cite{gpt2026}. 
\end{remark}

\subsection{Discrete versus continuous: gAL against gNLS on a  nonzero background}

We conclude our numerical investigation with a study of the solution lifespan for the gAL equation~\eqref{gal} with $\kappa=h^{-2} \gg 1$ against the lifespan of numerical solutions to the corresponding generalized NLS equation, computed using Besse's scheme \cite{b2004} with periodic boundary conditions. Specifically, we consider solutions arising from the following initial condition with nonzero background:
\begin{equation}\label{eq:inticond}
u(x,0) = 1 + i a f(x), \quad f(x) = \operatorname{sech}(x),
\end{equation}
where $a \in \mathbb{R}$ denotes the initial amplitude of the perturbation. As shown in \cite{hkmm2025}, this type of initial condition can lead to blow-up phenomena, as numerically assessed below using Besse's scheme~\cite{b2004}. To maintain consistency with the numerical studies on blow-up solutions for the NLS with nonzero boundary conditions presented in \cite{hkmm2025}, we adopt the parameter value $\mu=2$, which is another standard choice for the NLS equation. This selection corresponds to a unitary nonlinearity factor in the gAL lattice \eqref{gal}, allowing us to examine the potential blow-up behavior of the gAL system when doubling the nonlinearity strength (relative to the results in the previous section).
The numerical lifespans $T_b^{\text{Besse}}$ and $T_b^{\text{AL}}$ are estimated for various values of $p$ and $a$ using $h=10^{-2}$ and $\Delta t=10^{-4}$ in both schemes. The results are shown in Table \ref{tab:buptimes}. In these experiments, Besse's scheme does not exhibit numerical blow-up, but the solutions are contaminated by spurious oscillations. By contrast, the numerical scheme for the gAL equation eventually ceases to converge due to numerical overflow. The reported values of $T_b$ in Table \ref{tab:buptimes} correspond to the last timestep before the onset of spurious oscillations in Besse's scheme, or before numerical overflow occurs in the gAL equation.

\begin{table}[h!]
\begin{tabular}{c|cccc|cccc}
$a$ &  $p$ & $T_b^{\text{Besse}}$ & $T_b^{AL}$ & $T^\ast$ & $p$  & $T_b^{\text{Besse}}$ & $T_b^{AL}$ & $T^\ast$ \\
\hline
$1.8$ &  $2$ & $0.0553$ & $0.0550$ & $0.0134$ & $3$  & $0.0158$ & $0.0155$ & $0.0016$ \\ 
$2.0$ &  $2$ & $0.0440$ & $0.0439$ & $0.0087$ & $3$  & $0.0120$ & $0.0115$ & $0.0008$ \\
$2.2$ &  $2$ & $0.0360$ & $0.0357$ & $0.0060$ & $3$  & $0.0090$ & $0.0087$ & $0.0005$
\end{tabular}
\caption{Estimated numerical lifespan  $T_b$ of solutions with the initial condition~\eqref{eq:inticond} against the  estimate $T^*$ for the theoretical  minimum guaranteed lifespan of \cite{hkmm2025}.}
\label{tab:buptimes}
\end{table}

For reference, Table \ref{tab:buptimes} also includes an estimation of the theoretical lower bound $T^\ast$ for the lifespan of NLS solutions that was derived in \cite{hkmm2025}. We observe that the maximal interval of existence appears to be shorter for the gAL equation than for solutions computed with Besse's scheme under the same spatial resolution $h$. On the other hand, for the cases presented in the table, the gDNLS equation did not exhibit blow-up, similarly to Besse's scheme. It is also worth mentioning that when the spatial step was refined to $h=2\cdot 10^{-3}$, the corresponding values of $T_b$ increased slightly (by $\mathcal O(10^{-4})$), while following a pattern very similar to that shown in Table \ref{tab:buptimes}.

\bibliographystyle{myamsalpha}
\bibliography{references.bib}

\end{document}